\documentclass[10pt, leqno]{amsart}

\usepackage{graphicx}
\usepackage{amssymb,amsmath,amsthm,amscd}
\usepackage{mathrsfs}
\usepackage{upgreek}
\usepackage{enumerate}
\usepackage{lscape}
\usepackage{mathdots}
\usepackage[usenames,dvipsnames]{color}
\usepackage[colorlinks=true, pdfstartview=FitV,
 linkcolor=blue,citecolor=blue,urlcolor=blue]{hyperref}

\usepackage[all]{xy}
\allowdisplaybreaks[3]

\numberwithin{equation}{section}

\newenvironment{red}{\relax\color{red}}{\relax}
\newenvironment{blue}{\relax\color{blue}}{\hspace*{.5ex}\relax}

\newcommand{\ber}{\begin{red}}
\newcommand{\er}{\end{red}}
\newcommand{\beb}{\begin{blue}}
\newcommand{\eb}{\end{blue}}

\theoremstyle{plain}
\newtheorem{lemma}{Lemma}[section]
\newtheorem{proposition}[lemma]{Proposition}
\newtheorem{theorem}[lemma]{Theorem}

\newtheorem{corollary}[lemma]{Corollary}
\newtheorem{conjecture}{Conjecture}

\theoremstyle{definition}
\newtheorem{remark}[lemma]{Remark}
\newtheorem{example}[lemma]{Example}

\newtheorem{definition}[lemma]{Definition}

\newcommand{\hd}{{\operatorname{hd}}}
\newcommand{\soc}{{\operatorname{soc}}}

\newcommand{\g}{\mathfrak{g}}

\newcommand{\F}{\mathcal{F}}
\newcommand{\C}{\mathbb{C}}
\newcommand{\Z}{\mathbb{Z}}
\newcommand{\ta}{\mathtt{t}}

\newcommand{\A}{\mathbb{A}}
\newcommand{\N}{\mathsf{N}}

\newcommand{\Q}{\mathcal{Q}}

\newcommand{\cmA}{\mathsf{A}}
\newcommand{\W}{\mathsf{W}_0}
\newcommand{\seteq}{\mathbin{:=}}
\newcommand{\hooklongrightarrow}{\lhook\joinrel\longrightarrow}
\newcommand{\longtwoheadrightarrow}{\relbar\joinrel\twoheadrightarrow}
\newcommand{\al}{\alpha}

\newcommand{\be}{\beta}
\newcommand{\xiQ}{\xi^Q}
\newcommand{\mQ}{m^Q}
\newcommand{\ga}{\gamma}

\newcommand{\cqm}{\C[[q^{1/m}]]\;q^{1/m}}
\newcommand{\hconv}{\mathbin{\scalebox{.9}{$\nabla$}}}

\newcommand{\NR}{\Phi^-}
\newcommand{\WPR}{\widehat{\Phi}^+}
\newcommand{\PR}{\Phi^+}
\newcommand{\GQ}{\Gamma_Q}
\newcommand{\tw}{{\widetilde{w}}}
\newcommand{\len}{{\rm len}}
\newcommand{\rev}{{\rm rev}}

\newcommand{\redex}{{\widetilde{w}}}
\newcommand{\redez}{{\widetilde{w}_0}}
\newcommand{\um}{\underline{m}}

\newcommand{\us}{\underline{s}}
\newcommand{\up}{\underline{p}}

\newcommand{\tb}{\mathtt{b}}
\newcommand{\rl}{\mathsf{Q}}
\newcommand{\wl}{\mathsf{P}}
\newcommand{\cl}{{\rm cl}}
\newcommand{\wt}{{\rm wt}}
\newcommand{\dist}{{\rm dist}}
\newcommand{\gdist}{{\rm gdist}}
\newcommand{\mul}{{\rm mul}}
\newcommand{\het}{{\rm ht}}
\newcommand{\rds}{{\rm rds}}

\newcommand{\Rep}{{\rm Rep}}

\newcommand{\soplus}{\mathop{\mbox{\normalsize$\bigoplus$}}\limits}
\newcommand{\supp}{\operatorname{supp}}
\newcommand{\ve}{\varepsilon}
\newcommand{\lf}{\langle}
\newcommand{\rf}{\rangle}

\newcommand{\ko}{\mathbf{k}}
\newcommand{\conv}{\mathbin{\mbox{\large $\circ$}}}
\newcommand{\Lto}{\longrightarrow}
\newcommand{\Lgets}{\longleftarrow}
\newcommand{\tens}{\mathop\otimes}
\newcommand{\Rnorm}{R^{\rm{norm}}}
\newcommand{\Rren}{R^{\rm{ren}}}
\newcommand{\nmn}{ \{ n-1,n\} }
\newcommand{\kp}{ \kappa }
\newcommand{\FQup}{\mathbf{F}_Q^{{\rm up}}}

\newcommand{\Stom}{\overset{\to}{S}_{\redez}(\um)}
\newcommand{\Sgetsm}{\overset{\gets}{S}_{\redez}(\um)}
\newcommand{\sconv}{\mathbin{\mbox{$\Delta$}}}

\newcommand{\dct}[2]{\overset{#2}{\underset{#1}{\conv}} }
\newcommand{\prt}[2]{ \fontsize{5}{5}\selectfont \left( \begin{matrix} #1 \\ #2 \end{matrix} \right) \fontsize{10}{10}\selectfont }
\newcommand{\pprt}[4]{ {\scriptstyle \left( \begin{matrix} #1 \\ #2 \\ #3 \\ #4 \end{matrix}\right)} }

\newcommand{\To}[1][{\hspace{2ex}}]{\xrightarrow{\,#1\,}}
\newcommand{\cox}[1]{\tau_{\mspace{-2mu}\raisebox{-.5ex}{${\scriptstyle{#1}}$}}}
\newcommand{\AR}[1]{\Gamma_{\mspace{-2mu}\raisebox{-.5ex}{${\scriptstyle{#1}}$}}}

\newcommand{\rmat}[1]{{\mathbf r}_{\mspace{-2mu}\raisebox{-.5ex}{${\scriptstyle{#1}}$}}}

\newlength{\mylength}
\newcommand{\DpathOne}
{ \scalebox{0.82}{{\xy
(-20,0)*{}="DL";(-10,-10)*{}="DD";(0,20)*{}="DT";(10,10)*{}="DR";
"DT"+(-30,-4); "DT"+(55,-4)**\dir{.};
"DD"+(-20,-6); "DD"+(65,-6) **\dir{.};
"DD"+(-20,-10); "DD"+(65,-10) **\dir{.};
"DT"+(-34,-4)*{\scriptstyle 1};
"DT"+(-34,-8)*{\scriptstyle 2};
"DT"+(-34,-12)*{\scriptstyle \vdots};
"DT"+(-34,-16)*{\scriptstyle \vdots};
"DT"+(-36,-36)*{\scriptstyle n-1};
"DT"+(-34,-40)*{\scriptstyle n};
"DD"+(10,4); "DD"+(20,-6) **\dir{-};
"DD"+(10,4); "DD"+(0,-6) **\dir{-};
"DD"+(-10,4); "DD"+(0,-6) **\dir{-};
"DD"+(10,4); "DD"+(32,26) **\dir{.};
"DD"+(45,19);"DD"+(20,-6) **\dir{-};
"DD"+(45,19);"DD"+(38,26) **\dir{.};
"DD"+(45,19)*{\bullet};
"DD"+(10,4)*{\bullet};
"DD"+(10,6)*{\scriptstyle \al+\be};
"DD"+(27,10)*{(\GQ)};
"DD"+(20,-8)*{\scriptstyle k-\text{swing}};
"DD"+(55,19)*{\scriptstyle \al=\ve_k-\ve_{k+1}};
"DD"+(0,-6)*{\circ};
"DD"+(0,-10)*{\circ};
"DD"+(0,-8)*{\scriptstyle k+1-\text{swing}};
"DD"+(32,26); "DD"+(38,26) **\crv{"DD"+(35,28)};
"DD"+(35,29)*{\scriptstyle 2};
\endxy}}
}
\newcommand{\DpathTwo}
{\scalebox{0.82}{{\xy
(-20,0)*{}="DL";(-10,-10)*{}="DD";(0,20)*{}="DT";(10,10)*{}="DR";
"DT"+(-30,-4); "DT"+(55,-4)**\dir{.};
"DD"+(-20,-6); "DD"+(65,-6) **\dir{.};
"DD"+(-20,-10); "DD"+(65,-10) **\dir{.};
"DD"+(30,4); "DD"+(20,-6) **\dir{-};
"DD"+(30,4); "DD"+(40,-6) **\dir{-};
"DD"+(50,4); "DD"+(40,-6) **\dir{-};
"DD"+(30,4); "DD"+(8,26) **\dir{.};
"DD"+(-5,19);"DD"+(20,-6) **\dir{-};
"DD"+(-5,19);"DD"+(2,26) **\dir{.};
"DD"+(-5,19)*{\bullet};
"DD"+(30,4)*{\bullet};
"DD"+(27,10)*{(\Gamma_{s_k(Q)})};
"DD"+(30,6)*{\scriptstyle \al+\be'};
"DD"+(20,-8)*{\scriptstyle k-\text{swing}};
"DD"+(-15,19)*{\scriptstyle \al=\ve_k-\ve_{k+1}};
"DD"+(40,-6)*{\odot};
"DD"+(40,-10)*{\odot};
"DD"+(40,-8)*{\scriptstyle k+1-\text{swing}};
"DD"+(8,26); "DD"+(2,26) **\crv{"DD"+(5,28)};
"DD"+(5,29)*{\scriptstyle 2};
\endxy}}
}

\newcommand{\GammaA}
{\scalebox{0.85}{\xymatrix@R=0.5ex{
(i,p) & -6 & -5 & -4 & -3 & -2 & -1 & 0 \\
1&[5] \ar@{->}[dr] && [4]\ar@{->}[dr] && [2,3]\ar@{->}[dr] && [1] \\
2&& [4,5]\ar@{->}[dr]\ar@{->}[ur] && [2,4]\ar@{->}[dr]\ar@{->}[ur] && [1,3]\ar@{->}[dr]\ar@{->}[ur] \\
3&&& [2,5]\ar@{->}[dr]\ar@{->}[ur] && [1,4]\ar@{->}[dr]\ar@{->}[ur] && [3] \\
4&& [2] \ar@{->}[dr]\ar@{->}[ur], && [1,5]\ar@{->}[dr]\ar@{->}[ur] && [3,4]\ar@{->}[ur] \\
5&&& [1,2]\ar@{->}[ur] && [3,5]\ar@{->}[ur]
}}
}

\newcommand{\GammaD}
{\scalebox{0.85}{\xymatrix@R=0.5ex{
1&\lf 1 , -2 \rf \ar@{->}[dr] && \lf 2 , 4 \rf\ar@{->}[dr] && \lf 1 , -4 \rf \ar@{->}[dr]  \\
2&& \lf 1 , 4 \rf \ar@{->}[dr]\ar@{->}[ddr]\ar@{->}[ur] && \lf 1 , 2 \rf\ar@{->}[ddr]\ar@{->}[dr]\ar@{->}[ur] && \lf 2 , -4 \rf \ar@{->}[dr] \\
3&&& \lf 1 , 3\rf \ar@{->}[ur] && \lf 2 , -3 \rf \ar@{->}[ur] && \lf 3 , -4 \rf \\
4& \lf 3 , 4 \rf \ar@{->}[uur] && \lf 1 , -3 \rf \ar@{->}[uur] && \lf 2 , 3 \rf \ar@{->}[uur]
}}
}
\newcommand{\swing}
{\raisebox{2em}{\xymatrix@C=2.5ex@R=0.5ex{ &&&&\be\ar@{->}[dr] \\S_r \ar@{->}[r] & S_{r+1} \ar@{->}[r] & \cdots \ar@{->}[r]& S_{n-2}
\ar@{->}[ur]\ar@{->}[dr] && N_{n-2} \ar@{->}[r]& N_{n-3} \ar@{->}[r] & \cdots \ar@{->}[r] & N_s\\
&&&&\al\ar@{->}[ur] }}
}

\newcommand{\thetaswing}
{ \raisebox{1.8em}{\scalebox{0.9}{\xymatrix@C=3ex@R=0.5ex{ &&&&\ve_k \pm \ve_{\ta}\ar@{->}[dr]
\\ \theta^Q_k=S_{k} \ar@{->}[r] & S_{k+1} \ar@{->}[r] & \cdots \ar@{->}[r]& S_{n-2} \ar@{->}[ur]\ar@{->}[dr] &&
N_{n-2} \ar@{->}[r]& N_{n-3} \ar@{->}[r] & \cdots \ar@{->}[r] & N_{1}\\
&&&&\ve_k \mp \ve_{\ta}\ar@{->}[ur] }}}
}
\newcommand{\gammaswing}
{ \raisebox{1.8em}{\scalebox{0.9}{
\xymatrix@C=3ex@R=0.5ex{ &&&&\ve_k \pm \ve_{\ta}\ar@{->}[dr] \\S_{1} \ar@{->}[r] & S_{2} \ar@{->}[r] & \cdots \ar@{->}[r]& S_{n-2} \ar@{->}[ur]\ar@{->}[dr] &&
N_{n-2} \ar@{->}[r]& N_{n-3} \ar@{->}[r] & \cdots \ar@{->}[r] & N_{k}=\gamma^Q_k\\
&&&&\ve_k \mp \ve_{\ta}\ar@{->}[ur] }}}
}

\newcommand{\leftrightinterA}
{\xymatrix@R=3ex{ *{ \ }<3pt> \ar@{..}[r]  &*{\circ}<3pt>
\ar@{<-}[r]_<{i-1}  &*{\circ}<3pt>
\ar@{<-}[r]_<{i} &*{\circ}<3pt>
\ar@{..}[r]_<{i+1}  & *{ \ }<3pt>
}\quad
\text{ (resp. }
\xymatrix@R=3ex{ *{ \ }<3pt> \ar@{..}[r]  &*{\circ}<3pt>
\ar@{->}[r]_<{i-1}  &*{\circ}<3pt>
\ar@{->}[r]_<{i} &*{\circ}<3pt>
\ar@{..}[r]_<{i+1}  & *{ \ }<3pt>
} ) \quad
\text{in $Q$}
}

\usepackage[colorinlistoftodos]{todonotes}

\definecolor{darkred}{rgb}{0.7,0,0} % darkred color
\newcommand{\defn}[1]{{\color{darkred}\emph{#1}}} % emphasis of a definition

\title[AR-quiver and KLR-type duality] {Auslander-Reiten quiver and representation theories related to KLR-type Schur-Weyl duality}

\author[Se-jin Oh]{Se-jin Oh}
\address{Department of Mathematics Ewha Womans University Seoul 120-750, Korea}
\email{sejin092@gmail.com}

\thanks{This work was supported by NRF Grant \#2016R1C1B1010721.}


\date{\today}

\begin{document}

\begin{abstract}
We introduce new partial orders on the sequence positive roots and study the statistics of the poset
by using Auslander-Reiten quivers for finite type ADE. Then we can prove that the statistics provide
interesting information on the representation theories of KLR-algebras, quantum groups and
quantum affine algebras including Dorey's rule, bases theory for quantum groups, and
denominator formulas between fundamental representations. As applications, we prove Dorey's rule for quantum affine algebras $U_q(E_{6,7,8}^{(1)})$ and
partial information of denominator formulas for $U_q(E_{6,7,8}^{(1)})$. We also suggest conjecture on complete denominator formulas for $U_q(E_{6,7,8}^{(1)})$.
\end{abstract}

\maketitle

\tableofcontents

\section*{Introduction}
%\sejin{Shorten and emphasize Dorey's rule for $E$.Do not care about the number of page anymore. Consistency and easy to see are most important.}
%\referee{You have to compare your result with~\cite{FH15,YZ11}.}
%\sejin{Use the macro defn, emphasize statistics}
The category $\Rep(R)$ consisting of finite dimensional graded modules over KLR algebra $R$
provides the categorification of negative part of $U_q(\mathsf{g})$
for all symmetrizable Kac-Moody algebras $\mathsf{g}$ (\cite{KL09,R08}). When $U_q(\g_0)$ is associated with a finite simple Lie algebra $\g_0$ of type $A$,$D$ or $E$,
$U^-_{q}(\g_0)$ is also categorified by the categories $\mathcal{C}_Q^{(1)}$ consisting of finite dimensional integrable modules over
the quantum affine algebras $U_q'(\g)$ ($\g=A_n^{(1)}$, $D_n^{(1)}$ or $E_{6,7,8}^{(1)}$) (\cite{HL11}).
Here the definition of $\mathcal{C}_Q^{(1)}$ is closely related to the
Auslander-Reiten quiver (AR-quiver) $\Gamma_Q$ of a Dynkin quiver $Q$ for $\g_0$.

The two categories, $\Rep(R)$ and $\mathcal{C}_Q^{(1)}$ for $\g=A_n^{(1)}$ or $D_n^{(1)}$,
are closely related to each other by the KLR-type Schur-Weyl functor $\mathcal{F}^{(1)}_Q$ (\cite{KKK13A,KKK13B}) where
\begin{itemize}
\item[{\rm (i)}] $\mathcal{F}^{(1)}_Q$ is an exact tensor functor and sends simples to simples,
\item[{\rm (ii)}] $\mathcal{F}^{(1)}_Q$ is constructed by observations on the properties of AR-quiver $\Gamma_Q$ and
{\it denominator formulas} $d_{k,l}(z)$ for fundamental representations over $U_q'(\g)$.
\end{itemize}
The statement {\rm (ii)} implies that the partial information about $d_{k,l}(z)$ can be read from $\Gamma_Q$
(see~\cite[Introduction]{Oh14D} for more detail).

In $U^-_{q}(\g_0)$, there are distinguished bases, so called (dual) PBW-bases, which are associated to reduced
expressions $\redez$ of the longest element $w_0$ of its Weyl group $\W$ (\cite{Lus93}). Also, in $U^-_{\A}(\g_0)$,
there exists a unique basis, the Kashiwara/Lusztig's lower global/canonical (resp. upper global/dual canonical) basis (\cite{Lus93,Kash93}).

Interestingly, those distinguished bases for $U^-_{q}(\g_0)$ are categorified by modules over KLR-algebras $R$ (\cite{BKM12,Kato12,Mc12} and
\cite{R11,VV09}) under the suitable assumptions. More precisely, (dual) PBW-bases
are categorified by the set of all (proper) standard modules over $R$ and the (dual) canonical basis is categorified by the set of all
principal indecomposable (simple) modules over $R$. Furthermore, the transition map between the (dual) PBW-bases
associated to $\redez$ and (dual) canonical basis can be described by the composition series of the (proper) standard modules.
For the description of the transition map, the bi-lexicographical order $<^\tb_\redez$ on $\Z^{\ell(w_0)}_{\ge 0}$, induced by the reduced expression $\redez$ and
its convex total order $<_\redez$ on the positive roots $\PR$ for $\g_0$, plays important roles.

On the other hand, the dual PBW-basis associated to $\redez$ {\it adapted to} a Dynkin quiver $Q$ (see~\eqref{eq: adapted})
is also categorified by ordered tensor products of fundamental representations in $\mathcal{C}_Q^{(1)}$ with respect to the convex total order $<_\redez$ on $\PR$.
Also, the dual canonical basis is also categorified by the set of simple modules in $\mathcal{C}_Q^{(1)}$ (\cite{HL11}).

Interestingly, the AR-quiver $\Gamma_Q$ visualizes the convex partial order $\prec_{Q}$ on $\PR$
which is induced by convex total orders $<_\redez$ for all $\redez \in [Q]$. Here $[Q]$ denotes the set all reduced expressions $\redez$ of $w_0$ adapted to the Dynkin quiver $Q$
(see~\cite{B99}). Also it is known that each $\Gamma_Q$ has a unique Coxeter element $\tau_Q$ and $\Gamma_Q$ is determined by $\tau_Q$ (see~\cite{HL11}).
%\medskip

In the representation theory of quantum affine algebra $U_q'(\g)$, Coxeter elements and their twisted analogues
play an important role for {\it Dorey's rule}, which are closely related to the three-point coupling between the quantum particles
in Toda field theory~\cite{CP96,DO94}. More precisely, for untwisted quantum affine algebras of classical types $\g=A_n^{(1)}$, $B_n^{(1)}$, $C_n^{(1)}$ and $D_n^{(1)}$,
Chari-Pressley used the Coxeter elements and their twisted analogues to prove that
the tensor product of two fundamental representations $V^{(1)}(\varpi_i)_x \tens V^{(1)}(\varpi_j)_y$ over $U_q'(\g)$ satisfying {\it certain condition} has the simple socle as the another fundamental representation $V^{(1)}(\varpi_k)_z $:
\begin{align} \label{eq: the condition}
V^{(1)}(\varpi_k)_z \rightarrowtail V^{(1)}(\varpi_i)_x \tens V^{(1)}(\varpi_j)_y \ \ \text{ if $\{ (i,x),(j,y),(k,z) \}$ satisfies {\it certain conditions}.}
\end{align}

To sum up, AR-quivers $\Gamma_Q$ and hence their Coxeter elements $\tau_Q$ play central roles in the representation theories of
quantum affine algebras, quantum group, system of positive roots.
The goal of this paper is to understand those
representation theories including KLR-algebras one step further by developing new {\it statistics} on the sequences of positive roots and by using
AR-quivers and commutation classes for reduced expressions $\tw$ of $w \in \W$. As applications of these results, we can prove the Dorey's rule (Theorem~\ref{thm: Dorey's step2}) and obtain partial information of denominator formulas for quantum affine algebras of type $E^{(1)}_{6,7,8}$, which were studied in \cite{FH15,YZ11} and not known completely to the best knowledge of the author.

The most important new notion in this paper is a partial order
$$\text{ $\prec_{[\redez]}^\tb$ on $\Z_{\ge 0}^{\ell(w_0)}$ which is {\it far coarser} than $<^\tb_{\redez}$ (see Definition~\ref{def: redezprectb}).}$$

Using this new order, we prove (Theorem~\ref{thm: [redez] works}) that the transition map between a dual PBW-basis associated to any $\redez$ of {\it any finite type} and
the Kashiwara/Lusztig's  upper global/dual canonical can be refined by far. Also, we can prove that each dual PBW-basis of $U^-_q(\g_0)$ {\it does depend only} on the commutation class
$[\redez]$ of $\redez$ (up to $q^\Z$) indeed (see~\cite{Kato12} for $ADE$ cases).

With new orders $\prec_{[\redez]}^\tb$, we define new statistics on the sequences and the pairs (sequences consisting two distinct positive roots) of positive roots,
which arise from the poset $(\Z_{\ge 0}^{\ell(w_0)},\prec_{[\redez]}^\tb)$ and depend only on the commutation class $[\redez]$.
The new statistics on pairs are motivated by the work of Hernandez in~\cite{H01S} which was concentrated on the tensor product of two simple modules
over quantum affine algebras.

Using the statistics, we can obtain interesting results on the representation theories of
KLR-algebras, quantum affine algebras, quantum group, system of positive roots by concentrating on the special commutation classes $[Q]$ and corresponding AR-quivers $\Gamma_Q$.

Among the results in this paper, we focus on two results in this introduction.

(A) By the works in~\cite{Oh14A,Oh14D}, the Dorey's rules in~\eqref{eq: the condition}
for $U_q'(A^{(1)}_n)$ and $U_q'(D^{(1)}_n)$ were extended and interpreted in their corresponding category $\Rep(R)$ as follows:
There exist some Dynkin quiver $Q$ and $\al,\be,\ga=\al+\be \in \PR$  such that (see Definition~\ref{thm: BkMc})
$$V^{(1)}(\varpi_k)_{(-q)^c} \rightarrowtail V^{(1)}(\varpi_i)_{(-q)^a} \tens V^{(1)}(\varpi_j)_{(-q)^b} \text{ if and only if }
S_Q(\ga) \rightarrowtail S_Q(\al) \conv S_Q(\be),$$
where (i) the coordinates of $\al$, $\be$ and $\ga$ in $\Gamma_Q$ are $(i,a)$, $(j,b)$, $(l,c)$, respectively, (ii) $S_Q(\al)$ for $\al \in \PR$ is
a proper standard $R$-module corresponding to the dual root vector $\mathbf{F}_Q^{{\rm up}}(\al)$ of the dual PBW basis associated to $[Q]$.

By removing the condition that $\al+\be \in \PR$, we prove in this paper that the socle of $S_Q(\al) \conv S_Q(\be)$ (resp. $V^{(t)}_Q(\al) \tens V^{(t)}_Q(\be)$) is isomorphic to the simple module
$S_Q(\soc_Q(\al,\be))$ $\left(\text{resp}. \ V^{(t)}_Q(\soc_Q(\al,\be))\right)$
for any $\al,\be \in \PR$ and any $Q$ of types $A_n$, $D_n$ and $E_{6,7,8}$, which can be considered as
the generalization of Dorey's rule in both categories:
$$S_Q(\soc_Q(\al,\be)) \rightarrowtail S_Q(\al) \conv S_Q(\be) \quad \text{ and } \quad
V^{(t)}_Q(\soc_Q(\al,\be)) \rightarrowtail V^{(t)}_Q(\al) \tens V^{(t)}_Q(\be).$$
Here,
{\rm (i)} $\soc_Q(\al,\be)$ denotes the sequence of positive roots, called the $[Q]$-socle of a pair $(\al,\be)$,

\noindent
{\rm (ii)} $V^{(t)}_Q(\eta)$ $(\eta \in \PR)$ is a fundamental modules over $U_q'(A_n^{(t)})$, $U_q'(D_n^{(t)})$ ($t=1,2$) and $U_q'(E_{6,7,8}^{(1)})$.

\noindent
We prove also that $S_Q(\soc_Q(\al,\be))$ $\left(\text{resp}. \ V^{(t)}_Q(\soc_Q(\al,\be))\right)$ is a convolution product (resp. tensor product) of simple $R$-modules (resp. $U_q'(\g)$-modules) corresponding to dual root vectors (Theorem~\ref{thm: socle of Q-pair} and Theorem~\ref{thm: for C_Q}).
Finding the socle and the head of tensor product of two simple modules has been intensively studied (for instance, see~\cite{KKKO14S,L03}), as the linear combination
of the product of two dual canonical basis elements in terms of dual canonical basis elements. Interestingly, the result on socles provides the characterizations
{\rm (i)} for elements in $U^-_q(\g_0)$ which are contained in the dual PBW-basis and the dual canonical basis simultaneously (Corollary~\ref{cor: simple iff simple}),
{\rm (ii)} for a pair of dual root vectors whose product is a linear combination of two distinct canonical basis elements, known as length two property (Corollary~\ref{cor: l2}),
when the PBW-basis is associated to $[Q]$ for some Dynkin quiver $Q$.

(B) The denominator formulas $d_{k,l}(z)$ between fundamental representations over $U_q'(\g)$ provides crucial information about the representation theory
for integrable modules over $U_q'(\g)$ (see, Theorem~\ref{Thm: basic properties}).
As we mentioned above, one can observe partial information about $d_{k,l}(z)$ from
the combinatorics of $\Gamma_Q$. In this paper, we use the new notions to get complete information  $d_{k,l}(z)$
from $\Gamma_Q$. In other word, we can read the denominator formulas $d_{k,l}(z)$ in $\Gamma_Q$ {\it completely} for $\g=A^{(1)}_{n}$ and $D^{(1)}_{n}$, which were calculated in~\cite{{DO94,KKK13B}}, by
defining the distance polynomials $D_{k,l}(z,-q)$ on $\Gamma_Q$, which does not depend on the choice of $Q$ of type $A_n$, $D_n$ and $E_{6,7,8}$.
From the distance polynomials $D_{k,l}(z,-q)$ on $\Gamma_Q$ of type $E_{6,7,8}$, we can obtain
partial information of denominator formulas $d_{k,l}(z)$ for $U_q'(E_{6,7,8}^{(1)})$ (Corollary~\ref{cor: dist denom}). Moreover,
we can expect naturally that the complete formulas $d_{k,l}(z)$ and for $U_q'(E^{(1)}_{6,7,8})$
can be read from any $\Gamma_Q$ of type $E_{6,7,8}$ as in the cases of types $A_n$ and $D_n$ (see Conjecture~\ref{conj: dist E}). For examples, we give the conjectural denominators for $U_q'(E_6^{(1)})$ and $U_q'(E_7^{(1)})$
(see Subsection~\ref{subsec: E conjectures}).

\medskip

\noindent
{\bf Acknowledgements.} The author would like to express his sincere gratitude to Professor
Masaki Kashiwara, Myungho Kim and Chul-hee Lee for many fruitful discussions.

\section*{Notions and Conventions}

In this preliminary section, we fix the notions and the conventions.

\subsection*{Symmetrizable Cartan datum and quantum group}

Let $I$ be an index set. A \defn{symmetrizable Cartan datum} is a quintuple $(\cmA,\wl,\Pi,\wl^{\vee},\Pi^{\vee})$ consisting of
{\rm (a)} a \defn{symmetrizable generalized Cartan matrix} $\cmA=(a_{ij})_{i,j \in I}$,
{\rm (b)} a free abelian group $\wl$, called the \defn{weight lattice},
{\rm (c)} $\Pi= \{ \alpha_i \in \wl \mid \ i \in I \}$, called
the set of \defn{simple roots},
{\rm (d)} $\wl^{\vee}\seteq {\rm Hom}(\wl, \Z)$, called the \defn{coweight lattice},
{\rm (e)} $\Pi^{\vee}= \{ h_i \ | \ i \in I \}\subset P^{\vee}$, called
the set of \defn{simple coroots}.

The free abelian group $\rl\seteq \bigoplus_{i \in I} \Z \alpha_i$ is also called the
\defn{root lattice}. Set $\rl^{+}= \sum_{i \in I} \Z_{\ge 0}
\alpha_i$. For $\mathsf{b}=\sum_{i\in I}m_i\al_i\in\rl^+$,
we set $\het(\mathsf{b})=\sum_{i\in I}m_i$.

There exists a positive-definite symmetric bilinear form $\ \cdot \ : \rl \times \rl \to \Z$ satisfying
$$\cmA=(a_{ij})_{i,j \in I} = \left( \dfrac{2 \al_i \cdot \al_j }{ \al_i\cdot \al_i} \right)_{i,j \in I}.$$

For $\mathsf{b}= \sum_{i \in I}n_i\alpha_i \in \rl^+$ and $k \in \Z_{\ge 0}$, we define
$$\supp(\mathsf{b}) \seteq \left\{ i \in I \ | \ n_i \ne 0 \right \}.% \text{ and } \supp_{\ge k}(\mathsf{b}) \seteq \{ i \in I \ | \ n_i \ge k \}.$$
$$

We denote by $U_q(\mathsf{g})$ the \defn{quantum group} associated with a Cartan datum $(\cmA,\wl,\Pi,\wl^{\vee},\Pi^{\vee})$.

\subsection*{Dynkin diagrams and Positive roots}

In this paper, we mainly deal with the \defn{Dynkin diagrams} of type $A_n$ and $D_n$, $E_{6,7,8}$.
In the following, we list Dynkin diagrams $\Delta$ with an enumeration of vertices by simple roots.

\vskip -1.5em

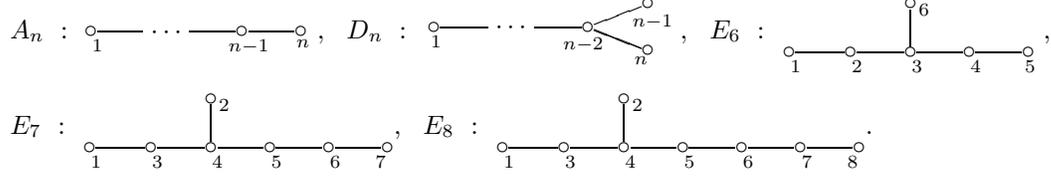
\begin{align*}
& A_n \ : \ \xymatrix@C=4ex@R=3ex{ *{ \circ }<3pt> \ar@{-}[r]_<{1}  &\cdots\ar@{-}[r] &*{\circ}<3pt>
\ar@{-}[r]_<{n-1} &*{\circ}<3pt>
\ar@{-}[l]^<{\ \ n}},  \ \
D_n \ : \ \raisebox{1.3em}{\xymatrix@C=4ex@R=0.1ex{
&&&*{\circ}<3pt> \ar@{-}[dl]^<{n-1}  \\
*{\circ}<3pt>\ar@{-}[r]_<{1 \ }&\cdots\ar@{-}[r]  &*{\circ}<3pt>
\ar@{-}[dr]_<{n-2} \\  &&&*{\circ}<3pt>
\ar@{-}[ul]^<{\ \ n}}}, \ \
 E_6 \ : \ \raisebox{1.3em}{\xymatrix@C=4ex@R=3ex{ && *{\circ}<3pt>\ar@{-}[d]^<{6} \\
*{ \circ }<3pt> \ar@{-}[r]_<{1}  &*{\circ}<3pt>
\ar@{-}[r]_<{2} &*{ \circ }<3pt> \ar@{-}[r]_<{3} &*{\circ}<3pt>
\ar@{-}[r]_<{4} &*{\circ}<3pt>
\ar@{-}[l]^<{\ \ 5}}}, \allowdisplaybreaks \\
& E_7 \ : \ \raisebox{1.3em}{\xymatrix@C=4ex@R=3ex{ && *{\circ}<3pt>\ar@{-}[d]^<{2} \\
*{ \circ }<3pt> \ar@{-}[r]_<{1}  &*{\circ}<3pt>
\ar@{-}[r]_<{3} &*{ \circ }<3pt> \ar@{-}[r]_<{4} &*{\circ}<3pt>
\ar@{-}[r]_<{5} &*{\circ}<3pt>
\ar@{-}[r]_<{6} &*{\circ}<3pt>
\ar@{-}[l]^<{7} } }, \ \
E_8 \ : \ \raisebox{1.3em}{\xymatrix@C=4ex@R=3ex{ && *{\circ}<3pt>\ar@{-}[d]^<{2} \\
*{ \circ }<3pt> \ar@{-}[r]_<{1}  &*{\circ}<3pt>
\ar@{-}[r]_<{3} &*{ \circ }<3pt> \ar@{-}[r]_<{4} &*{\circ}<3pt>
\ar@{-}[r]_<{5} &*{\circ}<3pt>
\ar@{-}[r]_<{6} &*{\circ}<3pt>
\ar@{-}[r]_<{7} &*{\circ}<3pt>
\ar@{-}[l]^<{8} } }.
\end{align*}
We denote by $\Delta(k,l)$ the \defn{distance} between vertices $k$ and $l$ in $\Delta$. We say that a vertex $i$ in $\Delta$ \defn{extremal} if
there exists only one arrow incident with $i$.

To denote a positive root, we will use several notations as follows (see~\cite[PLATE I$\sim$IX]{Bour}):
\begin{enumerate}
\item[{\rm (a)}] For $A_n$ cases, we use a notation $[a,b]$ for the positive root $\sum_{a \le k \le b} \al_k$ where $a \le b \le n$. In particular, if $a=b$, then we use $[a]$
instead of $[a,a]$.
\item[{\rm (b)}] For $D_n$ cases, we use a notation $\lf  a , \pm b \rf$ for the positive root $\ve_a\pm \ve_b$ where $a < b \le n$.
For $b \in \Z_{< 0}$, we sometimes use $\ve_b$ instead of $-\ve_{-b}$ and $\lf a , b \rf$ ($1 \le a \le -b$) instead of $\lf a , -(-b) \rf$.
\item[{\rm (c)}] For $E_6$, $E_7$ and $E_8$ cases, we use a notation $(c_1c_2\cdots c_n)$  for the positive root
$\sum_{1 \le k \le n} c_k\al_k$ $(n=6,7,8)$ where $c_k \in \Z_{\ge 0}$.
\end{enumerate}

\section{Positive roots of finite type} \label{sec: Positive roots of finite}
In this section, we briefly recall the system of positive roots of finite types and introduce orders on the system. Then we define
statistics by considering Poset structure given by the orders.

\subsection{System of positive roots.}  We choose $\cmA$ as the Cartan matrix of
a finite-dimensional simple Lie algebra $\g_0$ and the index set $I$ as
$\{ 1,2,\ldots,n\}$. % We denote by $\mathsf{FCD}$ the finite Cartan datum associated to $\g_0$.
We also denote by $\PR$ the set of \defn{positive roots} and $\NR$ the set of \defn{negative roots} associated to $\g_0$.

Let $\W$ be the Weyl group associated to $\g_0$, which is generated by \defn{simple reflections} $(s_i)_{i \in I}$.
We denote by $\ell(w)$ the \defn{length} of an element $w \in \W$ and $w_0$ the \defn{longest element} of $\W$. We also denote by $^*$ the involution on $I$ induced by
$w_0$; i.e.,
\begin{align} \label{eq: * involution}
w_0(\al_i) = -\al_{i^*}.
\end{align}

\medskip

The following proposition is well-known (see, for instance~\cite{Bour}).

\begin{proposition}
For $w \in \W$ and its reduced expression
$\tw=s_{i_1} \cdots s_{i_t}$, the set
\begin{align} \label{eq: PRw}
\PR_{\tw} \seteq \{ s_{i_1} \cdots s_{i_{k-1}}(\al_{i_k}) \ | \ 1 \le k \le t\}
\end{align}
is contained in $\PR$ and has its cardinality as $t=\ell(w)$.
Moreover, the set does not depend on the choice of reduced expression of $w$.
\end{proposition}
Thus we can write $\PR_{w}$ instead of $\PR_{\tw}$.

%Thus we write $\PR_{w}$ instead of $\PR_{\tw}$. Note that $\PR_{w}$ can be also characterized as follows (\cite{Bour}):
%\begin{align} \label{eq: cha of Phiw}
%\PR_{w} = \{ \be \in \PR \ | \ w^{-1}(\be) \in \NR  \}.
%\end{align}

\subsection{Total orders on $\PR_w$ and convexity.} For a fixed reduced expression $\tw=s_{i_1} \cdots s_{i_t}$,
the characterization $\PR_w$ in~\eqref{eq: PRw} gives a way of defining the total order $<_\tw$ on $\PR_w$ as follows:  Set $\be_k \seteq s_{i_1} \cdots s_{i_{k-1}}(\al_{i_k})$.
\begin{align} \label{def: total order tw}
\be_k <_\tw \be_l \quad \text{  if and only if } \quad k < l.
\end{align}

The following theorem tells that $<_\tw$ on $\PR_w$ is convex in the following sense:

\begin{theorem}~\cite{Papi94,Zel87} %
 For $\al,\be \in \PR_w$ with $\al+\be \in \PR_w$, we have either
\[
\al <_\tw  \al+\be <_\tw \be \quad \text{ or } \quad \be <_\tw \al+\be <_\tw \al.
\]
\end{theorem}

%By abstracting the property of $<_{\tw}$, we defined a notion of {\it convex order} on the subset $L$ of $\PR$ as follows:

We say an order $<$ ({\it not necessarily total order}) on a subset $L \subset \PR$ \defn{convex} if $\al,\be \in L$ with $\al+\be \in L$, we have either
$ \al <  \al+\be < \be$ or $\be < \al+\be < \al$.

\subsection{Commutation equivalence relation and convex partial orders}
We say that two reduced expressions $\widetilde{w}=s_{i_1}s_{i_2}\cdots s_{i_{\ell(w)}}$ and $\widetilde{w}'=s_{j_1}s_{j_2}\cdots s_{j_{\ell(w)}}$ of $w \in \W$
are \defn{commutation equivalent}, denoted by $\widetilde{w} \sim \widetilde{w}'$ and $[\widetilde{w}]$ its equivalence class, if $s_{j_1}s_{j_2}\cdots s_{j_{\ell(w)}}$ is obtained
from $s_{i_1}s_{i_2}\cdots s_{i_{\ell(w)}}$ by applying the commutation relations
$s_{k}s_{l} = s_{k}s_{l}$ for $|k-l|>1$.

By using $[\redex]$ and $<_{\tw'}$ for all $\redex' \in [\redex]$, \defn{the convex partial order $\prec_{[\redex]}$} is defined on $\PR_w$ as follows:
\begin{align}\label{eq:coarset order}
\alpha\prec_{[\redex]}\beta \quad \text{ if and only if } \quad \alpha <_{\redex'}\beta \quad \text{ for any }\quad \redex' \in [\redex].
\end{align}

\begin{remark} \label{rmk:coarset order property}
For a commutation class $[\tw]$, if $\{\al,\be\} \subset \big(\PR_w\big)^2$ is an incomparable pair with respect to $\prec_{[\tw]}$, then there exist
reduced expressions $\tw^{(1)},\tw^{(2)} \in [\tw]$ such that
$$   \al <_{\tw^{(1)}} \be \quad \text{ and } \quad  \be <_{\tw^{(2)}} \al.$$
\end{remark}

\subsection{Partial orders on the set of sequences of $\PR_w$}
%\begin{remark} \label{conv: sequence}~\referee{Convention? Remark? all conventions are vague.} \Sejin{Rewrite in the better form!1}
Let us choose a reduced expression $\tw=s_{i_1}s_{i_2}\cdots s_{i_{\ell(w)}}$ of $w \in \W$. Fix the convex total order $\le_\tw$ as in~\eqref{def: total order tw}.
\begin{enumerate}
\item[{\rm (i)}] We identify a sequence $\um_\tw=(m_1,m_2,\ldots,m_{\ell(w)}) \in \Z_{\ge 0}^{\ell(w)}$ with
$$(m_1\beta^\tw_1 ,m_2\beta^\tw_2,\ldots,m_{\ell(w)} \beta^\tw_{\ell(w)}) \in   (\Z_{\ge 0}\be_k^\tw)_{1 \le k \le \ell(w)} .$$
\item[{\rm (ii)}] For $\um_\tw$ and another reduced expression $\tw'$ of $w$, $\um_{\tw'}$ is a sequence in $\Z_{\ge 0}^{\ell(w)}$
by considering $\um_\tw$ as a sequence of positive roots,
rearranging with respect to $<_{\tw'}$ and applying {\rm (i)}.
\end{enumerate}
For simplicity of notations, we usually drop the script $\tw$ if there is no fear of confusion.
%\end{convention}

The \defn{weight} $\wt(\um)$ of a sequence $\um$ is defined by $$\displaystyle\sum_{i=1}^{\ell(w)} m_i\beta_i \in \rl^+.$$

The following order on $\Z_{\ge 0}^{\ell(w)}$ was introduced in~\cite{Mc12} and can be considered as the bi-lexicographical order corresponding
to $<_{\redex}$.

\begin{definition}~\cite{Mc12} \label{def: redezletb} % \sejin{Should mention about bi-lexicographically}
For sequences $\um$, $\um' \in \Z_{\ge 0}^{\ell(w)}$, we define an order $\le^\tb_{\tw}$ as follows:
\begin{eqnarray*}&&
\parbox{95ex}{
$\um'=(m'_1,\ldots,m'_{\ell(w)}) <^\tb_{\redex}  \um=(m_1,\ldots,m_{\ell(w)}) $ if and only if $\wt(\um')=\wt(\um)$, and there exist integers $k$, $s$ such that $1 \le k \le s \le \ell(w)$,
$m_t'=m_t$ $(t<k)$, $m'_k  <  m_k$, and $m_t'=m_t$ $(s<t\le \ell(w))$,
$m'_{s}  <  m_{s}$.
}
\end{eqnarray*}
\end{definition}
\noindent
Note that $\le^\tb_{\redex}$ is a partial order on the set of sequences of length $\ell(w)$.

\medskip

As we define $\prec_{[\tw]}$ from $<_{\tw}$, we define new order  $\prec^\tb_{[\tw]}$.
%\referee{Why is far coarser than the previous one?}

\begin{definition} \label{def: redezprectb}
For sequences $\um$, $\um' \in \Z_{\ge 0}^{\ell(w)}$, we define an order $\prec^\tb_{[\tw]}$ as follows:
\begin{eqnarray}&&
\parbox{85ex}{
$\um'=(m'_1,\ldots,m'_{\ell(w)}) \prec^\tb_{[\tw]}  \um=(m_1,\ldots,m_{\ell(w)}) $ if and only if  $\um'_{\tw'} <^\tb_{\tw'}  \um_{\tw'}$
for all reduced expressions $\tw' \in [\tw]$.
}\label{eq: redezprectb}
\end{eqnarray}
\end{definition}

\begin{remark}
Note that the order $\prec^\tb_{[\tw]}$ is {\it far coarser} than $<^\tb_{\tw}$, since some sequences are comparable with respect to $\prec^\tb_{[\tw]}$ only if they are
comparable with respect to $<^\tb_{\tw'}$ for all $\tw'\in [\tw]$.
\end{remark}

\begin{example}
For a reduced expression $\tw= s_3s_2s_1s_3s_4s_2s_4s_3s_1s_2$ of type $D_4$, we have
$$  (\lf 2,-3 \rf,\lf 1,3 \rf) <^\tb_{\tw} (\lf 1,-4 \rf,\lf 2,4 \rf) $$
while they are {\it not} comparable with respect to $\prec^\tb_{[\tw]}$.
\end{example}

\subsection{Sequences and new statistics} \label{subsec: new notion} In this subsection, we define new statistics arising from the Poset $(\Z_{\ge 0}\be_k^\tw)_{1 \le k \le \ell(w)} \simeq \Z_{\ge 0}^{\ell(w)}$ with respect to the partial order $\prec^\tb_{[\tw]}$.
These statistics will be used in the later sections to investigate the categories which we are interested in.

%\referee{This subsection should be re-written. What are the motivation of these definitions?}
%\sejin{Use the paper with Travis about exceptional types.}

A sequence $\um \in \Z^{\ell(w)}_{\ge 0}$ is called a \defn{pair} if $m_i \le 1$ for all $i$, and $|\um|\seteq \sum_{i=1}^{\ell(w)} m_i=2$.
In this paper, we use the notation $\up$ for a pair sequence. For brevity, we write a pair $\up$ as $(\alpha,\beta) \in (\PR_w)^2$ or $(\up_{i_1},\up_{i_2})$ such that
$\beta_{i_1}=\alpha$, $\beta_{i_2}=\beta$ and $i_1 < i_2$.

We say a sequence $\um=(\um_1,\um_2,\ldots,\um_{\ell(w)}) \in \Z^{\ell(w)}_{\ge 0}$ is \defn{$[\tw]$-simple} if it is minimal with respect to the partial order $\prec^\tb_{[\tw]}$.
For a given $[\tw]$-simple sequence $\us=(s_1,\ldots,s_{\ell(w)}) \in \Z^{\ell(w)}_{\ge 0}$, we call a cover\footnote{Recall that a cover of $x$ in a poset $P$ with partial order $\prec$ is an element $y \in P$ such that $x \prec y$ and there does not exists $y' \in P$ such that $x \prec y' \prec y$.} of $\us$ under $\prec^{\tb}_{[\tw]}$ a \defn{$[\tw]$-minimal sequence of $\us$} (see also~\cite{Mc12}).

For a pair $\up$, the \defn{$[\tw]$-distance} of $\up$ is the {\it largest} integer $k \geq 0$ such that
\[
\up^{(0)} \prec^\tb_{[\tw]} \cdots \prec^\tb_{[\tw]} \up^{(k)} = \um
\]
and $\up^{(0)}$ is a $[\tw]$-simple pair. We denoted $[\tw]$-distance by $\dist_{[\tw]}(\up)$.
(see Remark~\ref{ex: non-trivial 1} for example).
%\end{definition}

Consider a pair $\up$ such that there exists a unique $[\tw]$-simple sequence $\us$ satisfying
$\us \preceq^\tb_{[\tw]} \up$, we call $\us$ the \defn{$[\tw]$-socle} of $\up$ and denoted it by $\soc_{[\tw]}(\up)$.
At this moment, the existence and the uniqueness of $\soc_{[\tw]}(\up)$ are not guaranteed. In later section,
we will prove that $\soc_{[\tw]}(\up)$ exists uniquely for a certain family of commutation classes $[\tw]$.
The following is a consequence of Remark~\ref{rmk:coarset order property} {\rm (ii)}.

\begin{lemma} \label{lem: incomp simple}
For a pair $(\al,\be)$ which is incomparable with respect to $\prec_{[\tw]}$, $$\soc_{[\tw]}(\al,\be)=(\al,\be).$$
\end{lemma}

\begin{proposition} \label{pro: BKM minimal}~\cite[Lemma 2.6]{BKM12}
For $\ga \in \PR \setminus \Pi$ and any $\redez$ of $w_0$, a $[\redez]$-minimal sequence of $\ga$ is indeed a pair $(\al,\be)$ for some $\al,\be \in \PR$ such that
$\al+\be = \ga$.
\end{proposition}

%\Fixme{Up to here.}

\begin{example} \label{ex: D4} \hfill
\begin{enumerate}
\item[(a)] Let us consider the reduced expression $\redez$ of $D_4$ given as follows:
$$ \redez= s_3s_2s_1s_4s_3s_2s_1s_4s_3s_2s_1s_4.$$
Then we have (see also Example~\ref{ex: example D})
\begin{itemize}
\item $(\lf 2 , 4\rf ,\lf 1 , 3 \rf )$ is $[\redez]$-simple,
\item $(\lf 2 , -4\rf ,\lf 1 , 2 \rf )$ is not $[\redez]$-simple and
\[
\soc_{[\redez] }\bigl(\lf 2 , -4 \rf ,\lf 1 , 2 \rf \bigr)= \bigl(\lf 1 , -4 \rf ,\lf 2 , 3 \rf ,\lf 2 , -3 \rf \bigr),
\]
\item minimal pairs of  $(\lf 1 , 2 \rf )$ are
\[
\left\{ \bigl(\lf 1 , -4 \rf ,\lf 2 , 4 \rf \bigr), \
 \bigl(\lf 2 , -3 \rf ,\lf 1 , 3 \rf \bigr), \ \bigl(\lf 2 , 3 \rf ,\lf 1 , -3 \rf \bigl) \right\}.
 \]
\end{itemize}
\item[(b)] Let us consider the following reduced expression $\redez$ of $w_0$ of $E_6$:
\begin{align*}
\qquad \redez  = s_1s_2s_6s_3s_5s_4s_6s_1s_3s_2s_6s_3s_5s_6 s_4s_1s_3s_2s_6
s_3s_5s_6s_4s_1s_3s_2s_6s_3s_5s_6s_4s_1s_3s_2s_6s_3.
\end{align*}
Then one can check that, the pair $\up=(110000,123211)$
is not $[\redez]$-simple with $\dist_{[\redez]}(\up)=1$, and there are three $[\redez]$-simple pairs
$$ \us^{(1)}=(123111,111100), \ \us^{(2)}=(111000,122211), \ \us^{(3)}=(111110,122101)$$
%$$ \us^1=(12^23^2456,1234), \ \us^2=(123,12^23^24^256), \ \us^3=(12345,12^23^246) $$
satisfying $\wt(\us^{(i)}=\wt(\up)$ and $\us^{(i)} \prec_{[\redez]}^\tb \up$ for all $1 \le i \le 3$. Thus $\soc_{[\redez]}(\up)$ is {\it not} well-defined for the $[\redez]$
(see Appendix~\ref{Sec:Dynkin E6} (2)).
\end{enumerate}
\end{example}

Using the notion of $[\tw]$-distances only for pairs in $(\PR_w)^2$ (not sequences), we shall define  sequences of  pairs which are
{\it nearest} with respect to $\prec^{\tb}_{[\tw]}$
satisfying some property. These notion will be crucially used in later section, when we consider the socle and composition length of modules in certain categories.

\begin{definition} \label{def: tw-adjacent}  %\referee{Motivation??}
For  pairs $\up'=(\al^{(1)},\be^{(1)}) \prec^\tb_{[\tw]} \up=(\al^{(2)},\be^{(2)}) \in (\PR_w)^2$, we say that they are \defn{good adjacent neighbors} if
\begin{enumerate}
\item[{\rm (i)}] there exists $\eta \in \PR_w$ satisfying one of the following conditions:
\begin{enumerate}
\item[{\rm (a)}] $\eta+\beta^{(2)}=\beta^{(1)}, \ \eta+\al^{(1)}=\al^{(2)}$ and $\dist_{[\tw]}(\eta,\beta^{(2)}),\dist_{[\tw]}(\eta,\al^{(1)})<\dist_{[\tw]}(\up)$,
\item[{\rm (b)}] $\beta^{(1)}+\eta=\beta^{(2)}, \ \al^{(2)}+\eta=\al^{(1)}$ and $\dist_{[\tw]}(\beta^{(1)},\eta),\dist_{[\tw]}(\al^{(2)},\eta)<\dist_{[\tw]}(\up)$,
\end{enumerate}
\item[{\rm (ii)}] there exists no pair $\up''\prec^\tb_{[\tw]} \up$ such that it satisfies the conditions {\rm (i)} or
$$ \up' \prec^\tb_{[\tw]} \up'' \prec^\tb_{[\tw]} \up.$$
\end{enumerate}
\end{definition}

\begin{example}
For the $Q$ in Appendix~\ref{Sec:Dynkin E7}, one can check the pair $(0112100,0000011)$
is the cover of $(0111111,0001000)$ with respect to $\prec^{\tb}_{Q}$ but they are {\it not} good adjacent pair. More precisely,
\begin{itemize}
\item $(\al',\be') \seteq (0111111,0001000) \prec^{\tb}_{Q} (\al,\be) \seteq (0112100,0000011)$,
\item $\al - \al' \not \in \PR$ and $\be - \be' \not \in \PR$.
\end{itemize}
\end{example}

\begin{remark}\label{rem:p-dist}
When we restrict the order $\prec^\tb_{[\tw]}$ only on the set of pairs $(\PR_w)^2$, $\up$ in Definition~\ref{def: tw-adjacent} can be thought as a cover of $\up'$
satisfying the condition {\rm (i)}.
\end{remark}

\begin{definition} \label{def: good neighbor} %~\referee{Motivation??}
For a pair $\up\in (\PR_w)^2,$
the \defn{$[\tw]$-length} of $\up$, denoted by $\len_{[\tw]}(\up)$, is the integer which counts the number of all non $[\tw]$-simple pairs
$\up'\in (\PR_w)^2$ satisfying the following properties:
\begin{itemize}
\item $\up' \prec^\tb_{[\tw]} \up$ and there exists a sequence of pairs
$$\up^{(0)}=\up' \prec^\tb_{[\tw]} \up^{(1)} \prec^\tb_{[\tw]} \up^{(2)} \prec^\tb_{[\tw]} \cdots \prec^\tb_{[\tw]} \up^{(k)}=\up \quad  (k \in \Z_{\ge 1})$$
such that $\up^{(i)}, \up^{(i+1)}$ are good adjacent neighbor for all $0 \le i \le k-1$.
\end{itemize}
We call the pairs $\up',\up$ \defn{good neighbors} (see Example~\ref{ex: good neighbor} for good neighbors).
\end{definition}

\smallskip

Now we consider when $w$ is the longest element $w_0$. For a non-simple positive root $\gamma \in \PR \setminus \Pi$, the \defn{$[\redez]$-radius} of $\gamma$, denoted by $\rds_{[\redez]}(\gamma)$, is the integer defined as follows
$$\rds_{[\redez]}(\gamma)=\max({\rm dist}_{[\redez]}(\up) \ | \ \gamma  \prec_{[\redez]}^\tb \up).$$

\begin{example} In Example~\ref{ex: D4}, one can check that
$$\text{$\rds_{[\redez]}(\lf 1,2 \rf)=2$ and $\rds_{[\redez]}(\alpha)=1$ for all $\al \in \PR \setminus  \Pi \sqcup \{ \lf 1,2 \rf \} $ }$$
(see Example~\ref{ex: rds not eq mul} for exceptional cases).
\end{example}

\section{Auslander-Reiten quiver} \label{sec: AR-quiver}

In this section, we briefly review the Auslander-Reiten quiver and its basic properties. For more detail, we refer~\cite{ARS,ASS,HL11}.

\subsection{Dynkin quiver $Q$ of finite type ADE} Let $Q$ be a \defn{Dynkin quiver} of a Dynkin diagram $\Delta$ of type $A_n$, $D_n$ and $E_{6,7,8}$.
For any $i \in I$, let $s_iQ$ denote the quiver obtained by $Q$ by reversing the arrows incident with $i$.

For a reduced expression $\tw=s_{i_1}s_{i_2}\cdots s_{i_{\ell(w)}}$ of $w \in \W$, it is called
\defn{adapted to $Q$} if
\begin{align} \label{eq: adapted}
\text{ $i_k$ is a source of the quiver $s_{i_{k-1}} \cdots s_{i_2}s_{i_1}Q$ for all $1 \le k \le \ell(w)$.}
\end{align}

For a reduced expression $\redez$ of $w_0$ adapted to $Q$, if $\redez' \in [\redez]$, then $\redez'$ is adapted to $Q$. Conversely,
any $\redez'$ adapted to $Q$ is contained in $[\redez]$.
Thus the set of all reduced expressions $[Q]$ of $w_0$ adapted to $Q$ is well-defined.

\begin{remark} \label{rem: coxeter element and sink}
The followings are well-known:
\begin{itemize}
\item[{\rm (i)}] There is a unique \defn{Coxeter element} $\cox{Q} \in \W$ (a product of all simple reflections) whose reduced expressions are adapted to $Q$.
\item[{\rm (ii)}] $\cox{Q}^{-1}$ is the Coxeter element of $Q^{{\rm rev}}$ where $Q^{{\rm rev}}$ is the quiver obtained
by reversing all arrows of $Q$.
\item[{\rm (iii)}] For $\redez=s_{i_1} \cdots s_{i_{\N}}$ of $w_0$ adapted to $Q$ ($\N \seteq |\PR|$), we have
$$ s_{i_1} \cdots s_{i_{\N}} Q = Q^*,$$
where $Q^*$ is the quiver obtained from $Q$ by replacing vertices of $Q$ from $i$ to $i^*$.
\end{itemize}
\end{remark}

%\Fixme{Up to here.}

\subsection{Auslander-Reiten quiver and system of positive roots} \label{subsec: AR-quiver and PR}
A map $\xiQ:I \to \Z$ is called a \defn{height function} on $Q$ if $\xiQ_j=\xiQ_i-1$ when there exists
an arrow $i \to j$ in $Q$. Since $Q$ is connected, such a height function on $Q$ is unique up to $\Z$.
Set
\[
\Z Q \seteq \{ (i,p) \in I \times \Z \ | \ p -\xiQ_i \in 2\Z \}.
\]
By assigning arrows $(i,p) \to (j,p+1)$ for indices $i,j \in I$ with $\Delta(i,j)=1$, we call $\Z Q$ \defn{the repetition quiver}.
Note that $\Z Q$ does not depend on $Q$ but only on $\Delta$.

For $i \in I$, we define positive roots $\gamma^Q_i$ and $\theta^Q_i$ in the following way:
\begin{align}\label{eq: gamma,theta 2}
\gamma^Q_i = \sum_{j \in B(i)} \alpha_j \qquad  \text{and} \qquad \theta^Q_i = \sum_{j \in C(i)} \alpha_j,
\end{align}
where  $B(i)$ (resp. $C(i)$)  is the set of vertices $j$
such that there exists a path from $j$ to $i$ (resp. from $i$ to $j$) in $Q$.

The following relationship between $\{\gamma^Q_i \}$ and $\{\theta^Q_i \}$ is known as \defn{Nakayama permutation}:
\begin{align} \label{eq: Nakayama}
\theta^Q_{i^*} = \tau_Q^{\mQ_i}(\gamma^Q_{i}),\quad \text{ where } \mQ_i \seteq \max (k \ge 0 \ | \ \tau_Q^k(\ga^Q_i) \in \PR).
\end{align}

Set $\WPR \seteq \PR \times \Z$. There exists a bijection $\phi_Q: \Z Q \to \WPR$ in \cite[\S 2.2]{HL11}: %~\referee{The existence of the bijection $\phi_Q$ should be stated as a
%proper result and explained as it looks to be one of the motivations.} \sejin{Some credit for HL?}
\begin{itemize}
\item[{\rm (i)}] $\phi_Q(i,\xiQ_i) \seteq (\gamma^Q_i,0)$,
\item[{\rm (ii)}] for a given $\beta \in \PR$ with $\phi_Q(i,p)=(\beta,m)$, if $\cox{Q}^{\pm 1}(\beta) \in \Phi^\pm$, set
\[
\phi_Q(i,p \mp 2)=(\pm\cox{Q}^{\pm 1}(\beta), m).
\]
\end{itemize}

For $\be \in \PR$ and $\phi_Q^{-1}(\beta,0)=(i,p) \in I \times \Z$, we denote by
\[
 i=\phi^{-1}_{Q,1}(\beta) \qquad \text{ and } \qquad p=\phi^{-1}_{Q,2}(\beta).
\]
We call $i$ the \defn{residue} of $\beta$ with respect to $Q$.

The \defn{Auslander-Reiten quiver} (AR-quiver) $\GQ$ is the full subquiver of $\Z Q$ whose set of vertices is $\phi_Q^{-1}(\PR \times 0)$.
Thus we can label the vertices of $\GQ$ by $\PR$.

It is well-known that an AR-quiver $\GQ$ satisfies the \defn{additive property} which can be described as follows:
For $\al \in \PR$ we have
\begin{align} \label{eq: additive}
 \al + \tau_Q^{\pm 1}(\al) = \displaystyle\sum_{\be \in \al^\mp} \be \qquad  \text{ if $\tau^{\pm 1}_Q(\al) \in \PR$},
\end{align}
where $\al^-$ (resp. $\al^+$) denotes the set of positive roots $\be$ such that $\be \to \al$ (resp. $\al \to \be$) in $\GQ$.

%Furthermore, AR-quiver $\GQ$ can be characterized in $ \Z Q$ as follows:
%\begin{align} \label{eq: cha AR}\phi^{-1}_Q(\PR,0)=\{ (i,p) \in \Z Q \ | \ \xiQ_i -2\mQ_i \le p \le \xiQ_i\}.
%\end{align}

\medskip

Interestingly, $\GQ$ can be understood as the Hasse diagram of $\prec_{[\redez]}$ on $\PR_{w_0}=\PR$ for $\redez \in [Q]$.

\begin{theorem}~\cite{B99,R96} \label{thm: Ar-quiver and equ redu}
$\al \prec_{[Q]} \be$ if and only if there exists a path from $\be$ to $\al$ in $\GQ$.
Each $\redez' \in [Q]$ can be obtained by reading the residue of every vertex in a way compatible with the opposite directions of arrows.
\end{theorem}

\begin{remark}
In~\cite{OS15}, the author and Suh introduced new combinatorial model, \defn{combinatorial Auslander-Reiten quiver} $\Upsilon_{[\redex]}$
for {\it any} reduced expression $\redex$ of $w \in \W$ of {\it any} finite type and give an effective labeling algorithm for $\Upsilon_{[\redex]}$.
It can be understood as a generalization of AR-quiver since it
corresponds to the Hasse diagram of $\prec_{[\redex]}$ and every $\redez' \in [\redez]$ can be read from $\Upsilon_{[\redex]}$.
\end{remark}

\begin{remark}  Hereafter, we use $\prec_Q$, $\prec^\tb_Q$, $\soc_Q$, $\dist_Q$, $\len_Q$ and $\rds_Q$ instead of
$\prec_{[Q]}$, $\prec^\tb_{[Q]}$, $\soc_{[Q]}$, $\dist_{[Q]}$, $\len_{[Q]}$ and $\rds_{[Q]}$, respectively.
\end{remark}

\subsection{Reflection functor}\label{subsec: reflection functor} %For a sink $i$ (resp. a source $j$) of $Q$, we denote by $\mathbf{r}^+_i$ (resp. $\mathbf{r}^-_j$)
%the {\it reflection functor} from the category ${\rm Mod}(\C Q)$ to the category ${\rm Mod}(\C (s_iQ))$ (resp. ${\rm Mod}(\C (s_jQ))$). In this subsection,
%we briefly review the reflection functors $\mathbf{r}^+_i$ in the view points of changing adapted reduced expressions of $w_0$ and changing systems of $\PR$ related to the orders.
%
The following proposition is well-known:

\begin{proposition}
For $\redez=s_{i_1}s_{i_2}\cdots s_{i_{\N-1}}s_{i_\N}$,
$\mathbf{r}^+_{i_\N} \cdot \redez \seteq \redez'=s_{i_\N^*}s_{i_1}s_{i_2}\cdots s_{i_{\N-1}}$ is a reduced expression of $w_0$ and $[\redez'] \ne [\redez]$. Similarly,
$\mathbf{r}^-_{i_1} \cdot \redez \seteq \redez''=s_{i_2}\cdots s_{i_{\N-1}}s_{i_\N}s_{i^*_1}$ is a reduced expression of $w_0$
and $[\redez''] \ne [\redez]$.
\end{proposition}

The maps $\mathbf{r}^\pm_{i}$ are called \defn{the reflection functors}. For a sink $i$ of $Q$, Equation~\eqref{def: total order tw}, Remark~\ref{rem: coxeter element and sink} and Theorem~\ref{thm: Ar-quiver and equ redu} tell that
\begin{enumerate}
\item[{\rm (i)}] there exists a reduced expression $\redez=s_{i_1}\cdots s_{i_{\N}} \in [Q]$ such that $i_{\N} = i^*$,
\item[{\rm (ii)}] there exists a convex total order $<_{\redez}$ determined by $\{\be^\redez_k \ | \ 1 \le k \le \N \seteq  \ell(w_0)\}$
such that $\be^\redez_\N=\theta^Q_{i^*}=\al_{i}$.
\end{enumerate}

\begin{remark} \label{rem: comb refl}
The reflection functor $\mathbf{r}^+_i$ can be interpreted as the way of obtaining $\AR{s_iQ}$ from $\GQ$ as follows: Let $\mathsf{h}^\vee$ be the dual Coxeter number corresponding to Dynkin diagram $\Delta$ of $Q$.
\begin{enumerate}
\item[{\rm (A1)}] Remove the vertex $(i^*,p)$ such that $\phi_Q(i^*,p)=\al_i$ and the arrows exiting from $(i^*,p)$ in $\GQ$.
\item[{\rm (A2)}] Add the vertex $(i,p+\mathsf{h}^\vee)$ and the arrows entering into $(i,p+\mathsf{h}^\vee)$ in $\Z Q$.
\item[{\rm (A3)}] Label the vertex $(i,p+\mathsf{h}^\vee)$ with $\al_i$ and change the labels $\be$ to $s_i(\be)$ for all $\be \in \GQ \setminus \{\al_i\}$.
\end{enumerate}
\end{remark}

Thus taking the reflection functor $\mathbf{r}^+_i$ provides
\begin{enumerate}
\item[{\rm (i$'$)}] $\mathbf{r}^+_i \cdot \redez \seteq s_{i}s_{i_1}\cdots s_{i_{\N-1}} \in [s_iQ]$,
\item[{\rm (ii$'$)}] the convex total order $<_{\mathbf{r}^+_i \cdot \redez}$ determined by $\{ \be^{\mathbf{r}^+_i \cdot \redez}_k \ | \ 1 \le k \le \N \}$ such that
$$\text{$\be^{\mathbf{r}^+_i \cdot \redez}_1=\al_i$} \quad \text{ and } \quad \text{$\be^{\mathbf{r}^+_i \cdot \redez}_{k+1}=s_i(\be^\redez_{k}) \in \PR$ for all $1 \le k \le \N-1$.}$$
\end{enumerate}

\begin{remark} \label{eq: observations reflection} % \fixme{Write more clearly}
For a sink $i$ of $Q$, let us denote by $Q' \seteq s_iQ$, $\al'\seteq s_i\al$ and $\be'\seteq s_i\be$ and assume that $\be' \in \PR$. Then we can observe the followings:
\begin{enumerate}
\item[{\rm(i)}] $\al \prec_{Q} \be$ if and only if $\al' \prec_{Q'} \be'$.
\item[{\rm(ii)}] $\dist_{Q}(\al,\be) =\dist_{Q'}(\al',\be')$.
\item[{\rm(iii)}] For $\ga \in \PR \setminus \Pi$,
\[
\rds_{Q}(\ga) =\rds_{Q'}(s_i\ga)
\]
unless $(\al,\al_i)$ is the unique {\it biggest} pair of $\gamma$
with respect to $\prec^\tb_{Q}$  determining $\rds_{Q}(\ga)$.
\item[{\rm(iv)}] Under the assumption that $\soc_{Q}$ (resp. $\soc_{Q'}$) is well-defined,
\[
\soc_{Q}(\al,\be)= \us \quad \text{ if and only if } \quad \soc_{Q'}(\al',\be')= s_i(\us) .
\]
Here $s_i(\us) \seteq (s'_i \ | \ 1 \le i \le \N)_{\mathbf{r}_i^+ \cdot \redez}$ where $\us=\us_{\redez}$ such that $\redez = s_{i_1} \cdots s_{i_{\N-1}}s_{i^*} \in [Q]$,
$s_i'=s_{i+1}$ $(1\le i \le \N-1)$ and $s_1'=s_{\N}=0$.
\end{enumerate}
\end{remark}

\begin{remark} \label{rem: strategy E} % \referee{Uniform?}
Note that there are only finitely many Dynkin quivers for type $E_{6,7,8}$. In later sections, we will prove our assertions for type $E_{6,7,8}$ by observing
that each argument holds for one special quiver $Q$ in Appendices.
Then we can check for other quivers by applying the strategy given in Remark~\ref{eq: observations reflection},
since every Dynkin quiver $Q'$ can be obtained from the $Q$ by applying reflections functors properly:
\[
Q'=s_{i_1}\cdots s_{i_r}Q \qquad \text{for some $r \in \Z_{\ge 0}$ and $i_t \in I$ $(1 \le t \le r)$.}
\]
\end{remark}

\section{Combinatorial description of $\GQ$ of type $A_n$ and $D_n$} \label{Sec: Sejin results}

In this section, we review combinatorial descriptions of $\GQ$ of type $A_n$ and $D_n$
which are analyzed in~\cite{Oh14A,Oh14D}. The combinatorial descriptions play an important role in later sections for proving an
existence of the socle for $[\redex]$ adapted to $Q$ and well-definedness of distance polynomial, etc. %\Referee{Is this section used? If it is you have to clarify!}
%\Sejin{Yes. I will. Rewrite again. Sometimes, remove and just refer.}

Let $\rl$ be the root lattice associated to the Dynkin diagrams of type $A_n$, $D_n$ and $E_{6,7,8}$.
The \defn{multiplicity} of $\gamma = \sum_{i \in I} n_i \alpha_i \in \rl^+ \setminus \Pi$ is the integer, denoted by $\mul(\gamma)$,
defined as follows:
$$ \mul(\gamma) := \max( n_i \ | \ i \in I ).$$
In particular, if $\mul(\ga)=1$, then we say that $\ga$ is \defn{multiplicity free}.

\smallskip

For an AR-quiver $\GQ$,
\begin{enumerate}
\item[{\rm (i)}] a pair $(\al,\be)$ of positive roots is \defn{sectional} in $\Gamma_Q$ if
$$  \Delta(i,j)= |p-q| \qquad \text{ for } \phi_Q^{-1}(\al,0)=(i,p) \quad \text{ and } \quad \phi_Q^{-1}(\be,0)=(j,q),$$
\item[{\rm (ii)}] a full subquiver $\rho$ of $\GQ$ is \defn{sectional} if every pair $(\al,\be)$ in $\rho$ is
sectional.
\item[({\rm iii})] a connected subquiver $\rho$ in $\GQ$ is called an \defn{$S$-sectional} (resp. \defn{$N$-sectional}) path if it is a concatenation of downward (upward) arrows,
and there is no longer connected path consisting of downward arrows (resp. upward arrows) containing $\rho$.
%\item[({\rm iv})] An $S$-sectional (resp. $N$-sectional) path $\rho$ is \defn{maximal} if there is no longer $S$-sectional (resp. $N$-sectional) path
%containing all positive roots in $\rho$.
\end{enumerate}
%In this paper, we always assume that every sectional path is maximal.
We write \defn{$N$-path} (resp. \defn{$S$-path}) instead of
$N$-sectional path (resp. $S$-sectional path) for brevity.

\subsection{Type $A_n$} In this subsection, we assume that $Q$ is of type $A_n$.
%We say that a vertex $i \in Q$ is a {\it right} (resp. {\it left}) {\it intermediate} if
%$$\leftrightinterA.$$
%
For $\be=[a,b] \in \PR$ of type $A_n$, we say $a$ the \defn{first component} of $\be$ and $b$ the \defn{second component} of $\be$.

\begin{example} \label{ex: example A}
Consider the quiver $Q$ $\xymatrix@R=3ex{ *{ \circ }<3pt> \ar@{->}[r]_<{1}  &*{\circ}<3pt>
\ar@{<-}[r]_<{2}  &*{\circ}<3pt>
\ar@{->}[r]_<{3} &*{\circ}<3pt>
\ar@{->}[r]_<{4}  & *{\circ}<3pt> \ar@{-}[l]^<{\ \ 5}
}$ of type $A_5$. We set $\xiQ_1=0$. % Then $\GQ$ inside $\Z Q$ can be drawn as follows:
$$\Gamma_Q= \raisebox{4.5em}{ \GammaA}$$
\end{example}

In Example~\ref{ex: example A}, the full subquivers
\begin{align*}
& [1,2]\to[1,5]\to[1,4]\to[1,3]\to[1]  \ \ \text{and} \ \ [5]\to[4,5]\to[2,5]\to[1,5]\to[3,5]
\end{align*}
are $N$-path and $S$-path, respectively.

\begin{theorem}\cite[Theorem 1.11]{Oh14A} \label{Thm: sectional A}
Every positive root in an $N$-path has the same first
component and every positive root in an $S$-path has the same
second component. Thus for $1 \le i \le n$, $\GQ$ contains an $N$-path with $(n-i)$-arrows
  once and exactly once. At the same time, $\GQ$ contains an $S$-path with $(i-1)$-arrows once and exactly once.
\end{theorem}
Hence we can say that an $N$-path $\rho$ is the \defn{$(N,i)$-path} if it contains all positive roots whose first components are $i$. Similarly, the notion \defn{$(S,i)$-path} is defined.

%\smallskip

Let $\kp$ and $\sigma$ be subsets of $\PR$ defined as follows:
\[
\kp \seteq \{  \beta \in \PR \ | \ \phi^{-1}_{Q,1}(\beta)=1  \}, \qquad \sigma \seteq \{  \beta \in \PR \ | \ \phi^{-1}_{Q,1}(\beta)=n  \}.
\]
We label the positive roots in $\kp=\{\kp_1,\ldots,\kp_r\}$ and $\sigma=\{\sigma_1,\ldots,\sigma_s\}$ in the following way:
$$  \phi^{-1}_{Q,2}(\kp_{i+1})+2= \phi^{-1}_{Q,2}(\kp_{i}), \   \phi^{-1}_{Q,2}(\sigma_{j+1})-2= \phi^{-1}_{Q,2}(\sigma_{j}) \   \text{for } 1 \le i < r
\text{ and } 1 \le j <s.$$

\begin{lemma}~\cite[Corollary 1.15]{Oh14A} \label{lem: first and last A} %\hfill
\begin{enumerate}
\item[{\rm(a)}] If $\kp_i=[a,b]$, then $\kp_{i+1}=[b+1,c]$ for some $a \le b < c$.
\item[{\rm(b)}] If $\sigma_j=[a,b]$, then $\sigma_{j+1}=[b+1,c]$ for some $a \le b < c$.
\end{enumerate}
\end{lemma}

\subsection{Type $D_n$} In this subsection, we assume that $Q$ is of type $D_n$. The
involution $^*$ induced by $w_0 \in \W$ is given by $i^*=i$ for $1 \le i \le n-2$ and $(n-1)^*=n-1$, $n^*=n$ if $n$ is even,
$(n-1)^*=n$, $n^*=n-1$ if $n$ is odd. The $\mQ_i$ in~\eqref{eq: Nakayama} is given by $\mQ_i=n-2$ for $1 \le i \le n-2$ and %(\cite[Lemma 1.12]{Oh14D})
\begin{align}\label{eq: mQ D}
\begin{cases}
\mQ_{n-1}=n-3, \ \mQ_n=n-1 & \text{ if } n \equiv 1 \ ({\rm mod} \ 2) \text{ and } \xiQ_{n}=\xiQ_{n-1}+2,
\\
\mQ_{n-1}=n-1, \ \mQ_n=n-3 & \text{ if } n \equiv 1 \ ({\rm mod} \ 2) \text{ and }\xiQ_{n-1}=\xiQ_{n}+2,\\
\mQ_{n-1}=\mQ_n=n-2 & \text{ otherwise. }
\end{cases}
\end{align}

Note that each $\beta \in \PR$ can be written as $\ve_a \pm \ve_b$. We say $\ve_a$ and $\pm \ve_b$ as \defn{summands} of $\be$.

\begin{example} \label{ex: example D}
Let us consider the quiver $Q=\raisebox{1.2em}{\xymatrix@R=0.5ex{
&&*{\circ}<3pt> \ar@{->}[dl]^<{ 3} \\
*{\circ}<3pt> \ar@{<-}[r]_<{1}  &*{\circ}<3pt>
\ar@{->}[dr]_<{2}\\
&&*{\circ}<3pt> \ar@{-}[ul]^<{\ \ \ 4}}}$
of type $D_4$. Then the $\GQ$ can be described as follows:
\[
\GQ = \raisebox{2.2em}{ \GammaD}.
\]
\end{example}

\begin{lemma}~\cite[Lemma 1.12]{Oh14D} \label{lem: n-1,n D}
For a Dynkin quiver $Q$, let $\{\ta,\ta'\} \in \{n-1,n\}$ given as follows:
\begin{align} \label{eq: def t}
\ta \seteq \begin{cases} n-1 \\  n \end{cases} \quad \text{ and } \quad \ta'= \begin{cases} n & \text{ if } \xi_{n-1}-\xi_{n}=\pm 2,  \\
n-1 & \text{ if } \xi_{n-1}-\xi_{n}=0. \end{cases}
\end{align}
 Then every positive root  $\be$ with $\phi^{-1}_{Q,1}(\be) \in \{n-1,n\}$ has $\ve_{\ta }$ or $-\ve_{\ta}$ as its summand. Conversely,
every positive root $\be$ having $\ve_{\ta }$ or $-\ve_{\ta}$ as its summand has its residue with respect to $Q$ as $n-1$ or $n$.
\end{lemma}

\begin{definition} \hfill
\begin{enumerate}
\item[({\rm a})] An $S$-path (resp. $N$-path) is \defn{shallow} if it ends (resp. starts) at level less than $n-1$.
\item[({\rm b})] A connected subquiver $\varrho$ in $\GQ$ is called a \defn{swing} if it consists of vertices and arrows in the following way: There exist roots
$\al,\be \in \PR$ and $r,s \le n-2$ such that
\begin{align} \label{eq: swing def}
\swing \text{ where }
\end{align}
the connected quiver on the left (resp. right) to $\begin{matrix}\al \\ \be \end{matrix}$ in~\eqref{eq: swing def} is $S$-sectional (resp. $N$-sectional) and $\phi^{-1}_{Q,1}(\al)=n$, $\phi^{-1}_{Q,1}(\beta)=n-1$.
\end{enumerate}
\end{definition}

The following lemma tells information on the positions of positive roots $\be$ with $\mul(\be)=2$ in $\GQ$.

\begin{lemma}~\cite[Corollary 1.15]{Oh14D} \label{lem: nfree position D} Set
\begin{align*}
& a=\max( \phi^{-1}_{Q,2}(\be) \mid \ \phi^{-1}_{Q,1}(\be) \in \nmn, \ \het(\be)\ge 2 ), \\
& b=\min( \phi^{-1}_{Q,2}(\be) \mid \ \phi^{-1}_{Q,1}(\be) \in \nmn, \ \het(\be)\ge 2 ).
\end{align*}
Then we have the followings:
\begin{enumerate}
\item[{\rm (a)}] $a-b=2(n-3)$.
\item[{\rm (b)}] $\be$ is \emph{not} multiplicity free if and only if
\begin{align} \label{eq: nfree root condition}
 1 < \ell \seteq \phi^{-1}_{Q,1}(\beta) < n-1 \qquad \text{ and } \qquad b-(n-1-\ell)\le  \phi^{-1}_{Q,2}(\be) \le a-(n-1-\ell).
\end{align}
\end{enumerate}
\end{lemma}

\begin{theorem} \label{Thm: V-swing}~\cite[Theorem 1.19]{Oh14D}
 For every swing $\varrho$, there exists $1 \le k \le n-2$ such that
all roots in $\varrho$ contain $\ve_k$ as their summand. Moreover, $\varrho$ contains a simple root $\al_k$ and
is one of the following two forms:

\vskip -2em

\begin{align*}
& \thetaswing, \allowdisplaybreaks\\
&\quad \ \  \gammaswing.
\end{align*}
\end{theorem}

From the above theorem, for $1 \le a \le n-2$, the swing $\varrho$, containing all positive roots with $\ve_a$ as its summand, is called by the
\defn{$a$-swing}.

\begin{theorem} \label{thm: short path}~\cite[Theorem 1.22]{Oh14D}
Let $\rho$ be a shallow $S$-path $($resp.~$N$-path$)$.
Then there exists $k \le n-2+\delta$ such that all positive roots in $\rho$ contain $-\ve_k$ as their summand and $\rho$ starts $($resp. ends$)$ at level $1$.
Here $\delta=1$ if $\{ n-1,n\}$ are sink or source, $\delta=0$ otherwise.
\end{theorem}

Similarly, we can define the notion of \defn{shallow $(N,-a)($resp.$(S,-a))$-path}.

\medskip

For the rest of this subsection, we record lemmas and notations in~\cite{Oh14D} which are essential for our assertions in later section.

\medskip

Let $\sigma$ be a subset of $\PR$ defined as follows:
\begin{equation} \label{eq: sigma}
\sigma \seteq \big\{  \be \in \PR \ | \ \phi^{-1}_{Q,1}(\be)=n-1, \  \be \not \in  \{ \al_{n-1},\al_{n}\} \big   \}.
\end{equation}
Note that $|\sigma|=n-2$. For $1 \le k \le n-3$, We set
\begin{eqnarray} &&
\parbox{95ex}{
\begin{itemize}
\item the positive roots in $\sigma=\{\sigma_1,\ldots,\sigma_{n-2}\}$ as $   \phi^{-1}_{Q,2}(\sigma_{k+1})+2= \phi^{-1}_{Q,2}(\sigma_{k})$,
\item indices $\{ i_{\sigma_1},i_{\sigma_2},\ldots,i_{\sigma_{n-2}}\}=\{ 1,2,\ldots,n-2 \}$ such that $\sigma_{k}$ is contained in $i_{\sigma_{k}}$-swing.
\end{itemize}}\label{eq:sigma enumeration}
\end{eqnarray}

\begin{lemma}~\cite[Corollary 1.25]{Oh14D}\label{lem: reverse uni} %\hfill
\begin{itemize}
\item[({\rm a})] A positive root $\be =\lf a , b \rf$ with $\mul(\be)=2$ $($equivalently, $a<b \in \Z_{\ge 1})$ is contained in the longer part of the $b$-swing.
\item[({\rm b})] There exists $1 \le \ell \le n-2$ such that $i_{\sigma_{\ell}}=1$ and
\begin{equation} \label{eq: reverse unimodal}
i_{\sigma_{1}}>i_{\sigma_{2}}>\cdots >i_{\sigma_{\ell}}=1 < i_{\sigma_{\ell-1}} < \cdots < i_{\sigma_{n-2}},
\end{equation}
where the shorter part of $i_{\sigma_{a}}$-swing $(a < \ell)$ is the $N$-part and the shorter part of $i_{\sigma_{b}}$-swing $(b > \ell)$ is the $S$-part.
\end{itemize}
\end{lemma}

Let $\kp$ be a subset of $\PR$ defined as follows:
\begin{equation} \label{eq: kappa}
\kp \seteq \{  \be \in \PR \ | \ \phi^{-1}_{Q,1}(\be)=1  \}.
\end{equation}
We label the positive roots in $\kp=\{\kp_1,\ldots,\kp_{n-1}\}$ in the following way:
$$   \phi^{-1}_{Q,2}(\kp_{i+1})+2=  \phi^{-1}_{Q,2}(\kp_{i}),  \quad  \text{ for } 1 \le i \le n-2.$$
Note that $|\kp|=n-1$ and each element in $\kp$ is contained in
only one shallow path, a sectional path sharing $\ve_{\ta'}$ or a sectional path sharing $-\ve_{\ta'}$.
We set, for $1 \le k \le n-1$,
indices $$\{ j_{\kp_{1}},j_{\kp_{2}},\ldots,j_{\kp_{n-1}} \}=\{ -2,\ldots,-n+2, \ta', -\ta' \}$$
such that $\kp_{s}$ contains $-\ve_{-j_{\kp_s}}$ or $\pm \ve_{\ta'}$ as its summand. % \sejin{Notation should be changed!}

%The following lemma tells the positions of multiplicity free roots inside of $\GQ$.

\begin{lemma}~\cite[Corollary 1.26]{Oh14D} \label{lem: first}
\begin{enumerate}
\item[{\rm (1)}] There exists $\mathtt{l}$ such that
$$ |j_{\kp_{n-1}}|< \ldots < |j_{\kp_{\mathtt{l}}}|=\ta'=|j_{\kp_{\mathtt{l}-1}}| > |j_{\kp_{\mathtt{l}-2}}| > \cdots > |j_{\kp_{1}}|.$$
\item[{\rm (2)}] For  $s \le \mathtt{l}-1$, the sectional path sharing $-\ve_{j_{\kp_s}}$ is the $S$-sectional.
\item[{\rm (3)}] For  $s \ge \mathtt{l}$, the sectional path sharing $-\ve_{j_{\kp_s}}$ is the $N$-sectional.
\end{enumerate}
\end{lemma}

\section{The sequences of positive roots with respect to new notions} \label{Sec: socle PR}

In this section, we prove the existence of $[\tw]$-socle when $\tw$ is adapted to some Dynkin quiver $Q$ by using the results in Section~\ref{Sec: Sejin results} and
the new statistics in Section~\ref{sec: Positive roots of finite}. Then we investigate the relationship between the new statistics and the system of positive roots corresponding
to the underlying Dynkin diagram $\Delta$ of $Q$. These results will be applied to the representation theories for KLR-algebras and quantum affine algebra related to $\Delta$.

%\sejin{Rewrite this section!}

\subsection{Socle of pairs}
\begin{proposition} \label{prop: dir Q cnt}
For a Dynkin quiver $Q$, let us assume that a pair $(\al,\be)$ is sectional in $\GQ$.
\begin{enumerate}
\item[{\rm (a)}] The pair $(\al,\be)$ is $[Q]$-simple.
\item[{\rm (b)}] $\al \cdot \be=1$.
\item[{\rm (c)}] Either $\al-\be$ or $\be-\al \in \PR$.
\item[{\rm (d)}] Either $(\al,\be-\al)$ is a $[Q]$-minimal pair of $\al$ or $(\be,\al-\be)$ is a $[Q]$-minimal pair of $\be$.
\end{enumerate}
\end{proposition}

\begin{proof}
Since $(\al,\be)$ is sectional, Theorem~\ref{thm: Ar-quiver and equ redu} implies that
there exist $w \in \W$ and a sectional path
$$\left\{ \beta=w(\al_{i_1}) \to ws_{i_1}(\al_{i_2}) \to \cdots \to w \left(\prod_{s=1}^{k} s_{i_{s}}\right)(\al_{i_{k+1}})=\alpha \right\} \text{ in $\Gamma_Q$}.$$
Since the pairing $ \cdot $ is invariant under the action $\W$, the first three assertions follow from the observation on the set
$$\left\{ \al_{i_1}, s_{i_1}(\al_{i_2}),s_{i_1}s_{i_2}(\al_{i_3}) \ldots,\left(\prod_{s=1}^{k} s_{i_{s}}\right)(\al_{i_{k+1}}) \right\}.$$
The fourth assertion for type $A_n$ follows from~\cite[Theorem 3.4]{Oh14A}, and the one for type $D_n$ follows from~\cite[Corollary 4.22]{Oh14D}.

The fourth assertion for type $E_{6,7,8}$ can be checked for the Dynkin quivers $Q$ and their AR-quivers $\GQ$ in Appendices. Then one can check for each Dynkin quiver $Q$,
by the arguments in Remark~\ref{eq: observations reflection}, which is obtained by applying reflection functors.
Surely, this argument is also applicable to types $A_n$ and $D_n$.
%\referee{??}
\end{proof}

The rest of this subsection is devoted to prove the following theorem:

\begin{theorem} \label{thm: socle of pair}
For a reduced expression $\tw$ adapted to some Dynkin quiver $Q$ and a pair $\up=(\al,\be)$,
$$\text{$\soc_{[\tw]}(\up)$ exists uniquely.}$$
\end{theorem}

\begin{proposition} \label{prop: Q-socle of pair with sum in PR}
For any pair $\up=(\al,\be)$ with $\ga=\al+\be \in \PR$ and $[\redez]$, $\soc_{[\redez]}(\al,\be)$ is well-defined by $\ga$. In particular,
$\dist_Q(\up)=1$ implies that $\up$ is a $[\redez]$-minimal pair of $\ga$.
\end{proposition}

\begin{proof}
If $(\al,\be)$ is a pair of simple roots, our assertion is trivial. Assume that $(\al,\be)$ is not a pair of simple roots and there exists a $[\redez]$-simple
$\us$ such that $\us \prec_{[\redez]}^\tb (\al,\be)$.
\begin{enumerate}
\item[{\rm (i)}] If $|\us|=1$, then we have $\us=(\gamma)$, then our assertion follows from Proposition~\ref{pro: BKM minimal}.
\item[{\rm (ii)}] If $|\us|=2$, then it must be a pair. Since $\prec_{[\redez]}$ is convex, it yields the contradiction to the assumption that $\us$ is $[\redez]$-simple.
\item[{\rm (iii)}]  Now we can assume that $|\us|>2$. Then there exist indices $i \ne j$ such that
$s_i \ne 0$, $s_j \ne 0$ and $\be_i + \be_j \in \PR$. Thus our first assertion follows from the definition of $[\redez]$-simple sequence and the convexity of $\prec_{[\redez]}$.
\end{enumerate}
The second assertion follows from Proposition~\ref{pro: BKM minimal}.% and the first assertion.
\end{proof}

\begin{remark} \label{rem: skip E} %~\referee{The authors should explain why they do not have a uniform proof for all types and if they except to get one.}
For the rest of this paper, we usually skip the proofs of our assertions for Dynkin quiver $Q$ of types $E_{6,7,8}$.
Note that the combinatorics for $\GQ$ of type $E_{6,7,8}$ is not-well studied, comparing with the types for $A_n$ and $D_n$, and the system of positive roots for the types are quite complicated.
Since many parts of proofs in this paper use the combinatorial properties of $\GQ$ and the system of positive roots,
 we need to develop the $E_{6,7,8}$-analogue theories in Section~\ref{Sec: Sejin results}, to give a uniform proof for $A_n$, $D_n$ and $E_{6,7,8}$. However, as we mentioned in Remark~\ref{rem: strategy E},
there are only finitely many Dynkin quivers of type $E_{6,7,8}$, and the proofs for $E_{6,7,8}$ can be obtained by observing a fixed $\GQ$ and applying the strategy in Remark~\ref{eq: observations reflection} and
Remark~\ref{rem: strategy E} via reflection functors on AR-quivers in Remark~\ref{rem: comb refl}. If we give proofs for types $E_{6,7,8}$ with details, the paper would be long than necessary.
Hence, we sometimes give examples for types $E_{6,7,8}$ instead of giving proofs, by using the Dynkin quivers $Q$ in Appendices.
\end{remark}

\begin{proposition} \label{prop: Q-socle for pair}
For any pair $\up=(\al,\be)$ and any Dynkin quiver $Q$ of type $ADE$, $\soc_Q(\al,\be)$ is well-defined.
\end{proposition}

By Lemma~\ref{lem: incomp simple}, Proposition~\ref{prop: dir Q cnt} and Proposition~\ref{prop: Q-socle of pair with sum in PR}, it suffices to consider a pair $(\al,\be)$ such that
\begin{align} \label{eq:suffice 1}
\text{ {\rm (i)} it is comparable with respect to $\prec_Q$,  \rm {(ii)}  it is not  sectional,  {\rm (iii)}  $\al+\be \not \in \PR$.}
\end{align}

%Note that for non-simple positive root $\be=\sum_{k=1}^{s \ge 2}{\be_{i_k}}$ , there exist $1 \le t \ne l \le s$ such that $\be_{i_t} \prec_Q \be \prec_Q %\be_{i_l}$ by
%the convexity of $\prec_Q$.

Assume that
\[
\supp(\al) \cap \supp(\be) = \emptyset \qquad \text{and} \qquad  \al+\be \not\in \PR.
\]
  If $\al$ and $\be$ are simple roots, there is no sequence $\us \ne (\al,\be)$ with $\wt(\us)=\al+\be$.
If  $(\al,\be)$ is not a pair of simple roots and $\us$ is $[Q]$-simple such that $(\al,\be) \ne \us \prec_Q^\tb (\al,\be)$, then
there exists an index $i$ such that
\begin{align*}
s_i \ne 0, \quad \text{ and } \quad \be_i \prec_Q  \al \prec_Q \be \text{ or } \al \prec_Q  \be \prec_Q \be_i,
\end{align*}
by the convexity of $\prec_Q$. Hence, such a $\us$ can not exist and $\soc_Q(\al,\be)=(\al,\be)$.
Now it  suffices to consider that a pair satisfies~\eqref{eq:suffice 1} and
\begin{align} \label{eq:suffice 2}
{\rm (iv)} \  \supp(\al) \cap \supp(\be) \ne \emptyset.
\end{align}

\begin{proof}[Proof of {\rm Proposition~\ref{prop: Q-socle for pair}} for $A_n$-types]
Under the first two assumptions in~\eqref{eq:suffice 1}, paths from $\be$ to $\al$ are one of the followings: Write $\al=[a,b]$ and $\be=[c,d]$.
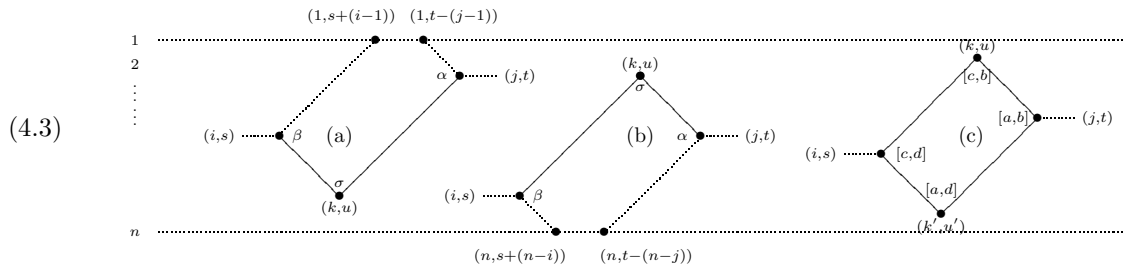
\begin{equation} \label{eq: A complicated socle}
\scalebox{0.8}{{\xy
(0,0)*{}="DL";(10,-10)*{}="DD";(20,20)*{}="DT";(30,10)*{}="DR";
"DL"; "DD" **\dir{-};"DL"; "DT"+(-4,-4) **\dir{.};
"DT"+(-40,-4); "DT"+(120,-4)**\dir{.};
"DD"+(-30,-6); "DD"+(130,-6) **\dir{.};
"DT"+(4,-4); "DR" **\dir{.};"DR"; "DD" **\dir{-};
"DL"+(-10,0)*{\scriptstyle (i,s)};
"DL"+(3,0)*{\scriptstyle \be};
"DL"+(10,0)*{{\rm (a)}};
"DL"+(-1,0); "DL"+(-6,0) **\dir{.};
"DR"+(10,0)*{\scriptstyle (j,t)};
"DR"+(-3,0)*{\scriptstyle \al};
"DR"+(1,0); "DR"+(6,0) **\dir{.};
"DL"*{\bullet};"DR"*{\bullet};
"DT"+(4,-4)*{\bullet};"DT"+(-4,-4)*{\bullet};"DD"*{\bullet};
"DD"+(0,2)*{\scriptstyle \sigma};
"DD"+(0,-2)*{\scriptstyle (k,u)};
"DT"+(-8,0)*{\scriptstyle (1,s+(i-1))};"DT"+(9,0)*{\scriptstyle (1,t-(j-1))};
"DT"+(-44,-4)*{\scriptstyle 1};
"DT"+(-44,-8)*{\scriptstyle 2};
"DT"+(-44,-12)*{\scriptstyle \vdots};
"DT"+(-44,-16)*{\scriptstyle \vdots};
"DT"+(-44,-36)*{\scriptstyle n};
"DL"+(40,-10); "DD"+(36,-6) **\dir{.};
"DR"+(40,-10); "DD"+(44,-6) **\dir{.};
"DT"+(40,-10); "DR"+(40,-10) **\dir{-};
"DT"+(40,-10); "DL"+(40,-10) **\dir{-};
"DL"+(30,-10)*{\scriptstyle (i,s)};
"DL"+(43,-10)*{\scriptstyle \be};
"DL"+(39,-10); "DL"+(34,-10) **\dir{.};
"DR"+(50,-10)*{\scriptstyle (j,t)};
"DR"+(37,-10)*{\scriptstyle \alpha};
"DR"+(41,-10); "DR"+(46,-10) **\dir{.};
"DL"+(40,-10)*{\bullet};"DR"+(40,-10)*{\bullet};
"DT"+(40,-10)*{\bullet};
"DT"+(40,-12)*{\scriptstyle \sigma};
"DT"+(40,-8)*{\scriptstyle (k,u)};
"DT"+(40,-10)*{\bullet};
"DD"+(44,-6)*{\bullet};"DD"+(36,-6)*{\bullet};
"DL"+(60,0)*{{\rm(b)}};
"DD"+(30,-10)*{\scriptstyle (n,s+(n-i))};
"DD"+(51,-10)*{\scriptstyle (n,t-(n-j))};
"DL"+(100,-3); "DD"+(100,-3) **\dir{-};
"DR"+(96,-7); "DD"+(100,-3) **\dir{-};
"DT"+(96,-7); "DR"+(96,-7) **\dir{-};
"DT"+(96,-7); "DL"+(100,-3) **\dir{-};
"DL"+(90,-3)*{\scriptstyle (i,s)};
"DL"+(105,-3)*{\scriptstyle [c,d]};
"DL"+(99,-3); "DL"+(94,-3) **\dir{.};
"DR"+(106,-7)*{\scriptstyle (j,t)};
"DR"+(92,-7)*{\scriptstyle [a,b]};
"DR"+(97,-7); "DR"+(102,-7) **\dir{.};"DL"+(115,0)*{{\rm(c)}};
"DL"+(100,-3)*{\bullet};"DR"+(96,-7)*{\bullet};
"DT"+(96,-7)*{\bullet}; "DD"+(100,-3)*{\bullet};
"DT"+(96,-5)*{\scriptstyle (k,u)};
"DT"+(96,-10)*{\scriptstyle [c,b]};
"DD"+(100,-5)*{\scriptstyle (k',u')};
"DD"+(100,1)*{\scriptstyle [a,d]};
\endxy}}
\end{equation}
Here $\phi_Q^{-1}(\al,0)=(j,t)$ and $\phi_Q^{-1}(\be,0)=(i,s)$.

\medskip

\noindent
{\rm (a)}  In this case, we have  $\sigma = [a,d]$  by Theorem~\ref{Thm: sectional A} where $\phi_Q^{-1}(\sigma,0)=(k,u)$.
\begin{itemize}
\item[{\rm (a-1)}] If $(t-s)-(i+j)+2 >2$, then Theorem~\ref{Thm: sectional A} implies that
\[
\eta=[x,b] \qquad \text{ and } \qquad \zeta=[c,y],
\]
where $\phi_Q^{-1}(\eta,0)=(1,j-(t-1))$ and $\phi_Q^{-1}(\zeta,0)=(1,s+(i-1))$.
Then Lemma~\ref{lem: first and last A} tells that $c-b>1$. Thus $\supp(\al) \cap \supp(\be)= \emptyset$ and $\al+\be \not \in \PR$, which implies $\soc_Q(\al,\be)=(\al,\be)$.
\item[{\rm (a-2)}] If $(t-s)-(i+j)+2=2$, then we have $c-b=1$ by Lemma~\ref{lem: first and last A} and hence $\al+\be \in \PR$. Thus $\soc_Q(\al,\be)=(\al+\be)$ by Proposition~\ref{prop: Q-socle of pair with sum in PR}
\end{itemize}

By applying the similar strategy of {\rm (a)}, the case {\rm (b)} can be proved.

\medskip

\noindent
{\rm (c)} In this case, Theorem~\ref{Thm: sectional A} tells that
$\phi_Q(k,u)=[c,b]$ and  $\phi_Q(k',u')=[a,b]$ as {\rm (c)} in~\eqref{eq: A complicated socle}.

Note that $\al+\be=\ve_a-\ve_b+\ve_c-\ve_d \not\in \PR$. If $([a,d],[c,b]) \ne \um \prec_Q^\tb (\al,\be)$ exists, $\um$ should contain positive roots in sectional paths in {\rm (c)} of~\eqref{eq: A complicated socle}.
But there is no such $\um$ by Theorem~\ref{Thm: sectional A}. Hence we have
\[ \soc_Q(\al,\be)=([a,d],[c,b]). \qedhere \]
\end{proof}

\begin{proof}[Proof of {\rm Proposition~\ref{prop: Q-socle for pair}} for $D_n$-types] Let $\phi_Q^{-1}(\be,0)=(i,s)$ and $\phi_Q^{-1}(\al,0)=(j,t)$. We assume that $j\le i$.
Recall that we have assumed that $\al \prec_Q \be$ is not sectional.

\noindent
{\rm(The case when $1 \le j \le i < n-1$)}
A path between them can be drawn as one of the following forms:
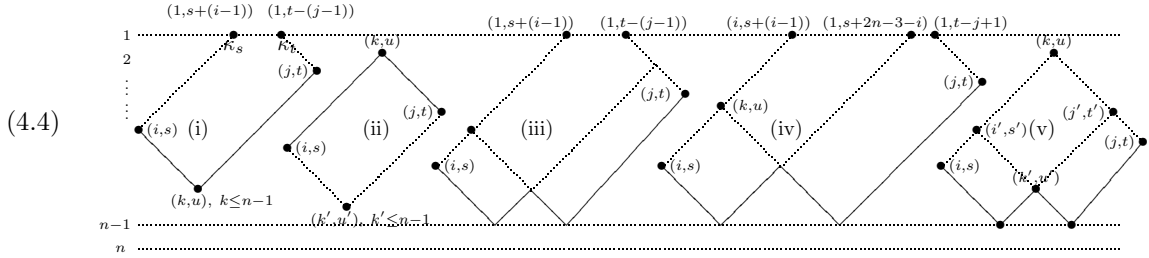
\begin{align} \label{eq: path D}
\scalebox{0.79}{{\xy
(-20,0)*{}="DL";(-10,-10)*{}="DD";(0,20)*{}="DT";(10,10)*{}="DR";
"DT"+(-30,-4); "DT"+(135,-4)**\dir{.};
"DD"+(-20,-6); "DD"+(145,-6) **\dir{.};
"DD"+(-20,-10); "DD"+(145,-10) **\dir{.};
"DT"+(-32,-4)*{\scriptstyle 1};
"DT"+(-32,-8)*{\scriptstyle 2};
"DT"+(-32,-12)*{\scriptstyle \vdots};
"DT"+(-32,-16)*{\scriptstyle \vdots};
"DT"+(-34,-36)*{\scriptstyle n-1};
"DT"+(-33,-40)*{\scriptstyle n};
"DL"+(-10,0); "DD"+(-10,0) **\dir{-};"DL"+(-10,0); "DT"+(-14,-4) **\dir{.};
"DT"+(-6,-4); "DR"+(-10,0) **\dir{.};"DR"+(-10,0); "DD"+(-10,0) **\dir{-};
"DL"+(-6,0)*{\scriptstyle (i,s)};
"DL"+(0,0)*{{\rm(i)}};
"DR"+(-14,0)*{\scriptstyle (j,t)};
"DL"+(-10,0)*{\bullet};"DR"+(-10,0)*{\bullet};
"DT"+(-6,-4)*{\bullet};"DT"+(-14,-4)*{\bullet};"DD"+(-10,0)*{\bullet};
"DD"+(-6,-2)*{\scriptstyle (k,u), \ k \le n-1};
"DT"+(-18,0)*{\scriptstyle (1,s+(i-1))};"DT"+(-1,0)*{\scriptstyle (1,t-(j-1))};
"DT"+(-14,-6)*{\kappa_s}; "DT"+(-5,-6)*{\kappa_t};
"DL"+(15,-3); "DD"+(15,-3) **\dir{.};
"DR"+(11,-7); "DD"+(15,-3) **\dir{.};
"DT"+(11,-7); "DR"+(11,-7) **\dir{-};
"DT"+(11,-7); "DL"+(15,-3) **\dir{-};
"DL"+(19,-3)*{\scriptstyle (i,s)};
"DR"+(7,-7)*{\scriptstyle (j,t)};
"DL"+(30,0)*{{\rm(ii)}};
"DL"+(15,-3)*{\bullet};"DR"+(11,-7)*{\bullet};
"DT"+(11,-7)*{\bullet}; "DD"+(15,-3)*{\bullet};
"DT"+(11,-5)*{\scriptstyle (k,u)};
"DD"+(19,-5)*{\scriptstyle (k',u'), \ k'\le n-1};
"DD"+(30,4); "DD"+(40,-6) **\dir{-};
"DD"+(46,0); "DD"+(40,-6) **\dir{-};
"DD"+(46,0); "DD"+(52,-6) **\dir{-};
"DD"+(72,16); "DD"+(52,-6) **\dir{-};
"DD"+(30,4); "DD"+(52,26) **\dir{.};
"DD"+(72,16); "DD"+(62,26) **\dir{.};
"DD"+(46,0); "DD"+(67,21) **\dir{.};
"DD"+(46,0); "DD"+(36,10) **\dir{.};
"DD"+(36,10)*{\bullet};
"DD"+(52,26)*{\bullet};
"DD"+(46,28)*{\scriptstyle (1,s+(i-1))};
"DD"+(62,26)*{\bullet};
"DD"+(65,28)*{\scriptstyle (1,t-(j-1))};
"DD"+(30,4)*{\bullet};
"DL"+(57,0)*{{\rm(iii)}};
"DD"+(34,4)*{\scriptstyle (i,s)};
"DD"+(72,16)*{\bullet};
"DD"+(67,16)*{\scriptstyle (j,t)};
"DD"+(68,4); "DD"+(78,-6) **\dir{-};
"DD"+(68,4); "DD"+(90,26) **\dir{.};
"DD"+(88,4); "DD"+(78,-6) **\dir{-};
"DD"+(88,4); "DD"+(110,26) **\dir{.};
"DD"+(122,18); "DD"+(114,26) **\dir{.};
"DD"+(122,18); "DD"+(98,-6) **\dir{-};
"DD"+(88,4); "DD"+(98,-6) **\dir{-};
"DD"+(122,18)*{\bullet};
"DD"+(118,18)*{\scriptstyle (j,t)};
"DD"+(68,4)*{\bullet};
"DD"+(72,4)*{\scriptstyle (i,s)};
"DD"+(90,26)*{\bullet};
"DD"+(86,28)*{\scriptstyle (i,s+(i-1))};
"DD"+(104,28)*{\scriptstyle (1,s+2n-3-i)};
"DD"+(114,26)*{\bullet};
"DD"+(120,28)*{\scriptstyle (1,t-j+1)};
"DD"+(110,26)*{\bullet};
"DD"+(78,14)*{\bullet};
"DD"+(83,14)*{\scriptstyle (k,u)};
"DL"+(99,0)*{{\rm(iv)}};
"DD"+(88,4); "DD"+(78,14) **\dir{.};
"DD"+(115,4); "DD"+(125,-6) **\dir{-};
"DD"+(131,0); "DD"+(125,-6) **\dir{-};
"DD"+(131,0); "DD"+(137,-6) **\dir{-};
"DD"+(149,8); "DD"+(137,-6) **\dir{-};
"DD"+(115,4); "DD"+(134,23) **\dir{.};
"DD"+(149,8); "DD"+(134,23) **\dir{.};
"DD"+(131,0); "DD"+(144,13) **\dir{.};
"DD"+(131,0); "DD"+(121,10) **\dir{.};
"DD"+(121,10)*{\bullet};
"DD"+(126,10)*{\scriptstyle (i',s')};
"DD"+(115,4)*{\bullet};
"DD"+(125,-6)*{\bullet};
"DD"+(134,23)*{\bullet};
"DD"+(144,13)*{\bullet};
"DD"+(139,13)*{\scriptstyle (j',t')};
"DD"+(137,-6)*{\bullet};
"DD"+(131,0)*{\bullet};
"DD"+(131,2)*{\scriptstyle (k',u')};
"DL"+(142,0)*{{\rm(v)}};
"DD"+(119,4)*{\scriptstyle (i,s)};
"DD"+(149,8)*{\bullet};
"DD"+(145,8)*{\scriptstyle (j,t)};
"DD"+(134,25)*{\scriptstyle (k,u)};
\endxy}}
\end{align}

We write $\al=\ve_a \pm \ve_b$, $\be=\ve_c \pm \ve_d$ and assume that $\be$ is in the $c$-swing.
Recall the index $\mathtt{l}$ in Lemma~\ref{lem: first} and the set $\kp=\{\kp_1,\ldots,\kp_{n-1} \}$ in~\eqref{eq: kappa}.

\medskip

Now we shall prove our assertion for $Q$ of type $D_n$ with respect to each shape of the paths in~\eqref{eq: path D}.
For the cases {\rm (i)} and {\rm (ii)}, we can assume that,
\begin{align} \label{eq: assumption 1}
\text{$\al$ is located at the $N$-part of $1$-swing; i.e., $\al=\lf 1, b \rf$,}
\end{align}
by Theorem~\ref{Thm: V-swing} and the fact that $\cox{Q}^k(\al)$ is contained in the $N$-part of $1$-swing for some $k \in \Z$.

\noindent
{\rm (i)} By Lemma~\ref{lem: nfree position D}, the pair $(\al,\be)$ can not be a pair with $\mul(\al)=\mul(\be)=1$.
We write $\kp_t$ for $\phi_Q^{-1}(\kp_t,0)=(1,t-(j-1))$ and $\kp_s$ for $\phi_Q^{-1}(\kp_s,0)=(1,i+(s-1))$.

\medskip

\noindent
{\rm (i-1: $(t-s)-(i+j)+2>2$)} We first assume that $s >  \mathtt{l}$. Since $s >  \mathtt{l}$, Lemma~\ref{lem: nfree position D}
and Lemma~\ref{lem: first} imply that
\begin{itemize}
\item $\mul(\al)=2$, $\mul(\be)=1$ and $\be$ is contained in a shallow sectional path by Theorem~\ref{thm: short path}.
\end{itemize}
Equivalently, $\al=\lf 1, b\rf$, $\be=\lf c, -d \rf$ where $1<b \le n-2$, $d \in I \setminus \{ \ta \}$ with $d>c$ and $\beta$ is located at the $S$-part of $c$-swing.

By Lemma~\ref{lem: reverse uni}, we can conclude that $b \ge c>0$. The multiplicity free root $\kp_s$ $(s>t+1)$ is of the form $\lf y ,-d \rf$. The assumption $(t-s)-(i+j)+2>2$ tells that
 $\kp_s+\kp_t \not \in \PR$ and hence  $d<b$ by~\cite[Corollary 1.15, Corollary 1.26]{Oh14D}.

Thus $\al+\be = \ve_1+\ve_b+\ve_c-\ve_d$, where $d>b>c \ge 1$. If $c>1$,~\cite[Theorem 1.19, Theorem 1.22]{Oh14D}
tells that there exist paths from $\lf 1 , -d \rf$ to $\beta$ and from $\al$ to $\lf b , c \rf$. Thus we have
$$(\al,\be) \prec_Q^\tb (\lf c , b\rf, \lf 1, -d \rf).$$
If $c=1$, then $\al+\be=2\ve_1+\ve_b-\ve_d$.

One can check that there exists no $[Q]$-simple sequence
$\us \ne (\al,\be)$ such that $\us \prec_Q^\tb (\al,\be)$ by the positive root system of $D_n$ and the convexity of $\prec_Q$.
Thus we have $$\soc_Q(\al,\be)=(\al,\be).$$

Applying the similar argument when $s \ge \mathtt{l}$, one can check that $\soc_Q(\al,\be)=(\al,\be)$.

\medskip

\noindent
{\rm (i-2: $(t-s)-(i+j)+2=2$)} Since $\kp_s+\kp_t \in \PR$, we have
\begin{equation}
\al+\be =
\begin{cases}
2\ve_1  \not\in \PR &\text{ if }  c=1, \\
\ve_1+\ve_c \in \PR &\text{ if } c > 1.
\end{cases}
\end{equation}

By Proposition~\ref{prop: Q-socle of pair with sum in PR}, it suffice to assume that $c=1$.
Then the pair $(\lf 1 , \ta\rf,\lf 1 , -\ta\rf)$ is only the $[Q]$-simple pair among the pairs $\us$ such that $\wt(\us)=2\ve_1$ by Theorem~\ref{Thm: V-swing}.
More precisely, if $\us \ne (\lf 1 , \ta\rf,\lf 1 , -\ta\rf)$, then it must be of the form $(\lf 1 , k \rf,\lf 1 , -k \rf)$ for $k \in I \setminus \{ \ta \}$.
But $(\lf 1 , \ta\rf,\lf 1 , -\ta\rf) \prec^\tb_Q (\lf 1 , k \rf,\lf 1 , -k \rf)$ by Theorem~\ref{Thm: V-swing}. If there exists a sequence $\us$ such that
$|\us|>2$ and $\wt(\us)=2\ve_1$, then one can check that $\us$ can not be $[Q]$-simple. Thus we have (see (4) in~\eqref{eq: complacted socle of D} below)
\[
\soc_Q(\al,\be)=(\lf 1 , \ta\rf,\lf 1 , -\ta\rf).
\]

%\noindent
%The other pair $(\al',\be')$ of the form {\rm (i)} can be obtained by applying $\tau_Q$ or $\tau_Q^{-1}$ proper times
%to $(\al,\be)$ we already dealt with. Thus we proved.

\medskip

\noindent
{\rm (ii)} We write $\zeta$ for $\phi_Q^{-1}(\zeta,0)=(k,u)$. Recall the assumption~\eqref{eq: assumption 1}.

\medskip

{\rm (ii-1: $k'<n-1$)} In this case, Theorem~\ref{Thm: V-swing} and Theorem~\ref{thm: short path} tell that
$$ \al+\be = \ve_1+\ve_b+\ve_c+\ve_d = \eta+\zeta,$$
where $b,c,d$ are distinct, $|b|,|c|,|d|>1$ and $\eta \in \PR$ for  $\phi_Q^{-1}(\eta,0)=(k',u')$. Note that the pair
$(\eta,\zeta)$ is $[Q]$-simple. Then there are three positive roots having $\ve_1$ as
its summand and one of $\ve_b$, $\ve_c$ and $\ve_d$ as its another summand. Thus there are at most three pairs whose weights are the same as
$\al+\be$. Assume that If there exists an another pair $(\eta',\zeta')$ such that $\eta'+\zeta'=\al+\be$,
$\zeta'$ must be located at the intersection of $S$-part of $1$-swing and the $N$-path of $\zeta$. Thus
there exists a path from $\zeta'$ to $\be$. Similarly, there exists a path from $\al$ to $\eta'$. Thus we have
$$(\al,\be) \prec^\tb_Q (\eta',\zeta').$$
Thus we can prove that (see (1) in~\eqref{eq: complacted socle of D} below)
$$\soc_Q(\al,\be)=(\eta,\zeta).$$

{\rm (ii-2: $k'=n-1$)} In this case, we have
$$ \al+\be = 2\ve_1+\ve_b+\ve_d = \eta+\eta'+\zeta,$$
where $b,d$ are distinct, $|b|,|d|>1$, $\eta \in \PR$ for  $\phi_Q^{-1}(\eta,0)=(n-1,u')$ and  $\eta' \in \PR$ for  $\phi_Q^{-1}(\eta',0)=(n,u')$ such that
$\{ \eta,\eta' \} = \{ \lf 1 , \ta \rf, \lf 1 , -\ta \rf \}$. Then one can check that (see (5) in~\eqref{eq: complacted socle of D} below)
$$\soc_Q(\al,\be)=(\eta,\eta',\zeta).$$

\medskip

The other pair $(\al',\be')$ of form  {\rm (i)} or {\rm (ii)} can be obtained by applying $\tau_Q$ or $\tau_Q^{-1}$ proper times
to $(\al,\be)$ we already dealt with. Thus we proved.

\medskip

For the cases {\rm (iii)}, {\rm (iv)} and {\rm (v)}, we can assume that $\be$ is contained in the $S$-part of $c$-swing by
Theorem~\ref{Thm: V-swing} and Lemma~\ref{lem: reverse uni}.

\medskip

\noindent
{\rm (iii)} Write the positive root as $\zeta$ located in the intersection of the swing containing $\al$ and the $N$-path of $\be$, and
the positive root as $\eta$ located in the intersection of $c$-swing and the $S$-path of $\al$. Then
Theorem~\ref{Thm: V-swing} and Theorem~\ref{thm: short path} tell that
$$ \al+\be=\eta+\zeta \text{ and } (\eta,\zeta) \prec_Q^\tb (\al,\be).$$

By {\rm (i)}, we have
$$
\soc_Q(\al,\be)=\begin{cases}
(\eta,\zeta) & \text{ if } (t-s)-(i+j)+2>2, \\
(\al+\be) \in \PR & \text{ if } (t-s)-(i+j)+2=2.
\end{cases}
$$

\noindent
{\rm (iv)} We write $\kp_s$ for $\phi_Q^{-1}(\kp_s,0)=(1,j-(t-1))$, $\kp_{s'}$ for $\phi_Q^{-1}(\kp_{s'},0)=(1,s+2n-3-i)$ and
$\kp_t$ for $\phi_Q^{-1}(\kp_t,0)=(1,i+(s-1))$ $(s<s'<t)$. By Lemma~\ref{lem: first}, we have $$s<s'< \mathtt{l} \le t.$$
More precisely,
$\kp_s=\lf x_1 ,b \rf$, $\kp_s=\lf c, y_2 \rf$ and $\kp_t=\lf x_3 , d \rf$
where $ 0<-b  \le c < |d|$ for $d \in -I \sqcup \{  \ta' \}$.

\medskip

{\rm (iv-1: $s'>s+1$)} Since $\kp_s+\kp_s' \not \in \PR$, we have $ 0<-b  < c < |d|$.
Thus $\supp(\al) \cap \supp(\be) = \emptyset$ and $\al +\be \not \in \PR$. Thus $\soc_Q(\al,\be)=(\al,\be)$.

\medskip

{\rm (iv-2: $s'=s+1$)} In this case, we have $\kp_s+\kp_s' \in \PR$. Thus $-b=c$. Hence $\al+\be$ is a multiplicity free positive root and $\soc_Q(\al,\be)=(\al+\be)$.
Moreover, one can check that $\phi_Q^{-1}(\al+\be,0)=(k,u)$.

\medskip

\noindent
{\rm (v)} We write $\al'$ for $\phi_Q^{-1}(\al',0)=(i',s')$, $\be'$ for $\phi_Q^{-1}(\be',0)=(j',t')$,
$\eta$ for $\phi_Q^{-1}(\eta,0)=(k,u)$ and $\zeta$ for $\phi_Q^{-1}(\zeta,0)=(k',u')$.
Then by Theorem~\ref{Thm: V-swing} and Theorem~\ref{thm: short path}, we have
$$ \al+\be=\al'+\be'=\eta+\zeta.$$

By {\rm (ii)}, we have (see (3) in~\eqref{eq: complacted socle of D} below)
$$\soc_Q(\al,\be)=\soc_Q(\al',\be')=(\eta,\zeta).$$

\medskip

The other pair $(\al',\be')$ of the form {\rm (iii)}, {\rm (iv)} or {\rm (v)} can be obtained by applying $\tau_Q$ or $\tau_Q^{-1}$ proper times
to $(\al,\be)$ we already dealt with. Thus we proved.

\medskip

{\rm (The case when $1 \le j < n-1$ and $i \in \{ n-1, n \}$)} A path between them can be drawn as one of the following forms:
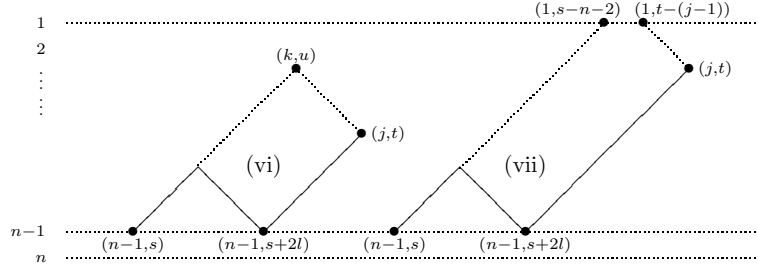
\[
\scalebox{0.87}{{\xy
(-20,0)*{}="DL";(-10,-10)*{}="DD";(0,20)*{}="DT";(10,10)*{}="DR";
"DT"+(-30,-4); "DT"+(75,-4)**\dir{.};
"DD"+(-20,-6); "DD"+(85,-6) **\dir{.};
"DD"+(-20,-10); "DD"+(85,-10) **\dir{.};
"DT"+(-34,-4)*{\scriptstyle 1};
"DT"+(-34,-8)*{\scriptstyle 2};
"DT"+(-34,-12)*{\scriptstyle \vdots};
"DT"+(-34,-16)*{\scriptstyle \vdots};
"DT"+(-36,-36)*{\scriptstyle n-1};
"DT"+(-34,-40)*{\scriptstyle n};
"DD"+(0,4); "DD"+(10,-6) **\dir{-};
"DD"+(0,4); "DD"+(-10,-6) **\dir{-};
"DD"+(0,4); "DD"+(15,19) **\dir{.};
"DD"+(25,9);"DD"+(10,-6) **\dir{-};
"DD"+(25,9);"DD"+(15,19) **\dir{.};
"DD"+(25,9)*{\bullet};
"DD"+(15,19)*{\bullet};
"DD"+(10,-6)*{\bullet};
"DD"+(10,-8)*{\scriptstyle (n-1,s+2l)};
"DD"+(15,21)*{\scriptstyle (k,u)};
"DD"+(29,9)*{\scriptstyle (j,t)};
"DD"+(10,4)*{{\rm (vi)}};
"DD"+(-10,-6)*{\bullet};
"DD"+(-10,-8)*{\scriptstyle (n-1,s)};
"DD"+(40,4); "DD"+(50,-6) **\dir{-};
"DD"+(40,4); "DD"+(30,-6) **\dir{-};
"DD"+(40,4); "DD"+(62,26) **\dir{.};
"DD"+(75,19);"DD"+(50,-6) **\dir{-};
"DD"+(75,19);"DD"+(68,26) **\dir{.};
"DD"+(75,19)*{\bullet};
"DD"+(50,-6)*{\bullet};
"DD"+(50,-8)*{\scriptstyle (n-1,s+2l)};
"DD"+(50,4)*{{\rm (vii)}};
"DD"+(79,19)*{\scriptstyle (j,t)};
"DD"+(30,-6)*{\bullet};
"DD"+(30,-8)*{\scriptstyle (n-1,s)};
"DD"+(68,26)*{\bullet};
"DD"+(74,28)*{\scriptstyle (1,t-(j-1))};
"DD"+(62,26)*{\bullet};
"DD"+(58,28)*{\scriptstyle (1,s-n-2)};
\endxy}}
\]
Write $\phi_Q^{-1}(\be,0)=(n-\delta,s)$ for $\{ \delta, \delta' \} = \{ 0,1 \}$.

{\rm (vi)}  In this case, Lemma~\ref{lem: n-1,n D} and Theorem~\ref{Thm: V-swing} tell that
$\al+\be = \eta+\zeta $ where $\phi_Q^{-1}(\eta)=(k,u)$ and
\begin{align} \label{eq: delta determination} \phi_Q^{-1}(\zeta)=
\begin{cases}
(n-\delta,s+2l) & \text{ if $l$ is even},\\
 (n-\delta',s+2l) & \text{ if $l$ is odd}.
\end{cases}
\end{align}
Note that $(\eta,\zeta)$ is $[Q]$-simple and there is are three pairs $\up$ such that $\wt(\up)=\al+\be$ and two of them are $(\al,\be)$ and $(\eta,\zeta)$.
Moreover, we can check that $(\al,\be)$ and $(\eta',\zeta')$ are incomparable with respect to $\prec_Q^\tb$, where $(\eta',\zeta')$
is the another pair. Then one can check that (see (6) in~\eqref{eq: complacted socle of D} below)
$$\soc_Q(\al,\be)=(\eta,\zeta)$$
by applying the similar arguments of previous cases.

\medskip

{\rm (vii)} Applying the similar argument of {\rm (iv)} and using Lemma~\ref{lem: n-1,n D}, we have
$$
\soc_Q(\al,\be) =
\begin{cases}
(\al,\be) & \text{ if }  s-t-n-3>2, \\
(\al+\be) \in \PR &\text{ if }  s-t-n-3=2.
\end{cases}
$$
Here, if $s-t-n-3=2$, then we have
$$
\phi_Q^{-1}(\al+\be,0)=
 \begin{cases}
 (n-\delta,s+2l) & \text{ if $l$ is even}, \\
 (n-\delta',s+2l)  & \text{ if $l$ is odd}.
 \end{cases}
$$

\noindent
{\rm (The case when $i,j \in \{ n-1,n\}$)}
\[
\scalebox{0.72}{{\xy
(-20,0)*{}="DL";(-10,-10)*{}="DD";(0,20)*{}="DT";(10,10)*{}="DR";
"DT"+(-30,-4); "DT"+(40,-4)**\dir{.};
"DD"+(-20,-6); "DD"+(50,-6) **\dir{.};
"DD"+(-20,-10); "DD"+(50,-10) **\dir{.};
"DT"+(-34,-4)*{\scriptstyle 1};
"DT"+(-34,-8)*{\scriptstyle 2};
"DT"+(-34,-12)*{\scriptstyle \vdots};
"DT"+(-34,-16)*{\scriptstyle \vdots};
"DT"+(-36,-36)*{\scriptstyle n-1};
"DT"+(-34,-40)*{\scriptstyle n};
"DD"+(12,16); "DD"+(34,-6) **\dir{-};
"DD"+(12,16); "DD"+(-10,-6) **\dir{-};
"DD"+(34,-8)*{\scriptstyle (n-1,t)};
"DD"+(-10,-8)*{\scriptstyle (n-1,s)};
"DD"+(12,4)*{ {\rm (viii)}};
"DD"+(34,-6)*{\bullet};
"DD"+(-10,-6)*{\bullet};
"DD"+(12,16)*{\bullet};
"DD"+(12,18)*{\scriptstyle (k,u)};
\endxy}}
\]

{\rm (viii)} Applying~\cite[Proposition 1.14]{Oh14D} and Theorem~\ref{Thm: V-swing}, one can easily check that
$$
\soc_Q(\al,\be) = \begin{cases}
(\al+\be) & \text{ if } |i-j| \equiv n-k-1 ({\rm mod} \ 2),\\
(\al,\be) \in \PR & \text{ otherwise.}
\end{cases}
$$
Here if $|i-j| \equiv n-k-1 ({\rm mod} \ 2)$, then we have $\phi_Q^{-1}(\al+\be,0)=(k,u)$.

Now we record the socle $\us$ of non $[Q]$-simple pairs $(\al,\be)$ with $\al+\be \not\in \PR$ of type $D_n$, for convenience of reader and later use.
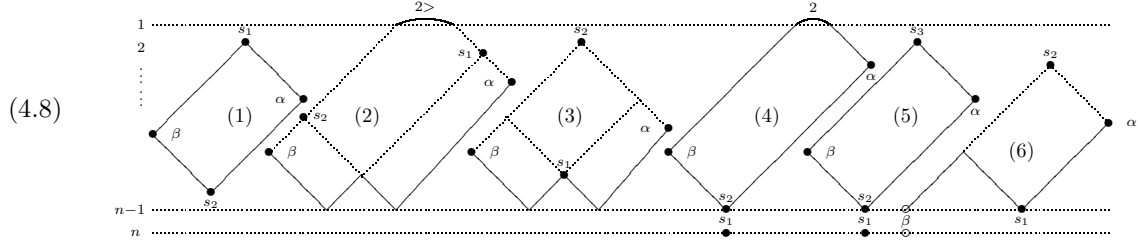
\begin{align} \label{eq: complacted socle of D}
\scalebox{0.77}{{\xy
(-20,0)*{}="DL";(-10,-10)*{}="DD";(0,20)*{}="DT";(10,10)*{}="DR";
"DT"+(-30,-4); "DT"+(135,-4)**\dir{.};
"DD"+(-20,-6); "DD"+(145,-6) **\dir{.};
"DD"+(-20,-10); "DD"+(145,-10) **\dir{.};
"DT"+(-32,-4)*{\scriptstyle 1};
"DT"+(-32,-8)*{\scriptstyle 2};
"DT"+(-32,-12)*{\scriptstyle \vdots};
"DT"+(-32,-16)*{\scriptstyle \vdots};
"DT"+(-34,-36)*{\scriptstyle n-1};
"DT"+(-33,-40)*{\scriptstyle n};
"DL"+(-10,-3); "DD"+(-10,-3) **\dir{-};
"DR"+(-14,-7); "DD"+(-10,-3) **\dir{-};
"DT"+(-14,-7); "DR"+(-14,-7) **\dir{-};
"DT"+(-14,-7); "DL"+(-10,-3) **\dir{-};
"DL"+(-6,-3)*{\scriptstyle \be};
"DR"+(-18,-7)*{\scriptstyle \al};
"DL"+(5,0)*{{\rm (1)}};
"DL"+(-10,-3)*{\bullet};"DR"+(-14,-7)*{\bullet};
"DT"+(-14,-7)*{\bullet}; "DD"+(-10,-3)*{\bullet};
"DT"+(-14,-5)*{\scriptstyle s_1};
"DD"+(-10,-5)*{\scriptstyle s_2};
"DD"+(0,4); "DD"+(10,-6) **\dir{-};
"DD"+(16,0); "DD"+(10,-6) **\dir{-};
"DD"+(16,0); "DD"+(22,-6) **\dir{-};
"DD"+(42,16); "DD"+(22,-6) **\dir{-};
"DD"+(0,4); "DD"+(22,26) **\dir{.};
"DD"+(42,16); "DD"+(32,26) **\dir{.};
"DD"+(16,0); "DD"+(37,21) **\dir{.};
"DD"+(16,0); "DD"+(6,10) **\dir{.};
"DD"+(37,21)*{\bullet};
"DD"+(34,21)*{\scriptstyle s_1};
"DD"+(6,10)*{\bullet};
"DD"+(9,10)*{\scriptstyle s_2};
%"DD"+(22,26)*{\bullet};
%"DD"+(16,28)*{\scriptstyle (1,s+(i-1))};
%"DD"+(32,26)*{\bullet};
%"DD"+(35,28)*{\scriptstyle (1,t-(j-1))};
"DD"+(22,26); "DD"+(32,26) **\crv{"DD"+(27,28)};
"DD"+(27,29)*{\scriptstyle 2 >};
"DD"+(0,4)*{\bullet};
"DL"+(27,0)*{{\rm (2)}};
"DD"+(4,4)*{\scriptstyle \be};
"DD"+(42,16)*{\bullet};
"DD"+(38,16)*{\scriptstyle \al};
"DD"+(35,4); "DD"+(45,-6) **\dir{-};
"DD"+(51,0); "DD"+(45,-6) **\dir{-};
"DD"+(51,0); "DD"+(57,-6) **\dir{-};
"DD"+(69,8); "DD"+(57,-6) **\dir{-};
"DD"+(35,4); "DD"+(54,23) **\dir{.};
"DD"+(69,8); "DD"+(54,23) **\dir{.};
"DD"+(51,0); "DD"+(64,13) **\dir{.};
"DD"+(51,0); "DD"+(41,10) **\dir{.};
"DD"+(35,4)*{\bullet};
"DD"+(54,23)*{\bullet};
%"DD"+(37,-6)*{\bullet};
"DD"+(51,0)*{\bullet};
"DD"+(51,2)*{\scriptstyle s_1};
"DL"+(62,0)*{{\rm(3)}};
"DD"+(39,4)*{\scriptstyle \be};
"DD"+(69,8)*{\bullet};
"DD"+(65,8)*{\scriptstyle \al};
"DD"+(54,25)*{\scriptstyle s_2};
"DD"+(69,4); "DD"+(79,-6) **\dir{-};
"DD"+(85,0); "DD"+(79,-6) **\dir{-};
"DD"+(69,4); "DD"+(91,26) **\dir{-};
"DD"+(104,19); "DD"+(97,26) **\dir{-};
"DD"+(85,0); "DD"+(104,19) **\dir{-};
"DD"+(69,4)*{\bullet};
%"DD"+(74,23)*{\bullet};
"DL"+(96,0)*{{\rm (4)}};
"DD"+(73,4)*{\scriptstyle \be};
"DD"+(91,26); "DD"+(97,26) **\crv{"DD"+(94,28)};
"DD"+(104,19)*{\bullet};
"DD"+(104,17)*{\scriptstyle \al};
"DD"+(79,-10)*{\bullet};
"DD"+(79,-8)*{\scriptstyle s_1};
"DD"+(79,-6)*{\bullet};
"DD"+(79,-4)*{\scriptstyle s_2};
"DD"+(94,29)*{\scriptstyle 2};
"DD"+(93,4); "DD"+(103,-6) **\dir{-};
"DD"+(109,0); "DD"+(103,-6) **\dir{-};
"DD"+(93,4); "DD"+(112,23) **\dir{-};
"DD"+(122,13); "DD"+(112,23) **\dir{-};
"DD"+(109,0); "DD"+(122,13) **\dir{-};
"DD"+(93,4)*{\bullet};
"DD"+(112,23)*{\bullet};
"DL"+(120,0)*{{\rm (5)}};
"DD"+(97,4)*{\scriptstyle \be};
"DD"+(112,25)*{\scriptstyle s_3};
"DD"+(122,13)*{\bullet};
"DD"+(122,11)*{\scriptstyle \al};
"DD"+(103,-10)*{\bullet};
"DD"+(103,-8)*{\scriptstyle s_1};
"DD"+(103,-6)*{\bullet};
"DD"+(103,-4)*{\scriptstyle s_2};
"DD"+(120,4); "DD"+(130,-6) **\dir{-};
"DD"+(120,4); "DD"+(110,-6) **\dir{-};
"DD"+(120,4); "DD"+(135,19) **\dir{.};
"DD"+(145,9);"DD"+(130,-6) **\dir{-};
"DD"+(145,9);"DD"+(135,19) **\dir{.};
"DD"+(145,9)*{\bullet};
"DD"+(135,19)*{\bullet};
"DD"+(130,-6)*{\bullet};
"DD"+(130,-8)*{\scriptstyle s_1};
"DD"+(135,21)*{\scriptstyle s_2};
"DD"+(149,9)*{\scriptstyle \al};
"DD"+(130,4)*{{\rm (6)}};
"DD"+(110,-6)*{\circ};
"DD"+(110,-8)*{\scriptstyle \be};
"DD"+(110,-10)*{\circ};
\endxy}}
\end{align}
For the case $(6)$, $\be$ is the one of $\circ$'s which is determined by~\eqref{eq: delta determination}.

\medskip

For the case when $j>i$, we can prove by the similar arguments.
\end{proof}

In the course of proof of the above proposition, one can notice that the following property holds for type $A_n$ and $D_n$:

\begin{corollary} \label{cor: dist1 minimal of soc AD}
For a pair $\up=(\al,\be)$ of type $A_n$ or $D_n$ with $\dist_Q(\al,\be)=1$, $\up$ is a $[Q]$-minimal pair of $\soc_Q(\up)$.
\end{corollary}

\begin{example} \label{ex: dist1 but not minimal}
In the $Q$ of type $E_6$ in Appendix~\ref{Sec:Dynkin E6}, one can check that the pair $$\up=(111001,123212)$$ has $\dist_Q(\up)=1$ with its socle $\us=(001001,122101,111111)$. But
$\up$ is not a $[Q]$-minimal sequence of $\us$. Actually, $$\um=(111101,122111,001001) \quad\text{ and }\quad \um'=(011001,112101,111111)$$ are less than $\up$ with respect to
$\prec^\tb_Q$ and $[Q]$-minimal sequences of $\us$.
\end{example}

Now we shall generalize the notion of $[\redez]$-distance by {\it releasing the pair condition} in Remark~\ref{rem:p-dist}:
The \defn{generalized $[\tw]$-distance} of a sequence $\um$ is the largest integer $k \geq 0$ such that
\[
\um^{(0)} \prec^\tb_{[\tw]} \cdots \prec^\tb_{[\tw]} \um^{(k)} = \um
\]
and $\um^{(0)}$ is $[\tw]$-simple. We denoted the integer by $\gdist_{[\tw]}(\um)$.

\begin{remark} \label{rem: distQ gdistQ AD}
From the proof of Proposition~\ref{prop: Q-socle of pair with sum in PR},
we have $$\dist_Q(\al,\be)=\gdist_Q(\al,\be) \text{ for any quiver $Q$ and any pair $(\al,\be)$ of type $A_n$ or $D_n$.}$$
However, the pair $\up$ in Example
\ref{ex: dist1 but not minimal} for $E_6$-case has its $[Q]$-distance $1$ while $\gdist_Q(\up)=2$.
\end{remark}

\begin{lemma} \label{lem: general convexity}
Assume the followings$\colon$ For $w \in \W$, we have
\begin{itemize}
\item a reduced expression $\tw$ of $w \in \W$,
\item a pair $(\alpha,\beta) \in (\PR_w)^2 \subset (\PR)^2$ with $\al \prec_{[\tw]} \be$,
\item a sequence $\us$ such that  $\us \prec^\tb_{[\redez]}  (\alpha,\beta)$ for some $\redez$ and $\tw'=w' \in \W$ such that $ww'=w_0$ and $\redez=\tw * \tw'$.
\end{itemize}
Then, for any $i$ with $s_i \ne 0$, $\beta_i$ is contained in $\PR_w$.
\end{lemma}

\begin{proof}
 By the assumptions,
it suffices to prove the following argument:
$$ \text{ If } \al \prec_{[\redez]} \beta_i \prec_{[\redez]} \be \text{ and } \al \prec_{[\tw]} \be, \text{ then we have } \be_i \in \PR_w.$$
If $\be_i \not \in \PR_w$, then $\be=\be^\redez_k$ and $\be_i=\be^\redez_l$ such that $k \le \ell(w) < l$; i.e., $\be <_\redez \be_i$, which yields a contradiction to the fact that
\[  \be_i \prec_{[\redez]} \be  \text{ if and only if } \be_i <_{\redez'} \be  \text{ for all } \redez' \in [\redez]. \qedhere \]
\end{proof}

\begin{proof}[Proof of {\rm Theorem~\ref{thm: socle of pair}}]
By Proposition~\ref{prop: Q-socle for pair}, the $\soc_Q(\al,\be)$ is well-defined for all pairs $(\al,\be)$ and Dynkin quivers $Q$. Then our assertion follows from
Lemma~\ref{lem: general convexity}.
\end{proof}

%Now we can define the notion of $\LF Q \RF$-family: % which was used in Section~\ref{subsec: new notion} without its definition.

%\begin{definition} \label{def: Q-family}
%For $\tw$ of $w \in \W$, we say that $\tw$ is contained in the $\LF Q \RF$-family, if there exists a Dynkin quiver $Q'$ such that $\tw$ is adapted to $Q'$.
%\end{definition}

\begin{remark} \label{rem: socle}
Before we close this subsection, we record the property of $[Q]$-socle of a pair $(\al,\be)$ in this remark: For a given non $[Q]$-simple pair $\up=(\al,\be)$,
$$\text{the $[Q]$-socle $\us$ of $\up$ is  a sequence such that $|\us| \le 3$ and $\us_i \le 1$ for all $1 \le i \le \N$.}$$
\end{remark}

\begin{example}
For a pair $\up=(11111100,12233321)$ and the Dynkin quiver $Q$ of type $E_8$ in Appendix~\ref{Sec:Dynkin E8}, the socle of $\up$ is given as follows:
$$ \soc_Q(\up)=(11111111,01111100,01011100).$$
\end{example}

\subsection{ $[Q]$-radius and multiplicity of $\ga \in \PR \setminus \Pi$} In this subsection, we study
the relationship between $[Q]$-radius of $(\al,\be)$ and multiplicity of $\al+\be$ when $\al+\be \in \PR$.

\begin{theorem}~\cite[Theorem 3.4]{Oh14A},~\cite[Theorem 3.13, Theorem 3.17]{Oh14D} \label{thm: radius 1,2 AD}
\begin{enumerate}
\item[{\rm (a)}] Let $\ga$ be a non-simple positive root in $\PR$ of type $A_n$ or $D_n$ with $\mul(\ga)=1$. For any pair $(\al,\be)$ with $\al+\be =\ga$, we have
$\dist_Q(\al,\be)=1$. Hence we have
$$  \dist_Q(\al,\be) = \mul(\ga)=1.$$
\item[{\rm (b)}] Let $\ga$ be a non-simple positive root in $\PR$ of type $D_n$ with $\mul(\ga)=2$. Then there are pairs $(\al,\be)$ such that $\al+\be =\ga$ and
$(\al,\be)$ is not $[Q]$-minimal.
\end{enumerate}
\end{theorem}

\begin{theorem} \label{thm: rds mul}
For $\ga \in \PR \setminus \Pi$ and any Dynkin quiver $Q$ of type $A_n$, $D_n$ or $E_6$, we have
$$  \rds_Q(\ga) = \mul(\ga).$$
\end{theorem}

\begin{proof}
By Theorem~\ref{thm: radius 1,2 AD}, our assertion was already proved for  $\ga$ in $\PR\setminus \Pi$ of type $A_n$ and $D_n$ with $\mul(\ga)=1$.

Assume that $\ga \in \PR$ of type $D$ with $\mul(\ga)=2$. By Theorem~\ref{Thm: V-swing}, $\ga= \lf a , b \rf$ is located at the intersection of the $a$-swing and
the $b$-swing. Moreover,~\cite[Theorem 3.17]{Oh14D} and Proposition~\ref{prop: Q-socle for pair} ($D$ \rm{(iii)}) imply that a non $[Q]$-minimal pair $(\al,\be)$ of $\ga$
happens in the following form in $\GQ$:
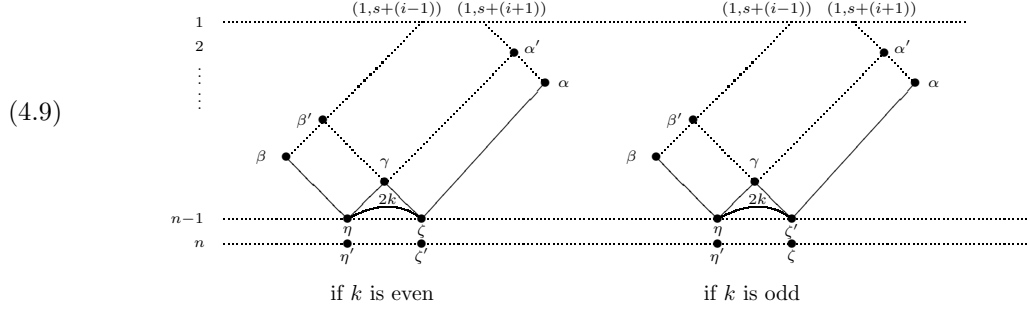
\begin{align} \label{eq: figure D dist 2}
\scalebox{0.82}{{\xy
(-20,0)*{}="DL";(-10,-10)*{}="DD";(0,20)*{}="DT";(10,10)*{}="DR";
"DT"+(-30,-4); "DT"+(90,-4)**\dir{.};
"DD"+(-20,-6); "DD"+(110,-6) **\dir{.};
"DD"+(-20,-10); "DD"+(110,-10) **\dir{.};
"DT"+(-34,-4)*{\scriptstyle 1};
"DT"+(-34,-8)*{\scriptstyle 2};
"DT"+(-34,-12)*{\scriptstyle \vdots};
"DT"+(-34,-16)*{\scriptstyle \vdots};
"DT"+(-36,-36)*{\scriptstyle n-1};
"DT"+(-34,-40)*{\scriptstyle n};
"DD"+(-10,4); "DD"+(0,-6) **\dir{-};
"DD"+(6,0); "DD"+(0,-6) **\dir{-};
"DD"+(6,0); "DD"+(12,-6) **\dir{-};
"DD"+(32,16); "DD"+(12,-6) **\dir{-};
"DD"+(-10,4); "DD"+(12,26) **\dir{.};
"DD"+(32,16); "DD"+(22,26) **\dir{.};
"DD"+(6,0); "DD"+(27,21) **\dir{.};
"DD"+(6,0); "DD"+(-4,10) **\dir{.};
"DD"+(-4,10)*{\bullet};
"DD"+(-7,10)*{\scriptstyle \be'};
"DD"+(6,0)*{\bullet};
"DD"+(6,3)*{\scriptstyle \ga};
%"DD"+(12,26)*{\bullet};
"DD"+(8,28)*{\scriptstyle (1,s+(i-1))};
%"DD"+(22,26)*{\bullet};
"DD"+(25,28)*{\scriptstyle (1,s+(i+1))};
"DD"+(-10,4)*{\bullet};
"DD"+(-14,4)*{\scriptstyle \be};
"DD"+(32,16)*{\bullet};
"DD"+(35,16)*{\scriptstyle \al};
"DD"+(27,21)*{\bullet};
"DD"+(30,22)*{\scriptstyle \al'};
"DD"+(0,-6)*{\bullet};
"DD"+(0,-8)*{\scriptstyle \eta};
"DD"+(12,-6)*{\bullet};
"DD"+(12,-8)*{\scriptstyle \zeta};
"DD"+(0,-10)*{\bullet};
"DD"+(0,-12)*{\scriptstyle \eta'};
"DD"+(12,-10)*{\bullet};
"DD"+(12,-12)*{\scriptstyle \zeta'};
"DD"+(0,-6);"DD"+(12,-6);**\crv{(-3,-12)}?(.5)+(0,1.3)*{\scriptstyle 2k};
"DD"+(5,-18)*{\text{ if $k$ is even} };
"DD"+(50,4); "DD"+(60,-6) **\dir{-};
"DD"+(66,0); "DD"+(60,-6) **\dir{-};
"DD"+(66,0); "DD"+(72,-6) **\dir{-};
"DD"+(92,16); "DD"+(72,-6) **\dir{-};
"DD"+(50,4); "DD"+(72,26) **\dir{.};
"DD"+(92,16); "DD"+(82,26) **\dir{.};
"DD"+(66,0); "DD"+(87,21) **\dir{.};
"DD"+(66,0); "DD"+(56,10) **\dir{.};
"DD"+(56,10)*{\bullet};
"DD"+(53,10)*{\scriptstyle \be'};
"DD"+(66,0)*{\bullet};
"DD"+(66,3)*{\scriptstyle \ga};
%"DD"+(12,26)*{\bullet};
"DD"+(68,28)*{\scriptstyle (1,s+(i-1))};
%"DD"+(22,26)*{\bullet};
"DD"+(85,28)*{\scriptstyle (1,s+(i+1))};
"DD"+(50,4)*{\bullet};
"DD"+(46,4)*{\scriptstyle \be};
"DD"+(92,16)*{\bullet};
"DD"+(95,16)*{\scriptstyle \al};
"DD"+(87,21)*{\bullet};
"DD"+(90,22)*{\scriptstyle \al'};
"DD"+(60,-6)*{\bullet};
"DD"+(60,-8)*{\scriptstyle \eta};
"DD"+(72,-6)*{\bullet};
"DD"+(72,-8)*{\scriptstyle \zeta'};
"DD"+(60,-10)*{\bullet};
"DD"+(60,-12)*{\scriptstyle \eta'};
"DD"+(72,-10)*{\bullet};
"DD"+(72,-12)*{\scriptstyle \zeta};
"DD"+(60,-6);"DD"+(72,-6);**\crv{(57,-12)}?(.5)+(0,1.3)*{\scriptstyle 2k};
"DD"+(65,-18)*{\text{ if $k$ is odd} };
\endxy}}
\end{align}
where $\phi_Q^{-1}(\be,0)=(i,s)$ and $\phi_Q^{-1}(\al,0)=(j,t)$. Moreover, Theorem~\ref{Thm: V-swing}, Theorem~\ref{thm: short path} and
\cite[Corollary 3.17]{Oh14D} imply that
\begin{align} \label{eq: minimal pair ga D}
\text{$(\al',\be')$, $(\eta,\zeta)$ and  $(\eta',\zeta')$ are $[Q]$-minimal pairs of $\ga$},
\end{align}
and hence they are incomparable with respect to $\prec_Q^\tb$. Thus we have
\begin{align} \label{eq: dist 2 of D}
\ga \prec_Q^\tb (\al',\be'),(\eta,\zeta),(\eta',\zeta') \prec_Q^\tb (\al,\be),
\end{align}
which implies that $\rds_Q(\ga) = \mul(\ga)=2.$
\end{proof}

\begin{example}
In Example~\ref{ex: example D}, we have
$$ \lf1,2\rf \prec_Q^\tb (\lf1,-4\rf,\lf2,4\rf),(\lf2,-3\rf,\lf1,3\rf),(\lf2,3\rf,\lf1,-3\rf)\prec_Q^\tb (\lf2,-4\rf,\lf1,4\rf)$$
corresponding to~\eqref{eq: dist 2 of D}.
\end{example}

In the previous theorem, our choice of $Q$ of type $A_n$, $D_n$ and $E_6$ is arbitrary. Thus the value $\rds_Q(\ga)$ for $\ga \in \PR_{A_n}$,
$\PR_{D_n}$ or $\PR_{E_6}$ do not depend on the choice of Dynkin quivers indeed:

\begin{corollary} For any reduced expressions $\redez$ and $\redez'$ of $w_0$
adapted to some Dynkin quivers $Q$ and $Q'$ of type $A_n$, $D_n$ or  $E_6$, we have
$$ \rds_{Q}(\ga)= \rds_{Q'}(\ga) = \mul{(\ga)}  \quad \text{ for all } \ga \in \PR \setminus \Pi.$$
\end{corollary}

\begin{theorem}
For $\ga \in \PR \setminus \Pi$ and any Dynkin quiver $Q$ of type $E_7$ and $E_8$, we have
$$  \mul(\ga)-1 \le \rds_Q(\ga) \le \mul(\ga)+1.$$
In particular,
\begin{itemize}
\item if $\mul(\ga)=1$, then $\rds_Q(\ga)=\mul(\ga)=1$,
\item if $\mul(\ga)>1$, then $\rds_Q(\ga)>1$.
\end{itemize}
\end{theorem}

\begin{example} \label{ex: rds not eq mul} \hfill
\begin{enumerate}
\item[{\rm (a)}] For the Dynkin quiver of type $E_7$ in Appendix~\ref{Sec:Dynkin E7}, the $[Q]$-radius of $\ga=(1122221)$ is $3$ since we have
$$ (1122221) \prec_Q^{\tb} (101110,0111111) \prec_Q^{\tb} (1122111,0000110) \prec_Q^{\tb} (1122110,0000111),$$
while $\mul(\ga)=2$.
\item[{\rm (b)}] For the Dynkin quiver of type $E_8$ in Appendix~\ref{Sec:Dynkin E8}, the $[Q]$-radius of $(23465432)$ is $5$ since
\begin{align*}
& (23465432) \prec_Q^{\tb} (12233221,11232211) \prec_Q^{\tb} (22344321,01121111) \\
& \qquad \prec_Q^{\tb} (22343321,01122111)\prec_Q^{\tb} (22343221,01122211)\prec_Q^{\tb} (23454321,00011111)
\end{align*}
is one of maximal chains for its radius.
\end{enumerate}
\end{example}

\begin{remark} \label{ex: non-trivial 1}
Consider the following reduced expression $\redez$ of $w_0$ of type $A_5$ which is not adapted to any Dynkin quiver $Q$:
$$ \redez= s_1s_2s_3s_5s_4s_3s_1s_2s_3s_5s_4s_3s_1s_2s_3.$$
The convex partial order $\prec_{[\redez]}$ can be visualized the following quiver $\Upsilon_{[\redez]}$ by~\cite{OS15}.
\begin{equation*} \label{modi 1 + dynkin}
\scalebox{0.77}{\xymatrix@C=3ex@R=1ex{
1&&&& [3,5]\ar@{->}[drr] &&  && [2]\ar@{->}[drr] && && [1]\\
2&& [3,4] \ar@{->}[dr]\ar@{->}[urr] &&  && [2,5] \ar@{->}[dr]\ar@{->}[urr] &&  && [1,2]\ar@{->}[urr]\\
3& [3] \ar@{->}[ur] && [4] \ar@{->}[dr] && [2,3] \ar@{->}[ur] && [4,5]\ar@{->}[dr]&& [1,3]\ar@{->}[ur]\\
4&&  && [2,4]\ar@{->}[ur]\ar@{->}[drr]  &&&& [1,5]\ar@{->}[ur]\ar@{->}[drr]\\
5&&&&  && [1,4]\ar@{->}[urr] &&  && [5]
}}
\end{equation*}
Then $\dist_{[\redez]}([1],[2,5])=\dist_{[\redez]}([1,2],[3,5])=2$ since
$$  ([1,5]) \prec_{[\redez]}^\tb  ([1,3],[4,5])  \prec_{[\redez]}^\tb \begin{matrix} ([1],[2,5]) \\ ([1,2],[3,5]) \end{matrix},$$
while $\mul([1,5])=1$.
Thus the above theorem does not hold for general $[\redez]$.
\end{remark}

\subsection{ $[Q]$-distances of pairs} %\sejin{generalization of length 2 property}
In this subsection, we investigate the $[Q]$-distances of pairs by observing the AR-quiver $\Gamma_Q$, which will be used for the composition length of
tensor product of two simple modules in the categories which we are interested in.

\begin{theorem} \label{thm: dist upper bound}
Let $\mathtt{m}= \max ( \mul(\ga) \ | \ \ga \in \PR \setminus \Pi )$. For any pair $(\al,\be)$, we have
$$ 0 \le \dist_Q(\al,\be) \le \mathtt{m}.$$
\end{theorem}

\begin{proof}
It suffices to consider a pair $(\al,\be)$ such that $\soc_Q(\al,\be) \ne (\al,\be)$.

\medskip

\noindent
($Q$ of type $A_n$) Note that $\mathtt{m}=1$ for $\PR$ of type $A_n$. If $\al+\be \in \PR$, we have
$\dist_Q(\al,\be)=1$  by~\cite[Theorem 3.4]{Oh14A}. Thus the remained pairs we have to deal with satisfy the following properties:
\begin{align}\label{eq: ccond}
\supp{\al} \cap \supp{\be} \ne \emptyset  \quad \text{and}\quad \al+\be \not \in \PR.
\end{align}
Equivalently $c\le b$ or $a \le d$ when we write $\al=[a,b]$ and $\be=[c,d]$. Then there exists only one pair $(\eta,\zeta)$ such that $\{ \eta,\zeta \} =\{ [a,d],[b,c]\}$
 and $\al+\be = \eta+\zeta$. By {\rm (c)} in Proposition~\ref{prop: Q-socle for pair}, $\dist_Q(\al,\be) \le 1$.

\medskip

\noindent
($Q$ of type $D$) Note that $\mathtt{m}=2$ for $\PR$ of type $D_n$. If $\al+\be \in \PR$, we have
$\dist_Q(\al,\be)\le 2$  by Theorem~\ref{thm: rds mul}. Thus the remained pairs satisfy the properties in~\eqref{eq: ccond}.
By recalling the cases in Proposition~\ref{prop: Q-socle for pair}, such pairs happen in the cases {\rm (D-ii)}, {\rm (D-v)}
or {\rm (D-vi)} in the proof of Proposition~\ref{prop: Q-socle for pair}.

Write $\al=\lf a , b \rf$ and $\be=\lf c, d \rf$ and set $m = \min( a,c ) \in I$.
Assume first that $a \ne c$. Then there are three positive roots having $\ve_a$ as
its summand and one of $\ve_b$, $\ve_c$ and $\ve_d$ as its another summand; i.e.,
\[
\ve_a+\ve_b, \quad \ve_a+\ve_c, \quad \ve_a+\ve_d.
\]
Thus there are at most three pairs whose weights are the same as
$\al+\be$. Moreover, we proved that one of them is $[Q]$-simple. Thus $\dist_Q(\al,\be) \le 2$.

Now we assume that $a=c$. Then such pairs happen only in the case {\rm (D-ii)} with $k'=n-1$ in the proof of Proposition~\ref{prop: Q-socle for pair}. The pair $(\al,\be)$ is the only the pair with its weight $\al+\be$
and hence $\dist_Q(\al,\be)= 1$ since $\soc_Q(\al,\be) \ne (\al,\be)$. Thus our assertion follows.
\end{proof}

By the above theorem, one can notice that every pair $(\al,\be)$ of type $A_n$ is $[Q]$-minimal or simple.

\begin{proposition} \label{prop: good neighbor}
Let $Q$ be a Dynkin quiver of type $A_n$, $D_n$ and $E_6$. For a pair $\up=(\al,\be)$ with $\dist_{Q}(\al,\be) \ge 2$ and $\al+\be=\ga \in \PR$,
there are at least two $[Q]$-minimal pairs $\up^{(1)}$ and $\up^{(2)}$ such that $\up^{(i)}$ and $\up$ are good neighbors.
\end{proposition}

\begin{proof} ($Q$ of type $D$) By Theorem~\ref{thm: rds mul}, we can assume that $(\al,\be)$ is a pair such that $\dist_Q(\al,\be)=2$.
Then we know that $\ga$ is not multiplicity free  and the relative positions of $\al,\be$ and $\ga$ in $\GQ$ can be described as in~\eqref{eq: figure D dist 2}.
Take any pair among the minimal pairs $(\al',\be'),(\eta,\zeta),(\eta',\zeta')$ of $\ga$ in~\eqref{eq: dist 2 of D} and write it as $(\al^{(1)},\be^{(1)})$.
By Proposition~\ref{prop: dir Q cnt} and~\eqref{eq: minimal pair ga D}, we have either
\begin{itemize}
\item[{\rm (i)}]  $\be-\be^{(1)} \in \PR$ and $\al^{(1)}-\al \in \PR$ or
\item[{\rm (ii)}]  $\al-\al^{(1)} \in \PR$ and $\be^{(1)}-\be \in \PR$.
\end{itemize}
Assume the case {\rm (i)} first. Write $\eta=\be-\be^{(1)} \in \PR$. Then we have $\be \prec_Q \eta$, $\al^{(1)} \prec_Q \eta$ and $\al^{(1)}+\eta=\al$
by the convexity of $\prec_Q$ and~\eqref{eq: minimal pair ga D}. Since the pairs $(\be^{(1)},\be)$ and $(\al,\al^{(1)})$ are sectional,
Proposition~\ref{prop: dir Q cnt} tells that the pairs $(\beta^{(1)},\eta)$ and $(\al,\eta)$ are minimal pairs for $\be$ and $\al^{(1)}$ respectively; i.e.,
$$ \dist_Q(\be^{(1)},\eta)=\dist_Q(\al,\eta)=1.$$
Thus our assertion follows since $\dist_Q(\al,\be)=2$. The case ${\rm (ii)}$ can be proved in the similar way.
\end{proof}

Now we record the following observation on $E_{6,7,8}$-types which can be obtained by applying the strategy in Remark~\ref{eq: observations reflection}:

\begin{proposition} \label{prop: good neighbor2}
For any pairs $\up=(\al,\be)$ of $\ga=\al+\be \in \PR$ and $\dist_Q(\up) \ge 3$, there are at least two $[Q]$-minimal pairs $\up^{(1)}$ and $\up^{(2)}$ of $\ga$
such that $\up^{(i)}$ and $\up$ are good neighbors $(i=1,2)$.
\end{proposition}

\begin{example} \label{ex: good neighbor}
For a pair $\up=(1111100,0112221)$ and the Dynkin quiver $Q$ of type $E_7$ in Appendix~\ref{Sec:Dynkin E7}, the minimal pairs $\up^{(1)}=(0111111,1112210)$ and
$\up^{(2)}=(1122211,0101110)$ are good neighbors of $\up$ since we have sequences of pairs
\begin{align*}
& \up^{(1)} \prec_Q^{\tb} (1111110,0112211) \prec_Q^{\tb} \up \ \ \text{ and } \ \  \up^{(2)} \prec_Q^{\tb} (0111100,1112221) \prec_Q^{\tb} \up,
\end{align*}
which satisfy the conditions in Definition~\ref{def: good neighbor}.
\end{example}

%\subsection{The number of pairs $(\al,\be)$ of $\ga \in \PR$}
%In~\cite{Oh14A,Oh14D}, the following propositions are already observed for finite types $A_n$ and $D_n$. We record the extended results here for
%finite types $A_n$, $D_n$ and $E_{6,7,8}$ without their proofs.

%\begin{proposition}
%For $\ga \in \PR$, there exists $\het(\ga)$-many sequences $\um_\redez$ $($for any $\redez$ of $w_0)$ such that $|\um| \le 2$ and $\wt(\um)=\ga$.
%\end{proposition}

%\begin{proposition}
%For $\ga \in \PR \setminus \Pi$ of type $A_n$, $D_n$ or $E_6$ and any Dynkin quiver $Q$,
%$$  \big|\{ \up \ | \ \up \text{ is a pair}, \wt(\up)=\ga \text{ and } k=\dist_Q(\up) \ge 1 \}\big|=\supp_{\ge k}(\ga)-\delta_{k,1}.$$
%\end{proposition}

\section{Application to KLR algebras} \label{Sec: KLR}

In this section, we shall apply our results in the previous sections to the representation theory of KLR-algebras. In particular, we will deal with
the category consisting of finite dimensional module over the KLR-algebras.
%\sejin{Rewrite!}

\subsection{KLR algebras and categorifications} In this subsection, we review the representation theories on KLR algebras which were introduced
in~\cite{KL09,KL11,R08}.

\medskip

Let $\ko$ be a field and we assume that a symmetrizable Cartan datum
$(\cmA,\wl,\Pi,\wl^\vee,\Pi^\vee)$ is given.
The symmetric group $\mathfrak{S}_m= \langle \mathfrak{s}_1,\mathfrak{s}_2,\ldots,\mathfrak{s}_{m-1} \rangle$ acts on $I^m$ by place permutations.
For $n \in \Z_{\ge 0}$ and $\mathsf{b} \in \rl^+$ such that $\het(\mathsf{b}) = n$, we set
$$I^{\mathsf{b}} = \{\nu = (\nu_1, \ldots, \nu_n) \in I^{n} \ | \ \alpha_{\nu_1} + \cdots + \alpha_{\nu_n} = \beta \}.$$

%\begin{definition}
For $\mathsf{b} \in \rl^+$, we denote by $R(\mathsf{b})$ \defn{the KLR algebra} at $\mathsf{b}$ associated
with the symmetrizable Cartan datum and a suitable family of
$(\Q_{i,j})_{i,j \in I} \in \ko[u,v]$. It is a $\Z$-graded $\ko$-algebra generated by
the generators
$$\{ e(\nu) \}_{\nu \in  I^{\mathsf{b}}}, \ \  \{x_k \}_{1 \le
k \le \het(\mathsf{b})} \  \text{ and } \  \{ \tau_m \}_{1 \le m < \het(\mathsf{b})}$$ with the certain defining relations (see~\cite[Definition 2.7]{Oh14A} for the relations).

Let $\Rep(R(\mathsf{b}))$ be the category consisting of finite dimensional graded $R(\mathsf{b})$-modules and $[\Rep(R(\mathsf{b}))]$ be the Grothendieck group of $\Rep(R(\mathsf{b}))$.
Then $[\Rep(R(\mathsf{b}))]$ has a natural $\Z[q^{\pm 1}]$-module structure induced by the grading shift. In this paper,
we usually ignore grading shifts.

\medskip

For $M \in \Rep(R(\mathsf{a}))$ and $N \in \Rep(R(\mathsf{b}))$, we denote by $M \conv N$ the
\defn{convolution product} of $M$ and $N$. Then $[\Rep(R)] \seteq \soplus_{\mathsf{b} \in \rl^+} [\Rep(R(\mathsf{b}))]$ has a natural $\Z[q^{\pm 1}]$-algebra structure
induced by the convolution product $\conv$ and the \defn{grading shift functor} $q$.

For $M \in \Rep(\mathsf{b})$ and $M_k \in \Rep(\mathsf{b}_k)$ $(1 \le k \le n)$, we denote by
$$ M^{\conv 0} \seteq \ko, \quad M^{\conv r} = \overbrace{ M \conv \cdots \conv M }^r, \quad \dct{k=1}{n} M_k = M_1 \conv \cdots \conv M_n.$$

For $M \in \Rep(R(\mathsf{a}))$, the dual space
$$M^\psi \seteq {\rm Hom}_{\ko}(M,\ko)$$
admits an $R(\mathsf{a})$-module structure via
$$ (r \cdot f)(u) \seteq f(\psi(r)u) \quad (r \in R(\mathsf{a}), \ u \in M) $$
where $\psi$ denotes the $\ko$-algebra anti-involution on $ R(\mathsf{a})$ fixing generators
$e(\nu)$, $\tau_m$ $(1 \le m < \het(\mathsf{a}))$ and $x_k$ $(1 \le m \le \het(\mathsf{a}))$. A simple module $M$ is \defn{self dual} if $M^\psi \simeq M$. Every simple module is isomorphic to a grading shift of a self-dual simple module (\cite[\S 3.2]{KL09}).

\begin{theorem}[{\cite{KL09, R08}}] \label{Thm:categorification}
For a given symmetrizable Cartan datum $\mathsf{D}$, Let $U^-_{\A}(\mathsf{g})^\vee$ $(\A=\Z[q^{\pm 1}])$ the dual of the integral form
of the negative part of quantum group $U_q(\mathsf{g})$ related to $\mathsf{D}$ and $R$ be the KLR algebra corresponding to $\mathsf{D}$
and $(\Q_{ij}(u,v))_{i,j \in I}$. Then we have
\begin{align}
\Psi : [\Rep(R)]  \overset{\simeq}{\Lto} U^-_{\A}(\mathsf{g})^{\vee}.
\label{eq:KLRU}
\end{align}
\end{theorem}

We say that the KLR-algebra $R$ is \defn{symmetric} if $\cmA$ is symmetric and
$\Q_{ij}(u,v)$ is a polynomial in $u-v$ for all $i,j\in I$.

\begin{theorem}~\cite{R11,VV09} \label{thm:categorification 2}
Assume that the KLR-algebra $R$ is symmetric and
$\ko$ is of characteristic zero.
 Then under the isomorphism~\eqref{eq:KLRU},
the dual canonical/upper global basis $\mathbf{B}^{{\rm up}}$ of $U^-_{\A}(\mathsf{g})^{\vee}$ corresponds to
the set of the isomorphism classes of self-dual simple $R$-modules.
\end{theorem}

\subsection{Simple heads and socles} In this subsection, we collect the results on the convolution product of
the real simple module and simple module over symmetric KLR algebra $R$
for our purpose (\cite{KKKO14S}).
For the rest of this paper, we assume that all KLR algebras are symmetric and
$\ko$ is of characteristic zero unless stated otherwise.

%%%%%%%%%%%%%%%%%%%
\medskip
We say that a simple $R(\mathsf{b})$-module $M$ is \defn{real} if $M\conv M$ is simple.

\begin{theorem} \label{thm: socle head} $($\cite[Theorem 3.2, Corollary 3.9]{KKKO14S}$)$
Let $M$ be a simple $R(\mathsf{a})$-module and $N$ be a simple $R(\mathsf{b})$-module, one of which is real. Then there exist unique non-zero homomorphisms $($up to constant multiple$)$
$$ \rmat{M,N} \colon M \conv N \Lto N \conv M \quad \text{and} \quad \rmat{N,M} \colon N \conv M \Lto M \conv N,$$
satisfying the following properties:
\begin{enumerate}
\item[{\rm (i)}] $M \conv N$ has a simple socle and a simple head.
%\item[{\rm (ii)}] ${\rm Hom}_{R}(M \conv N,N\conv M)= \ko \ \rmat{M,N}$.
\item[{\rm (ii)}] The socle and head of $M \conv N$ are distinct and appear once in the composition series of $M \conv N$ unless $M \conv N \simeq N \conv M$ is simple.
\item[{\rm (ii)}] If the head of $M\conv N$ and the head of $N\conv M$ are isomorphic, then $M\conv N$ is simple and $M\conv N \simeq  N\conv M$.

\end{enumerate}
\end{theorem}

% We say that $b \in \mathbf{B}^{{\rm up}}$ is {\it real} if
%$b^2 \in q^\Z\,\mathbf{B}^{{\rm up}} \seteq\bigsqcup_{n\in\Z}q^n\mathbf{B}^{{\rm up}}$.

%\begin{theorem}\label{thm: simple and socle ch} [{\cite[Theorem 3.1]{KKKO15A}}]  Let $b_1$ and $b_2$ be elements in $\mathbf{B}^{{\rm up}}$ such that one of them is real and $b_1b_2 \not\in q^\Z\mathbf{B}^{{\rm up}}$.
%Then the expansion of $b_1b_2$ with respect to $\mathbf{B}^{{\rm up}}$ is of the
%form
%$$ b_1b_2= q^m b' + q^s b'' + \sum_{c \ne b',b''}\gamma^c_{b_1,b_2}(q)c,$$
%where $b' \ne b''$, $m,s \in \Z$, $m < s$,
%and
%$$  \gamma^c_{b_1,b_2}(q) \in q^{m+1}\Z[q] \cap q^{s-1}\Z[q^{-1}].$$
%More precisely, $q^mb'$ and $q^sb''$ correspond to the
%simple head
%and the simple socle of $M\conv N$, respectively,
%when $b_1$ corresponds to a %real
%simple module $M$ and $b_2$ corresponds to a simple module $N$.
%\end{theorem}

\begin{proposition}~\cite[Corollary 3.11]{KKKO14S} \label{prop: not vanishing}
Let $M_k$ be a finite dimensional graded $R$-module $(k=1,2,3)$, and let $\varphi_1:L \to M_1 \conv M_2$ and
$\varphi_2:M_2 \conv M_3 \to L'$ be non-zero homomorphisms. Assume further that $M_2$ is simple. Then the composition
\[
L\conv M_3\To[\varphi_1\conv M_3] M_1\conv M_2\conv M_3\To[M_1\conv\varphi_2]
M_1\conv L'
\]
does not vanish. Similarly, for non-zero homomorphisms $\varphi_1:L \to M_2 \conv M_3$,
$\varphi_2: M_1 \conv M_2 \to L'$ and an assumption that $M_2$ is simple, we have a non-zero composition
\[
M_1 \conv L \To[M_1\conv \varphi_1] M_1\conv M_2\conv M_3 \To[\varphi_2 \conv M_3]
L '\conv M_3.
\]
\end{proposition}

Let $\mathsf{a},\mathsf{b}\in\rl^+$. For
a simple $R(\mathsf{a})$-module $M$ and a simple $R(\mathsf{b})$-module $N$,
let us denote by $M\hconv N$ the \defn{head} of $M\conv N$ and $M\sconv N$ the \defn{socle} of $M\conv N$

\subsection{$\Rep(R)$ and $\prec_{[\redez]}^\tb$.} In this subsection, we first review the results on $\Rep(R)$
associated to a simple Lie algebra $\g_0$ which were investigated in~\cite{BKM12,Kato12,KR09,Mc12}, briefly. Then we shall refine the results by using the order
$\prec_{[\redez]}^\tb$ which is far coarser than $<_\redez^\tb$. We would like to emphasize that the results in
this subsection work for {\it all} reduced expressions $\redez$ of {\it all} finite types.

\begin{theorem}~\cite{BKM12,Mc12} $($see also~\cite{Kato12,KR09}$)$ \label{thm: BkMc}
For a finite simple Lie algebra $\g_0$, we fix a convex total order $\le_{\redez}$ on $\PR$
induced from the reduced expression $\redez$.% $($see {\rm Convention~\ref{conv: sequence}}$)$.

Let $R$ be the KLR algebra corresponding to $\g_0$. For each positive root $\beta \in \PR$, there exists a self-dual simple module $S_{\redez}(\beta)$
satisfying the following properties$:$
\begin{itemize}
\item[({\rm a})] $S_{\widetilde{w}_0}(\beta)^{\conv m}$ is a real simple $R(m\be)$-module.
\item[({\rm b})] For every $\um_{\redez} \in \Z_{\ge 0}^{\N}$, there exists a unique non-zero $R$-module homomorphism (up to constant)
\begin{equation} \label{eq: def of r}
\begin{aligned}
&  \Stom \seteq S_{\redez}(\beta_1)^{\conv m_1}\conv\cdots \conv S_{\redez}(\beta_\N)^{\conv m_\N} \overset{\rmat{\um}}{\Lto} \Sgetsm \seteq  S_{\redez}(\beta_\N)^{\conv m_\N}\conv\cdots \conv S_{\redez}(\beta_1)^{\conv m_1}
\end{aligned}
\end{equation}
such that  ${\rm Im}(\rmat{\um}) \simeq \hd\left(\Stom\right) \simeq \soc\left(\Sgetsm\right)$ is simple.
\item[({\rm c})] For any sequence $\um_{\redez} \in \Z_{\ge 0}^{\ell(w_0)}$, we have
\begin{align} \label{eq: not good filter}
[\Stom] \in [{\rm Im}(\rmat{\um})] + \displaystyle\sum_{ \um' <^{\tb}_\redez \um } \Z_{\ge 0}[{\rm Im}(\rmat{\um'})].
\end{align}
\item[({\rm d})] For distinct $\um, \um' \in \Z_{\ge 0}^{\ell(w_0)}$, the simple modules ${\rm Im}(\rmat{\um})$ and ${\rm Im}(\rmat{\um'})$ are distinct.
\item[({\rm e})] For every simple $R$-module $M$, there exists a unique sequence $\um \in \Z_{\ge 0}^{\ell(w_0)}$ such that
$M \simeq {\rm Im}(\rmat{\um}) \simeq {\rm hd}\big(\Stom \big).$
\item[({\rm f})] For any $\redez$-minimal pair $(\beta^\redez_k,\beta^\redez_l)$ of $\beta^\redez_j=\beta^\redez_k+\beta^\redez_l$, there exists an exact sequence
$$ \qquad 0 \to S_{\redez}(\beta_j)
\to S_{\redez}(\beta_k) \conv S_{\redez}(\beta_l) \overset{{\mathbf
r}_{e_k+e_l}}{\Lto} S_{\redez}(\beta_l) \conv S_{\redez}(\beta_k) \to
S_{\redez}(\beta_j) \to 0.$$
\end{itemize}
Moreover,
\begin{eqnarray} &&
\parbox{90ex}{
under the isomorphism $\Psi$ in~\eqref{eq:KLRU}, each isomorphism class $[S_{\redez}(\beta)]$ of $S_{\redez}(\beta)$ $(\beta \in \PR)$
is mapped onto the \defn{dual root vector} $\mathbf{F}_\redez^{{\rm up}}(\beta)$
of the dual PBW basis $P_{\redez}$ of $U^-_\A(\g_0)^\vee$,
 which is associated to the reduced expression $\redez$ of $w_0$ $($see~\cite[\S 4]{KKK13B} for more details$)$.
}\label{eq: map dual}
\end{eqnarray}

\end{theorem}

\begin{remark}
The property in {\rm (f)} of Theorem~\ref{thm: BkMc} is called a \defn{length two property of minimal pair} (see \cite[Section 4.3]{BKM12}), since the composition
series of $S_{\redez}(\beta_l) \conv S_{\redez}(\beta_k)$ has its length as $2$ and consists of its distinct head and socle. The
length two property sometimes holds for some $\redez$ and its non $[\redez]$-minimal pair $(\al,\be)$ of $\ga$. For the $\redez$
in Remark~\ref{ex: non-trivial 1} and the non $[\redez]$-minimal pair $([1],[2,5])$, one can compute that
$\overset{\to}{S}_{\redez}([1],[2,5])$ has its composition length $2$ where its composition series consists of
${S}_{\redez}([1]) \hconv {S}_{\redez}([2,5])$ as its head and ${S}_{\redez}([1,3]) \hconv {S}_{\redez}([4,5]) \simeq {S}_{\redez}([2,5]) \hconv {S}_{\redez}([1])$
as its socle.
\end{remark}

For any $\redez ,\redez' \in [\redez]$, we have
$$ S_{\redez}(\be) \simeq  S_{\redez'}(\be)  \quad \text{ for all } \be \in \PR,$$
by Remark~\ref{rmk:coarset order property} and Theorem~\ref{thm: BkMc} $({\rm f})$. Thus we denote by
$S_{[\redez]}(\be)$ the simple module $S_{\redez'}(\be)$ for any $\redez' \in [\redez]$. However, at this moment,
the definition of $\Stom$ and~\eqref{eq: not good filter} {\it does} depend on $\redez$.

\begin{proposition} \label{prop: incompa simple} %~\referee{Proposition 5.9 should be reformulated.} %\sejin{revise more!!}
Let $(\alpha,\beta)$ be an incomparable pair with respect to the
order $\prec_{[\redez]}$. Then we have
$$  S_{[\redez]}(\al) \conv S_{[\redez]}(\be) \simeq  S_{[\redez]}(\be) \conv S_{[\redez]}(\al) \text{ is simple.} $$
\end{proposition}

\begin{proof}
By the assumption and Remark~\ref{rmk:coarset order property}, there exist reduced expressions $\redez^{(1)},\redez^{(2)} \in [\redez]$ such that
$$\text{{\rm (i)} $\al <_{\redez^{(1)}} \be$} \quad \text{ and } \quad \text{{\rm (ii)} $\be <_{\redez^{(2)}} \al$.}$$

By Theorem~\ref{thm: BkMc}, we have
\begin{itemize}
\item[{\rm (i)}] $\big[\overset{\to}{S}_{\redez^{(1)}}(\al,\be)\big] \seteq \big[S_{[\redez]}(\al) \conv S_{[\redez]}(\be)\big] \ \in \big[{\rm Im}(\rmat{(\al,\be)})\big] +
\hspace{-3ex} \displaystyle\sum_{\quad \um'_{\redez^{(1)}} <^{\tb}_{\redez^{(1)}} (\al,\be) } \hspace{-3ex} \Z_{\ge 0} \big[{\rm
Im}(\rmat{\um'_{\redez^{(1)}}})\big],$ where \\
$\rmat{(\al,\be)}$ is the unique non-zero homomorphism (up to constant).
\item[{\rm (ii)}] $\big[\overset{\to}{S}_{\redez^{(2)}}(\be,\al)\big] \seteq \big[S_{[\redez]}(\be) \conv S_{[\redez]}(\al)\big] \ \in \big[{\rm Im}(\rmat{(\be,\al)})\big] +
\hspace{-3ex}\displaystyle\sum_{\quad \um'_{\redez^{(2)}} <^{\tb}_{\redez^{(2)}} (\be,\al) } \hspace{-3ex} \Z_{\ge 0} \big[{\rm Im}(\rmat{\um'_{\redez^{(2)}}})\big],$ where \\
$\rmat{(\be,\al)}$ is the unique non-zero homomorphism (up to constant).
\end{itemize}
Note that the Grothendieck ring $[\Rep(R)]$ is commutative up to $q^\Z$ and hence
$$\big[S_{[\redez]}(\al) \conv S_{[\redez]}(\be)\big]=\big[S_{[\redez]}(\al)\big]\big[ S_{[\redez]}(\be)\big]=q^a\big[ S_{[\redez]}(\be)\big]\big[S_{[\redez]}(\al)\big]= q^a\big[S_{[\redez]}(\be) \conv S_{[\redez]}(\al)\big]$$
for some $a \in \Z$. Assume that ${\rm Im}(\rmat{(\al,\be)}) \not \simeq {\rm Im}(\rmat{(\be,\al)})$.
By Theorem~\ref{thm: BkMc},
$\rmat{(\al,\be)}$ must be one of $\rmat{\um'_{\redez^{(2)}}}$.
However, it can not happen. Thus we have
\[
[{\rm Im}(\rmat{(\be,\al)})] = [{\rm Im}(\rmat{(\al,\be)})] \iff {\rm Im}(\rmat{(\be,\al)}) \simeq q^a \ {\rm Im}(\rmat{(\al,\be)}),
\]
for some $a \in \Z$. Now our assertion follows from Proposition~\ref{thm: socle head} {\rm (iii)}.
\end{proof}

By the above proposition,
the isomorphism class of the module $\overset{\to}{S}_{\redez'}(\um)$ and the homomorphism $\rmat{\um_{\redez'}}$ do not depend
on the reduced expression $\redez' \in [\redez]$ but on $[\redez]$. Thus we can denote it by
$\overset{\to}{S}_{[\redez]}(\um)$ even though $\um = \um_{\redez'}$ for $\redez' \in [\redez]$. Furthermore,~\eqref{eq: not good filter} can be refined:

\begin{theorem} \label{thm: [redez] works} \hfill
\begin{enumerate}
\item[{\rm (a)}] $\overset{\to}{S}_{[\redez]}(\um)$ is well-defined; that is,
$$\overset{\to}{S}_\redez(\um_\redez)  \simeq \overset{\to}{S}_{\redez'}(\um_{\redez'}) \quad \text{ for all } \redez,\redez' \in [\redez].$$
\item[{\rm (b)}] For any $\redez$ of $w_0$ and any sequence $\um$, we have
$$[\overset{\to}{S}_{[\redez]}(\um)] \in [{\rm Im}(\rmat{\um})] +
\sum_{\um' \prec^{\tb}_{[\redez]} \um} \Z_{\ge 0} [{\rm Im}(\rmat{\um'})].$$
\end{enumerate}
\end{theorem}

\begin{remark}
By~\eqref{eq: map dual} in Theorem~\ref{thm: BkMc}, Theorem~\ref{thm: [redez] works} implies that the dual PBW-basis $P_{\redez}$ for a reduced expression $\redez$ of $w_0$
{\it does} depend on its commutation class $[\redez]$ up to $q^\Z$ indeed. Thus we can write $P_{[\redez]}$ instead of $P_{\redez}$. It was known in $A_n$, $D_n$ and $E_{6,7,8}$ (see~\cite{Kato12}) but seems to be new result
for other types, to the best knowledge of the author.
\end{remark}

%\Sejin{Some remark on PBW basis and $[\redez]$ are needed}

\begin{corollary} \label{cor: simple simple module}
For any $[\redez]$-simple sequence $\us$, then $\overset{\to}{S}_{[\redez]}(\us)\simeq \overset{\gets}{S}_{[\redez]}(\us)$ is real simple.
\end{corollary}

\begin{proof}
By the definition of $[\redez]$-simple sequence, the $[\redez]$-simple sequence $\us$ is a minimal element with respect to $\prec_{[\redez]}^\tb$.
Thus our assertion follows from Theorem~\ref{thm: [redez] works}.
\end{proof}

%\begin{corollary} \label{cor simple if simple}
%For a $[\redez]$-simple sequence $\us$, the module $\overset{\to}{S}_{[\redez]}(\us) \simeq  \overset{\gets}{S}_{[\redez]}(\us)$ is real simple.
%\end{corollary}

%\begin{proof}
%For any $i,j$ such that $s_i,s_j \ne 0$, the pair $\overset{\to}{S}_{[\redez]}(\be_i,\be_j) \simeq \overset{\gets}{S}_{[\redez]}(\be_i,\be_j)$
%by Corollary~\ref{cor: simple simple module}.
%Thus there exists an isomorphism between $\overset{\to}{S}_{[\redez]}(\us)$ and $\overset{\gets}{S}_{[\redez]}(\us)$.
%Then our assertion follows from Theorem~\ref{thm: BkMc} {\rm (b)}. % and~\cite[Corollary 3.9]{KKKO14S}.
%\end{proof}

For the simplicity of notations, we denote by $S_{[\redez]}(\us)$ instead of $\overset{\to}{S}_{[\redez]}(\us) \simeq \overset{\gets}{S}_{[\redez]}(\us)$
when %the KLR-algebra $R$ is symmetric and
$\us$ is a $[\redez]$-simple sequence.

\subsection{$[Q]$-socle} \label{subsec: Q-socle}
In this subsection, we only deal with the commutation class $[Q]$ for some Dynkin quiver $Q$. We shall study
the representation theoretical meanings of $[Q]$-socle and $[Q]$-length in the category $\Rep(R)$, where $R$ is the symmetric KLR algebra
associated to the underlying Dynkin diagram $\Delta$ of $Q$.

\medskip
For the simplicity of notation, we denote by $\overset{\to}{S}_Q(\um)$ instead of $\overset{\to}{S}_{[Q]}(\um)$.

\begin{proposition} \label{prop: socle mathod 1}
Let $M_i$ and $N_i$ $(i=1,2)$ be simple $R$-modules such that one of them is real and $M_i \conv N_i$ has composition length $2$. For a real simple $R$-modules $M$ and
a simple module $N$, assume
\begin{enumerate}
\item[{\rm (a)}] there exists a non-zero $R$-homomorphism $\varphi_i \colon M_i \conv N_i \to M \conv N$ $(i=1,2)$,
\item[{\rm (b)}] $M_1 \sconv N_1 \simeq  M_2 \sconv N_2$ and the pairs $(M_i,N_i)$ are distinct.
\end{enumerate}
Then we have
$$M_1 \sconv N_1 \simeq M_2 \sconv N_2  \simeq  M \sconv N  \quad \text{ and } \quad \varphi_i \colon M_i \conv N_i \hooklongrightarrow M \conv N.$$
\end{proposition}

\begin{proof} By the assumptions and Theorem~\ref{thm: socle head}, we have $M \hconv N \subset {\rm Im}(\varphi_1) \cap {\rm Im}(\varphi_2)$.
However, if one of $\varphi_1$ and $\varphi_2$ is not injective,
${\rm Im}(\varphi_1) \cap {\rm Im}(\varphi_2)$ becomes empty since
\[
M_1 \hconv N_1  \not\simeq  M_2\hconv N_2 %\qquad  \text{and} \qquad M_1 \sconv N_1 \simeq M_2 \sconv N_2.
\]
Thus our assertion follows.
\end{proof}

Since $\overset{\to}{S}_Q(\al,\be)$ has composition length $2$ for a $[Q]$-minimal pair $(\al,\be)$ of $\ga$, we have the followings:

\begin{corollary} \label{cor: socle mathod 1}
 Let $(\al_i,\be_i)$ $(i=1,2,3)$ be distinct pairs for $\gamma =\al_i+\be_i \in \PR$ such that $(\al_j,\be_j)$ $(j=1,2)$ is $[Q]$-minimal and
$(\al_3,\be_3)$ is not. If there exists a non-zero $R(\gamma)$-homomorphism
$$\overset{\to}{S}_Q(\al_j,\be_j)  \longrightarrow \overset{\to}{S}(\al_3,\be_3) \quad \ (j=1,2),$$ then we have
$$ \soc\big(\overset{\to}{S}_{Q}(\al_3,\be_3)\big) \simeq S_Q(\gamma) \quad \text{ and } \quad
\overset{\to}{S}_Q(\al_j,\be_j)  \hooklongrightarrow \overset{\to}{S}(\al_3,\be_3) \quad \ (j=1,2).$$
\end{corollary}

\begin{proposition} \label{prop: socle mathod 2}
We assume that
\begin{enumerate}
\item[{\rm (a)}] $M_i$ and $N_i$ $(i=1,2)$ be simple $R$-modules such that one of them is real,
\item[{\rm (b)}] there exists an injective $R$-homomorphism $M_1 \conv N_1 \hooklongrightarrow M_2 \conv N_2$.
\end{enumerate}
Then we have
\[
M_1 \sconv N_1  \simeq M_2 \sconv N_2.
\]
\end{proposition}

\begin{proof} The proof follows from the definition of socle and Theorem~\ref{thm: socle head}.
\end{proof}

\begin{corollary} \label{cor: socle mathod 2}
 Let $(\al_i,\be_i)$ $(i=1,2)$ be pairs for $\gamma =\al_i+\be_i \in \PR$ such that $ \soc\big(\overset{\to}{S}_{Q}(\al_1,\be_1)\big) \simeq S_Q(\gamma)$.
Assume there exists an injective $R$-homomorphism
$$\overset{\to}{S}_Q(\al_i,\be_i)  \hooklongrightarrow M \conv N \quad (i=1,2)$$
for simple $R$-modules $M$ and $N$ such that one of them is real. Then we have
$$ \soc\big(\overset{\to}{S}_{Q}(\al_2,\be_2)\big) \simeq S_Q(\gamma).$$
\end{corollary}

\begin{example} \label{ex: bad length2}
Consider the positive root $\ga=(0112111)$ of $\PR_{E_7}$. From the quiver $Q$ in Appendix~\ref{Sec:Dynkin E7}, one can check
that the pair $(0112100,0000011)$ of $\ga$ has $[Q]$-distance
$2$ since $$\up'\seteq (0111111,0001000) \prec^{\tb}_{Q} \up \seteq (0112100,0000011).$$
Note that $\up'$, $\up$ are {\it not} good neighbors. Furthermore, there is no more pair $(\al,\be)$ of $\ga$ such that
$ (\al,\be) \prec^{\tb}_{Q} \up$. Then we have
$$
\scalebox{0.9}{\xymatrix@C=3ex@R=3ex{
S_Q(0112100) \conv S_Q(0000011) \ar@{^{(}->}[rr] && S_Q(0111100) \conv S_Q(0001000) \conv S_Q(0000011)  \ar[dd]_{\simeq}  \\
& S_Q(0112111) \ar@{^{(}.>}[ul]\ar@{^{(}->}[dl] \\
 S_Q(0111111) \conv S_Q(0001000) \ar@{^{(}->}[rr] && S_Q(0111100) \conv S_Q(0000011) \conv  S_Q(0001000)  }}
$$
by {\rm (f)} of Theorem~\ref{thm: BkMc} and the $[Q]$-minimalities of
\[
\text{$(0111100,0001000)$, $(0111100,0000011)$ and $(0111111,0001000)$.}
\]
Note that the isomorphism in the second column follows from the fact that the pair $(0001000,0000011)$ is $[Q]$-simple. Then
Corollary~\ref{cor: simple simple module} implies that $S_Q(0001000) \conv S_Q(0000011)$ is real simple.
By taking $M=S_Q(0111100)$ and $N=S_Q(0001000) \conv S_Q(0000011)$, we can apply the above corollary.
\end{example}

The following theorem is already proved for types $A_n$ and $D_n$ in~\cite{Oh14A,Oh14D} by using Dorey's rule and KLR-type Schur-Weyl duality functor. %\sejin{Check this notation}
Now, we can extend the result to finite $E$-types:

\begin{theorem} \label{thm: socle of Q-pair with al+be in PR}
For a pair $\up=(\al,\be)$ with $\al+\be = \ga \in \PR$ of type $A_n$, $D_n$ and $E_{6,7,8}$, we have
$$ \soc\big(\overset{\to}{S}_{Q}(\al,\be)\big) \simeq S_{Q}(\ga).$$
\end{theorem}

\begin{proof}
For the case when $(\al,\be)$ is a $[Q]$-minimal pair of $\gamma$, i.e., $\dist_Q(\al,\be)=1$, our assertion follows from
{\rm (f)} of Theorem~\ref{thm: BkMc}. Since each pair $(\al,\be)$ of $\ga$ of type $A_n$ is $[Q]$-minimal (\cite[Theorem 3.4]{Oh14A}),
we can first assume that $\dist_Q(\al,\be)=2$ and $Q$ is of type $D_n$ or $E_6$. From the Proposition~\ref{prop: good neighbor},
there exist good {\it adjacent} neighbors $\up^{(1)}$ and $\up^{(2)}$ of $\up$ which are also minimal pairs of $\ga$; that is; there exists $\eta^{(i)} \in \PR$ satisfying either
\begin{enumerate}
\item[{\rm (a1)}] $\eta^{(i)}+\beta=\beta^{(i)}, \ \eta+\al^{(i)}=\al$ and $1=\dist_{Q}(\eta,\beta)=\dist_{Q}(\eta,\al^{(i)})<\dist_{Q}(\up)=2$,
\item[{\rm (b1)}] $\beta^{(i)}+\eta^{(i)}=\beta, \ \al+\eta^{(i)}=\al^{(i)}$ and $1=\dist_{Q}(\beta^{(i)},\eta)=\dist_{Q}(\al,\eta)<\dist_{Q}(\up)=2$.
\end{enumerate}
In both cases, we have homomorphisms
\begin{enumerate}
\item[{\rm (a2)}] $S_Q(\beta^{(i)}) \hooklongrightarrow S_Q(\eta^{(i)}) \conv S_Q(\beta), \ S_Q(\al^{(i)}) \conv S_Q(\eta^{(i)}) \longtwoheadrightarrow  S_Q(\al)$,
\item[{\rm (b2)}] $S_Q(\eta^{(i)}) \conv S_Q(\beta^{(i)})\longtwoheadrightarrow S_Q(\be) , \ S_Q(\al^{(i)})\hooklongrightarrow S_Q(\al) \conv S_Q(\eta^{(i)})$.
\end{enumerate}
Thus we have non-zero compositions
\begin{enumerate}
\item[{\rm (a3)}] $S_Q(\al^{(i)}) \conv S_Q(\beta^{(i)})\hooklongrightarrow S_Q(\al^{(i)}) \conv S_Q(\eta^{(i)}) \conv S_Q(\beta) \longtwoheadrightarrow S_Q(\al) \conv S_Q(\be)$,
\item[{\rm (b3)}] $S_Q(\al^{(i)}) \conv S_Q(\beta^{(i)})\hooklongrightarrow S_Q(\al) \conv S_Q(\eta^{(i)}) \conv S_Q(\beta^{(i)})\longtwoheadrightarrow  S_Q(\al) \conv S_Q(\be)$,
\end{enumerate}
by Proposition~\ref{prop: not vanishing}. Then our assertion follows from Corollary~\ref{cor: socle mathod 1}.

For the cases when $Q$ is of type $E_7$ or $E_8$, there are several $\ga \in \PR$ which has a pair $\up=(\al,\be)$ of $\ga$  (see Example~\ref{ex: bad length2})
such that $\dist_Q(\al,\be)=2$ and
\begin{itemize}
\item $\up$ does not have good neighbor at all,
\item there exists a unique minimal pair $\up'=(\al',\be')$ of $\ga$ such that $\up' \prec_Q^\tb \up$.
\end{itemize}
However, we can apply the same argument of Example~\ref{ex: bad length2} in those cases. More precisely, there exists $\eta \in \PR$ such that
one of the following conditions holds:
\begin{enumerate}
\item[{\rm (i)}] $(\eta,\be')$ is a $[Q]$-minimal pair of $\al$, $(\eta,\be)$ is a $[Q]$-minimal pair of $\al'$ and $(\be,\be')$ are $[Q]$-simple,
\item[{\rm (ii)}] $(\al',\eta)$ is a $[Q]$-minimal pair of $\be$, $(\al,\eta)$ is a $[Q]$-minimal pair of $\be'$ and $(\al,\al')$ are $[Q]$-simple.
\end{enumerate}

Thus we have
\[
\scalebox{0.9}{\xymatrix@C=3ex@R=3ex{
S_Q(\al) \conv S_Q(\be) \ar@{^{(}->}[rr] && S_Q(\eta) \conv S_Q(\be') \conv S_Q(\be)  \ar[dd]_{\simeq} & \text{(resp. } S_Q(\al) \conv S_Q(\al') \conv S_Q(\eta) \text{) \ar[dd]_{\simeq}}  \\
& S_Q(\ga) \ar@{^{(}.>}[ul]\ar@{^{(}->}[dl] \\
 S_Q(\al') \conv S_Q(\be') \ar@{^{(}->}[rr] && S_Q(\eta) \conv S_Q(\be) \conv  S_Q(\be') & \text{(resp. } S_Q(\al') \conv S_Q(\al) \conv  S_Q(\eta) \text{)}  }}
\]
Thus our assertion holds for the pairs $(\al,\be)$ with $\dist_Q(\al,\be) \le 2$.

For the pairs $(\al,\be)$ with $\dist_Q(\al,\be) \ge 3$, we can apply the induction. More precisely, by Proposition~\ref{prop: good neighbor}
and Proposition~\ref{prop: good neighbor2}, there are good neighbors $\up^{(i)}=(\al^{(i)},\be^{(i)})$ $(i=1,2)$ of $\up$, which are also $[Q]$-minimal pair of $\ga$. By the induction hypothesis, we have non-zero
compositions
\begin{enumerate}
\item[{\rm (a)}] $S_Q(\al^{(i)}) \conv S_Q(\beta^{(i)})\hooklongrightarrow S_Q(\al^{(i)}) \conv S_Q(\eta^{(i)}) \conv S_Q(\beta) \longtwoheadrightarrow S_Q(\al) \conv S_Q(\be)$ or
\item[{\rm (b)}] $S_Q(\al^{(i)}) \conv S_Q(\beta^{(i)})\hooklongrightarrow S_Q(\al) \conv S_Q(\eta^{(i)}) \conv S_Q(\beta^{(i)})\longtwoheadrightarrow  S_Q(\al) \conv S_Q(\be)$,
\end{enumerate}
since
\[
\dist_{Q}(\eta,\beta),\dist_{Q}(\eta,\al^{(i)})<\dist_{Q}(\up) \quad \text{ or } \quad  \dist_{Q}(\beta^{(i)},\eta),\dist_{Q}(\al,\eta)<\dist_{Q}(\up).
\]
Thus our assertion follows from Corollary~\ref{cor: socle mathod 1}.
\end{proof}

\begin{example}
In Remark~\ref{ex: non-trivial 1}, one can check that $\soc_{[\redez]}([5],[2,4])=([2,5])$ but
$S_{[\redez]}([2,5])$ do not appear as a composition factor of $S_{[\redez]}([5])\conv S_{[\redez]}([2,4])$. Thus the above theorem
does not holds for $[\redez] \ne [Q]$ in general.
\end{example}

Now we shall generalize Theorem~\ref{thm: socle of Q-pair with al+be in PR} to all pairs $(\al,\be)$ by using the assumptions and the arguments of
Proposition~\ref{prop: Q-socle for pair}.

\begin{theorem} \label{thm: socle of Q-pair}
For a pair $(\al,\be)$ of type $A_n$, $D_n$ and $E_{6,7,8}$, we have
\[
 S_{Q}\big(\soc_Q(\al,\be) \big)    \simeq  \soc\big(\overset{\to}{S}_{Q}(\al,\be)\big).
\]
\end{theorem}

\begin{proof}
When $\al+\be \in \PR$, it is already covered in Theorem~\ref{thm: socle of Q-pair with al+be in PR}. Thus it is enough to consider when
\begin{align} \label{cond: socle}
\text{$\al+\be \not \in \PR$ and $\up=(\al,\be)$ is not $[Q]$-simple.}
\end{align}

\noindent
($Q$ of type $A_n$) In this case,~\eqref{cond: socle} can happen only in {\rm (c)} of~\eqref{eq: A complicated socle}.
Then $\us$ and $\up$ are good adjacent neighbor such that $\us \prec^\tb_Q \up$ by Proposition~\ref{prop: dir Q cnt} and Theorem~\ref{thm: dist upper bound}.
Thus our assertion holds from the existence of non-zero homomorphism $S_{Q}(\us) \to \overset{\to}{S}_{Q}(\up)$,
the simplicity of $S_{Q}(\us)$ and Theorem~\ref{thm: socle head}.

\medskip

\noindent
($Q$ of type $D$) In this case,~\eqref{cond: socle} can happen only in~\eqref{eq: complacted socle of D} which we have recorded; i.e.,
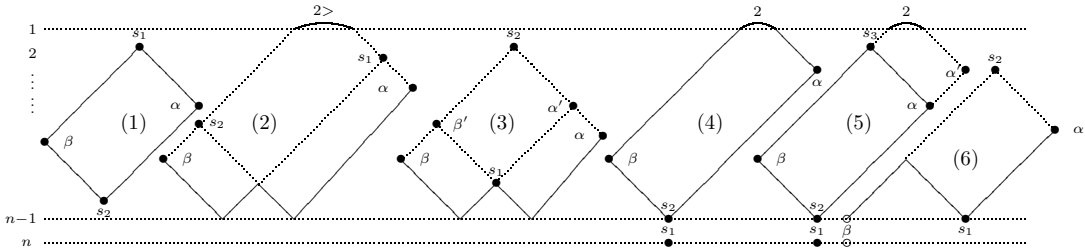
\begin{align*}
\scalebox{0.79}{{\xy
(-20,0)*{}="DL";(-10,-10)*{}="DD";(0,20)*{}="DT";(10,10)*{}="DR";
"DT"+(-30,-4); "DT"+(135,-4)**\dir{.};
"DD"+(-20,-6); "DD"+(145,-6) **\dir{.};
"DD"+(-20,-10); "DD"+(145,-10) **\dir{.};
"DT"+(-32,-4)*{\scriptstyle 1};
"DT"+(-32,-8)*{\scriptstyle 2};
"DT"+(-32,-12)*{\scriptstyle \vdots};
"DT"+(-32,-16)*{\scriptstyle \vdots};
"DT"+(-34,-36)*{\scriptstyle n-1};
"DT"+(-33,-40)*{\scriptstyle n};
"DL"+(-10,-3); "DD"+(-10,-3) **\dir{-};
"DR"+(-14,-7); "DD"+(-10,-3) **\dir{-};
"DT"+(-14,-7); "DR"+(-14,-7) **\dir{-};
"DT"+(-14,-7); "DL"+(-10,-3) **\dir{-};
"DL"+(-6,-3)*{\scriptstyle \be};
"DR"+(-18,-7)*{\scriptstyle \al};
"DL"+(5,0)*{{\rm (1)}};
"DL"+(-10,-3)*{\bullet};"DR"+(-14,-7)*{\bullet};
"DT"+(-14,-7)*{\bullet}; "DD"+(-10,-3)*{\bullet};
"DT"+(-14,-5)*{\scriptstyle s_1};
"DD"+(-10,-5)*{\scriptstyle s_2};
"DD"+(0,4); "DD"+(10,-6) **\dir{-};
"DD"+(16,0); "DD"+(10,-6) **\dir{-};
"DD"+(16,0); "DD"+(22,-6) **\dir{-};
"DD"+(42,16); "DD"+(22,-6) **\dir{-};
"DD"+(0,4); "DD"+(22,26) **\dir{.};
"DD"+(42,16); "DD"+(32,26) **\dir{.};
"DD"+(16,0); "DD"+(37,21) **\dir{.};
"DD"+(16,0); "DD"+(6,10) **\dir{.};
"DD"+(37,21)*{\bullet};
"DD"+(34,21)*{\scriptstyle s_1};
"DD"+(6,10)*{\bullet};
"DD"+(9,10)*{\scriptstyle s_2};
%"DD"+(22,26)*{\bullet};
%"DD"+(16,28)*{\scriptstyle (1,s+(i-1))};
%"DD"+(32,26)*{\bullet};
%"DD"+(35,28)*{\scriptstyle (1,t-(j-1))};
"DD"+(22,26); "DD"+(32,26) **\crv{"DD"+(27,28)};
"DD"+(27,29)*{\scriptstyle 2 >};
"DD"+(0,4)*{\bullet};
"DL"+(27,0)*{{\rm (2)}};
"DD"+(4,4)*{\scriptstyle \be};
"DD"+(42,16)*{\bullet};
"DD"+(37,16)*{\scriptstyle \al};
"DD"+(40,4); "DD"+(50,-6) **\dir{-};
"DD"+(56,0); "DD"+(50,-6) **\dir{-};
"DD"+(56,0); "DD"+(62,-6) **\dir{-};
"DD"+(74,8); "DD"+(62,-6) **\dir{-};
"DD"+(40,4); "DD"+(59,23) **\dir{.};
"DD"+(74,8); "DD"+(59,23) **\dir{.};
"DD"+(56,0); "DD"+(69,13) **\dir{.};
"DD"+(56,0); "DD"+(46,10) **\dir{.};
"DD"+(40,4)*{\bullet};
"DD"+(59,23)*{\bullet};
"DD"+(46,10)*{\bullet};
"DD"+(50,10)*{\scriptstyle \be'};
"DD"+(69,13)*{\bullet};
"DD"+(66,13)*{\scriptstyle \al'};
"DD"+(56,0)*{\bullet};
"DD"+(56,2)*{\scriptstyle s_1};
"DL"+(67,0)*{{\rm(3)}};
"DD"+(44,4)*{\scriptstyle \be};
"DD"+(74,8)*{\bullet};
"DD"+(70,8)*{\scriptstyle \al};
"DD"+(59,25)*{\scriptstyle s_2};
"DD"+(75,4); "DD"+(85,-6) **\dir{-};
"DD"+(91,0); "DD"+(85,-6) **\dir{-};
"DD"+(75,4); "DD"+(97,26) **\dir{-};
"DD"+(110,19); "DD"+(103,26) **\dir{-};
"DD"+(91,0); "DD"+(110,19) **\dir{-};
"DD"+(75,4)*{\bullet};
%"DD"+(74,23)*{\bullet};
"DL"+(102,0)*{{\rm (4)}};
"DD"+(79,4)*{\scriptstyle \be};
"DD"+(97,26); "DD"+(103,26) **\crv{"DD"+(100,28)};
"DD"+(110,19)*{\bullet};
"DD"+(110,17)*{\scriptstyle \al};
"DD"+(85,-10)*{\bullet};
"DD"+(85,-8)*{\scriptstyle s_1};
"DD"+(85,-6)*{\bullet};
"DD"+(85,-4)*{\scriptstyle s_2};
"DD"+(100,29)*{\scriptstyle 2};
"DD"+(100,4); "DD"+(110,-6) **\dir{-};
"DD"+(116,0); "DD"+(110,-6) **\dir{-};
"DD"+(100,4); "DD"+(119,23) **\dir{-};
"DD"+(122,26); "DD"+(119,23) **\dir{.};
"DD"+(128,26); "DD"+(135,19) **\dir{.};
"DD"+(129,13); "DD"+(135,19) **\dir{.};
"DD"+(122,26); "DD"+(128,26) **\crv{"DD"+(125,28)};
"DD"+(125,29)*{\scriptstyle 2};
"DD"+(135,19)*{\bullet};
"DD"+(133,19)*{\scriptstyle \al'};
"DD"+(129,13); "DD"+(119,23) **\dir{-};
"DD"+(116,0); "DD"+(129,13) **\dir{-};
"DD"+(100,4)*{\bullet};
"DD"+(119,23)*{\bullet};
"DL"+(127,0)*{{\rm (5)}};
"DD"+(104,4)*{\scriptstyle \be};
"DD"+(119,25)*{\scriptstyle s_3};
"DD"+(129,13)*{\bullet};
"DD"+(126,13)*{\scriptstyle \al};
"DD"+(110,-10)*{\bullet};
"DD"+(110,-8)*{\scriptstyle s_1};
"DD"+(110,-6)*{\bullet};
"DD"+(110,-4)*{\scriptstyle s_2};
"DD"+(125,4); "DD"+(135,-6) **\dir{-};
"DD"+(125,4); "DD"+(115,-6) **\dir{-};
"DD"+(125,4); "DD"+(140,19) **\dir{.};
"DD"+(150,9);"DD"+(135,-6) **\dir{-};
"DD"+(150,9);"DD"+(140,19) **\dir{.};
"DD"+(150,9)*{\bullet};
"DD"+(140,19)*{\bullet};
"DD"+(135,-6)*{\bullet};
"DD"+(135,-8)*{\scriptstyle s_1};
"DD"+(140,21)*{\scriptstyle s_2};
"DD"+(154,9)*{\scriptstyle \al};
"DD"+(135,4)*{{\rm (6)}};
"DD"+(115,-6)*{\circ};
"DD"+(115,-8)*{\scriptstyle \be};
"DD"+(115,-10)*{\circ};
\endxy}}
\end{align*}
(Recall that we have assumed $j \le i$ where $\phi_Q^{-1}(\al,0)=(j,t)$ and $\phi_Q^{-1}(\be,0)=(i,s)$.)% in Proposition~\ref{prop: Q-socle for pair}.)

In the cases {\rm (1)}, {\rm (2)}, {\rm (4)} and {\rm (6)}, one can easily check that $\us$ and $\up$ form good adjacent neighbors by Proposition~\ref{prop: Q-socle for pair}.
Hence our assertion follows from the same argument of type $A_n$. In the case $(3)$, $\us \prec^{\tb}_Q \up'=(\al',\be') \prec^{\tb}_Q \up$ is
a sequence of good neighbors such that ($\us$, $\up'$) and ($\up'$, $\up$) are pairs of good adjacent neighbors. Thus our assertion follows.

In the case $(5)$, the positive roots should be of the following forms:
\[
 \al=\ve_a-\ve_c, \ \be=\ve_a+\ve_b, \ s_3=\ve_b-\ve_c, \ \{s_1,s_2\}=\{\ve_a-\ve_{\ta},\ve_a+\ve_{\ta} \}.
\]
Thus there exists $\al'=\ve_a-\ve_b$ such that $\al' \prec_Q \al$ and the pair $(\al',\be)$ is of the case {\rm (4)} and has its socle as $(s_1,s_2)$. Thus we
have a non-zero composition
\begin{equation}\label{eq: dual}
\begin{aligned}
& S_Q(\be) \conv S_Q(\al) \hooklongrightarrow S_Q(\be) \conv S_Q(s_3) \conv S_Q(\al') \simeq  S_Q(s_3) \conv S_Q(\be) \conv S_Q(\al') \\
& \hspace{45ex} \longtwoheadrightarrow S_Q(s_3)\conv S_Q(s_1) \conv S_Q(s_2) \seteq S_Q(\us)
\end{aligned}
\end{equation}
by Proposition~\ref{prop: not vanishing}. Thus our assertion can be obtained by taking dual on~\eqref{eq: dual}.

\medskip

Recall Remark~\ref{rem: socle}. For $Q$ of type $E_{6,7,8}$, we can apply the same strategy of type $D_n$ and we shall show a non-trivial example below.
\end{proof}

\begin{example} \label{ex: dist1 but not minimal 2}
Consider the pair $\up=(111001,123212)$ and its socle $\us=(111111,122101,001001)$ in Example~\ref{ex: dist1 but not minimal}.
Since $(110000,001001)$ is a pair for $(111001)$, we have
\begin{equation} \label{eq: ex E proof}
\begin{aligned}
S_Q(123212)\conv S_Q(111001) & \hooklongrightarrow S_Q(123212)\conv S_Q(110000) \conv S_Q(001001) \\
& \hspace{8ex} \longtwoheadrightarrow S_Q(122101) \conv S_Q(111111) \conv S_Q(001001).
\end{aligned}
\end{equation}

Here the second surjection can be obtained by the following way: Since
\[
\text{(i) $(122101,001111)$ is a pair for $(123212)$}\quad \text{and} \quad  \text{(ii) $(110000,001111)$ is a pair for $(111111)$.}
\]
we have a non-zero composition
\begin{align*} S_Q(123212)\conv S_Q(110000) & \hooklongrightarrow S_Q(122101)\conv S_Q(001111)\conv S_Q(110000) \longtwoheadrightarrow S_Q(122101)\conv S_Q(111111).
\end{align*}
Since $S_Q(122101)\conv S_Q(111111)$ is $[Q]$-simple, the non-zero composition is surjective.

By taking dual on~\eqref{eq: ex E proof}, we have
\[
S_Q(\us) \hooklongrightarrow \overset{\to}{S}_Q(\up).
\]
\end{example}

\begin{corollary} \label{cor: simple iff simple}
For $\um \in \Z^{\ell(w_0)}_{\ge 0}$,
$ \overset{\to}{S}_{Q}(\um) \simeq \overset{\gets}{S}_{Q}(\um)$ is simple if and only if
$\um$ is $[Q]$-simple.
\end{corollary}

\begin{proof}
Our assertion is an immediate consequence of Corollary~\ref{cor: simple simple module} and Theorem~\ref{thm: socle of Q-pair}.
\end{proof}

\begin{theorem} \label{thm: num of distinct simples}
For non $[Q]$-simple pairs $\up=(\al,\be)$ and $\up'=(\al',\be')$ such that
\[
\text{$\up' \prec^\tb_Q \up $ and they are good neighbors,}
\]
we have an injective homomorphism
\[
\overset{\to}{S}_{Q}(\up') \hooklongrightarrow \overset{\to}{S}_{Q}(\up).
\]
Hence the composition length of $\overset{\to}{S}_{Q}(\up)$ for $\dist_Q(\up) \ge 1$ is lager than or equal to $\len_Q(\up)+2$.
\end{theorem}

\begin{proof}
If $\up$ and $\up'$ are good adjacent neighbors, we have a non-zero homomorphism $ \overset{\to}{S}_{Q}(\up') \to \overset{\to}{S}_{Q}(\up)$
by the same argument of the preceding theorem. Since their socles are isomorphic to each other and the socle appears once in their composition series,
the homomorphism should be injective. Then our first assertion follows from the definition of good neighbors.
More precisely, we have a sequences of injective homomorphisms (see Definition~\ref{def: good neighbor})
\[
\overset{\to}{S}_{Q}(\up')\hooklongrightarrow \overset{\to}{S}_Q(\up^{(1)}) \hooklongrightarrow \cdots \hooklongrightarrow \overset{\to}{S}_{Q}(\up).
\]
The second assertion follows from the first assertion and {\rm (d)} in Theorem~\ref{thm: BkMc}.
\end{proof}

\begin{corollary} \label{cor: l2}
For a pair $\up=(\al,\be)$, $\gdist_Q(\up)=1$ if and only if the module $\overset{\to}{S}_Q(\up)$ has a composition length $2$ such that
\[
[\overset{\to}{S}_Q(\up)] = [S_Q(\al) \hconv S_Q(\be)]+[S_Q(\al) \sconv S_Q(\be)].
\]
\end{corollary}

\begin{proof}
Our assertion follows from {\rm (iii)} of Theorem~\ref{thm: socle head}, {\rm (c)} of Theorem~\ref{thm: BkMc},
Theorem~\ref{thm: socle of Q-pair} and the previous theorem.
\end{proof}

\begin{example} \label{ex: dist1 but not minimal 3}
Recall that for the pair $\up=(111001,123212)$ in Example~\ref{ex: dist1 but not minimal}, we have $\dist_Q(\up)=1$ and $\gdist_Q(\up)=2$.
The module $\overset{\to}{S}_Q(\up)$ has its composition length larger than or equal to $4$ by the following argument: For sequences
$$\um=(111101,122111,001001) \quad \text{ and } \quad \um'=(011001,112101,111111),$$ we already observed that
\begin{itemize}
\item[{\rm (i)}] $\us=(001001, 122101, 111111) \prec_Q^{\tb} \begin{matrix} \um \\ \um' \end{matrix}  \prec_Q^{\tb} \up$
\item[{\rm (ii)}] $\um$ and $\um'$ are incomparable with respect to $\prec_Q^{\tb}$.
\end{itemize}
Furthermore $(122111,001001)$ and $(112101,111111)$ are $[Q]$-simple pairs. Thus one can check that $\overset{\to}{S}_Q(\um)$ and $\overset{\to}{S}_Q(\um')$
have their socles as $\overset{\to}{S}_Q(\us)$ and there exist non-zero homomorphisms $\overset{\to}{S}_Q(\um) \to \overset{\to}{S}_Q(\up)$
and $\overset{\to}{S}_Q(\um') \to \overset{\to}{S}_Q(\up)$. Since the socles of $\overset{\to}{S}_Q(\um)$, $\overset{\to}{S}_Q(\um')$ and $\overset{\to}{S}_Q(\up)$
are the same as $\overset{\to}{S}_Q(\us)$, the non-zero homomorphisms are injective indeed. Thus (i) the
composition length of $\overset{\to}{S}_Q(\up)$ is larger than or equal to $4$, (ii)
$\overset{\to}{S}_Q(\um)$ and $\overset{\to}{S}_Q(\um')$ have their composition length as $2$.
\end{example}
\section{Application to Quantum affine algebras} \label{Sec: Qunatum affine}
%\referee{"through" the author should explain if a direct proof
%(purely in terms of quantum groups) can be expected.}
%\sejin{Emphasize $E$-Dorey more as a simple socle of tensor products!!}
In this section, we first prove that our results in Section~\ref{Sec: KLR} hold also for the representation theory for quantum affine algebras through the KLR-type Schur-Weyl duality functor. Then we shall prove that the denominator formulas for $U_q'(A_n^{(1)})$ and $U_q'(D_n^{(1)})$
can be read from $\GQ $ for any Dynkin quiver $Q$ of type $A_n$ and $D_n$.
As an application, we can obtain Dorey's rule for $U_q'(E_{6,7,8}^{(1)})$ and partial information
of the denominator formulas for $U_q'(E_{6,7,8}^{(1)})$, which were not known to the best knowledge of the author. In the last part of
section, we shall propose a conjecture on {\it complete} denominator formulas for $U_q'(E_{6,7,8}^{(1)})$. In this section, we follow \cite{Kas02}
for notions on quantum affine algebras.

\subsection{Quantum affine algebras and categorifications}
Let $\cmA$ be a generalized Cartan matrix of affine type. %; i.e., $\cmA$ is positive semi-definite of corank $1$.
We choose $0 \in I\seteq \{ 0,1,\ldots,n\}$ as the leftmost vertices in the tables in~\cite[pages 54, 55]{Kac} except $A^{(2)}_{2n}$-case in which we take
the longest simple root as $\al_0$. We set $I_0 \seteq I \setminus \{ 0 \}$.
The weight
lattice $\wl$ associated to $\cmA$ is given as follows:
$$ \wl = \big( \soplus_{i \in I} \Z \Lambda_i \big) \oplus \Z \delta, $$
where $\delta \seteq \sum_{i \in I} d_i \al_i$ denotes \defn{the imaginary root}. We also denote by $c \seteq \sum_{i \in I}c_i h_i$ \defn{the center}.

Set $\mathfrak{h} = \mathbb{Q} \tens_\Z \wl^\vee$.
Then there exists a symmetric bilinear form $( \ , \ )$ on $\mathfrak{h}^*$ satisfying
$$ \langle h_i,\lambda \rangle = \dfrac{  2(\alpha_i,\lambda)}{(\alpha_i,\alpha_i)} \quad \text{ for any } i \in I \text{ and } \lambda \in \mathfrak{h}^*.$$
We normalize the bilinear form by
$$ \langle c,\lambda \rangle = (\delta,\lambda) \quad \text{ for any } \lambda \in \mathfrak{h}^*.$$

Let $\gamma$ be the smallest positive integer such that $\gamma \dfrac{(\al_i,\al_i)}{2} \in \Z$ for all $i \in I$. For each $i \in I$, set $q_i = q ^{\frac{(\al_i,\al_i)}{2}} \in \Q(q^{1/\gamma})$.
For $m,n \in \Z_{\ge 0}$ and $i \in I$, we define
$$ [n]_i= \dfrac{q_i^n-q_i^{-n}}{q_i-q_i^{-1}}, \ \  [n]_i!=\prod_{k=1}^n [k]_i, \ \  \left[ \begin{matrix} m \\ n \end{matrix} \right] = \dfrac{[m]_i!}{[m-n]_i![n]_i!}.$$

\begin{definition}
The \defn{quantum group $U_q(\g)$} associated to an affine Cartan matrix $\mathsf{A}$ %~\referee{What's this?? Change it!!}
is the algebra over $\Q(q^{1/\gamma})$ generated by $e_i, \ f_i$ $(i \in I)$ and $q^h$ $(h \in \gamma^{-1}P)$ subject to
the following relations:
\begin{enumerate}[(i)]
  \item  $q^0=1, q^{h} q^{h'}=q^{h+h'} $ for $ h,h' \in \gamma^{-1}P^{\vee},$
  \item  $q^{h}e_i q^{-h}= q^{\langle h, \alpha_i \rangle} e_i,
          \ q^{h}f_i q^{-h} = q^{-\langle h, \alpha_i \rangle }f_i$ for $h \in \gamma^{-1}P^{\vee}, i \in I$,
  \item  $e_if_j - f_je_i =  \delta_{ij} \dfrac{K_i -K^{-1}_i}{q_i- q^{-1}_i }, \ \ \text{ where } K_i=q_i^{ h_i},$
  \item  $\displaystyle \sum^{1-a_{ij}}_{k=0}
  (-1)^ke^{(1-a_{ij}-k)}_i e_j e^{(k)}_i =  \sum^{1-a_{ij}}_{k=0} (-1)^k
  f^{(1-a_{ij}-k)}_i f_jf^{(k)}_i=0 \quad \text{ for }  i \ne j, $
\end{enumerate}
where $e_i^{(k)}=e_i^k/[k]_i!$ and $f_i^{(k)}=f_i^k/[k]_i!$.
\end{definition}

We denote by $U_q^+(\g)$ (resp. $U_q^-(\g)$) by the subalgebra of $U_q(\g)$ generated by $e_i$ (resp. $f_i$) $(i \in I)$.
We also denote by $U_\A^+(\g)$ (resp. $U_\A^-(\mathsf{g})$) by the $\A\seteq\Z[q,q^{-1}]$-subalgebra of $U_q(\g)$ generated by $e^{(n)}_i$ (resp. $f^{(n)}_i$) for all $i \in I$ and $n\in \Z$.

We also denote by
\begin{itemize}
\item $\g$ the affine Kac-Moody algebra associated to $\mathsf{A}$,
\item $\g_0$ the subalgebra of $\g$ generated by $e_i,f_i$ and $h_i$ for $i \in I_0$ and
\item $U_q'(\g)$, called the \defn{quantum affine algebra}, the subalgebra of $U_q(\g)$ generated by $e_i$, $f_i$ and $K_i^{\pm 1}$ for all $i \in I$.
\end{itemize}
In this paper, we mainly deal with the quantum affine algebras $U_q'(\g)$.

For the rest of this paper, we take the field $\ko$,
the algebraic closure of $\C(q)$ in $\cup_{m >0} \C((q^{1/m}))$, as the base field of $U_q'(\g)$-modules.

%Note $U_q'(\g)$ has a comultiplication $\Delta$ given by
%$$\Delta(K_i)=K_i \otimes K_i, \ \Delta(e_i)=e_i\otimes K_i^{-1}+1 \otimes e_i , \ \Delta(f_i)=f_i\otimes 1+K_i \otimes f_i.$$

Set $\wl_\cl \seteq \wl / \Z\delta$.
%
% and call it the {\em classical weight lattice}. We say that a $U_q'(\g)$-module $M$ is {\em integrable} if
%\begin{itemize}
%\item[{\rm (i)}] it is $\wl_\cl$-graded; i.e.,
%$$M = \soplus_{\lambda \in \wl_\cl} M_\lambda \ \ \text{ where } M_\lambda=\{ u \in M \ | \ K_i u =q_i^{\langle h_i,\lambda \rangle} u \ \text{ for all } i \in I \},$$
%\item[{\rm (ii)}] for all $i \in I$, $e_i$ and $f_i$ act on $M$ locally nilpotently.
%\end{itemize}
%
We denote by $\mathcal{C}_\g$ the category of finite-dimensional integrable $U_q'(\g)$-modules.
A simple module $M$ in $\mathcal{C}_\g$ contains a non-zero vector $u$ of weight $\lambda\in \wl_\cl$ such that
\begin{itemize}
\item $\langle c,\lambda \rangle =0$ and $\langle h_i,\lambda \rangle \ge 0$ for all $i \in I_0$,
\item all the weights of $M$ are contained in $\lambda - \sum_{i \in I_0} \Z_{\ge 0} \cl(\alpha_i)$.
\end{itemize}
Such a $\lambda$ is unique and $u$ is unique up to a constant multiple.
We call $\lambda$ the \defn{dominant extremal weight} of $M$ and $u$ the \defn{dominant extremal weight vector} of $M$.

For $M \in \mathcal{C}_\g$ and $x$,
let $M_x$ be the $U_q'(\g)$-module with the actions of $e_i$, $f_i$, $K_i$ replaced with $x^{\delta_{i0}}e_i$, $x^{-\delta_{i0}}f_i$, $K_i$, respectively.
Then $M_x$ is contained in $\mathcal{C}_\g$ when $x$ in $\ko$ and isomorphic to $\ko[x,x^{-1}] \otimes M$ when $x$ is an indeterminate.

For each $i \in I_0$, we set
\begin{align} \label{eq: funda}
\varpi_i \seteq {\rm gcd}(c_0,c_i)^{-1}\cl(c_0\Lambda_i-c_i\Lambda_0) \in P^0_\cl
\end{align}
which is called a \defn{level $0$ fundamental weight}. Then $\{ \varpi_i \}_{i \in I_0}$
forms a basis of $$P_\cl^0\seteq \{ \lambda \in P_\cl \ | \  \langle c, \lambda \rangle =0 \} \subset P_\cl.$$

Then for each $i \in I_0$, there exists a unique simple $U_q'(\g)$-module $V(\varpi_i) \in \mathcal{C}_\g$ whose dominant extremal weight is $\varpi_i$ and which satisfies the following properties:
\begin{eqnarray} &&
\parbox{90ex}{
We can take $u_\mu \in u_\mu$ for each $\mu \in W\varpi_i \subset P_{{\rm cl}}$ such that
\begin{enumerate}
\item[{\rm (a)}] for $j \in I$ and $\mu \in W\varpi_i$ such that $\langle h_i,\mu \rangle \ge 0$, $e_j u_\mu =0$ and $f_j^{(\langle h_i,\mu \rangle)}u_\mu=u_{s_j\mu}$,
\item[{\rm (b)}] for $j \in I$ and $\mu \in W\varpi_i$ such that $\langle h_i,\mu \rangle < 0$, $f_j u_\mu =0$ and $e_j^{(-\langle h_i,\mu \rangle)}u_\mu=u_{s_j\mu}$,
\item[{\rm (c)}] $V(\varpi)$ is generated by $u_{\varpi_i}$,
\end{enumerate}
}\label{eq: strcuture funda}
\end{eqnarray}
where $W$ denotes the Weyl group generated by $s_i \in {\rm Aut}(P)$ $(i \in I)$
(see~\cite[\S 1.3]{AK} for more detail). We call the $V(\varpi_i)$ \defn{fundamental representation}.
%\referee{It should be reminded that for some papers
%these are particular cases of fundamental representations???}

Since $\mathcal{C}_\g$ is \defn{rigid},
there exist the left dual $V(\varpi_i)^*$ and the right dual ${}^*V(\varpi_i)$ of $V(\varpi_i)$ with the following
$U_q'(\g)$-homomorphisms
\begin{equation} \label{eq: dual affine}
 V(\varpi_i)^* \otimes V(\varpi_i)  \overset{{\rm tr}}{\longrightarrow} \ko \quad \text{ and } \quad
V(\varpi_i) \otimes {}^*V(\varpi_i)  \overset{{\rm
tr}}{\longrightarrow} \ko
\end{equation}
where
\begin{equation} \label{eq: p star}
V(\varpi_i)^*\seteq  V(\varpi_{i^*})_{ p^{-1}}, \
{}^*V(\varpi_i)\seteq  V(\varpi_{i^*})_{p} \ \  \text{ and } \ \
p \seteq (-1)^{\langle \rho^\vee ,\delta
\rangle}q^{(\rho,\delta)}.
\end{equation}
Here $\rho$ is defined by  $\langle h_i,\rho \rangle=1$,
$\rho^\vee$ is defined by $\langle \rho^\vee,\alpha_i  \rangle=1$ for all $i \in I$ and
$^*$ is the involution on $I_0$ in~\eqref{eq: * involution}.

We say that a $U_q'(\g)$-module $M$ is \defn{good} if it has a \defn{bar involution}, a crystal basis with \defn{simple crystal graph},
and a \defn{global basis} (see~\cite{Kas02} for precise definitions).
For instance, $V(\varpi_i)_x$ is a good module for every $i \in I$ and $x \in \ko^\times$. Note that
every good module is a simple $U_q'(\g)$-module and \defn{real}; i.e $M \otimes M$ is simple again.

\begin{definition}~\cite{HL11,KKKOIV} \label{def: Q module category}%~\referee{Rewrite this!! The definition of the category $\mathcal{C}^{(t)}_Q$ should be stated as a proper definition.}
%\sejin{it is enough?}
Let $Q$ be a Dynkin quiver of type $A_n$, $D_n$ or $E_{6,7,8}$, and
$U_q'(\g)$ be the quantum affine algebra of type $A^{(t)}_{n}$
, $D^{(t)}_{n}$ or $E_{6,7,8}^{(1)}$, respectively ($t=1,2$). For any positive root $\beta \in \PR$ associated to $Q$, we define the $U_q'(\g)$-module
$V_Q^{(t)}(\beta)$ in $\mathcal{C}_\g$ as follows: For $\phi_Q^{-1}(\beta,0)=(i,p)$, define
\begin{align} \label{V_Q(beta)}
V_Q^{(t)}(\beta) \seteq \begin{cases}
V(\varpi_i)_{(-q)^{p}} & \text{ if } t=1, \\
 V(\varpi_{i^\star})_{((-q)^{p})^\star} & \text{ if } t=2,
 \end{cases}
\end{align}
where $i^\star$ and $((-q)^{p})^\star$ are given as follows:
\begin{align*}
(i^\star,((-q)^{p})^\star) \seteq
\begin{cases}
(i,(-q)^{p}) & \text{ if } \g=A^{(2)}_n \text{ and } 1 \le i \le \left\lfloor \dfrac{n+1}{2} \right\rfloor, \\
(n+1-i,(-1)^n(-q)^{p}) & \text{ if } \g=A^{(2)}_n \text{ and } \left\lfloor \dfrac{n+1}{2} \right\rfloor \le i \le n , \\
(i,(\sqrt{-1})^{n-i}(-q)^{p}) & \text{ if } \g=D^{(2)}_n \text{ and } 1 \le i \le n-2, \\
(n-1,(-1)^i(-q)^{p})& \text{ if } \g=D^{(2)}_n \text{ and } n-1 \le i \le n.
\end{cases}
\end{align*}

We define the smallest abelian full subcategory
$\mathcal{C}^{(t)}_Q$ inside $\mathcal{C}_\g$ such that
\begin{itemize}
\item[({\rm a})] it is stable by taking %submodule,
subquotient, tensor product and extension,
\item[({\rm b})] it contains $V_Q^{(t)}(\beta)$ for all $\beta \in \PR$.
\end{itemize}
\end{definition}

%\Referee{About the following theorem, it seems the proof for untwisted types was given
%in [12] ? (then extended in [17] to the twisted cases).}
\begin{theorem} \label{thm: KOU}~\cite{HL11,KKKOIV} $($see also~\cite[Theorem 4.2.1]{KKK13B}$)$ Keeping the notations in Definition~\ref{def: Q module category},
there exist surjective homomorphism isomorphisms
\begin{align}
%\left[\mathcal{C}^{(1)}_Q \right] \simeq U^-_{\A}(\g_0)^{\vee}|_{q=1}  \simeq
%\left[\mathcal{C}^{(2)}_Q \right].
\Uppsi^{(t)}: U^-_{\A}(\g_0)^{\vee} \overset{\simeq}{\Lto} \left[\mathcal{C}^{(t)}_Q \right]
\label{eq:KOHLU}
\end{align}
whose kernels coincide with each other as $(q-1)U^-_{\A}(\g_0)^{\vee}$ . Moreover,
\begin{enumerate}
\item[{\rm (1)}] They send the dual canonical/upper global basis $\mathbf{B}^{{\rm up}}$ of $U^-_{\A}(\g_0)^{\vee}$ to
the set of the isomorphism classes in $\mathcal{C}^{(t)}_Q$ bijectively.
\item[{\rm (2)}] They send the dual root vector $\mathbf{F}_Q^{{\rm up}}(\beta)$ $(\beta \in \PR)$ of the PBW basis
$P_Q$ associated to $[Q]$ to $[V_Q^{(t)}(\beta)]$ of
$V_Q^{(t)}(\beta)$ in $\mathcal{C}^{(t)}_Q:$ %\sejin{notation??}
\begin{align} \label{eq: PBW 2}
\mathbf{F}_Q^{{\rm up}}(\beta) \longmapsto [V_Q^{(t)}(\beta)].
\end{align}
\end{enumerate}
\end{theorem}

\begin{remark}
The above theorem was proved in~\cite{HL11} for untwisted types. Then the author and his collaborators proved the above theorem when $\g$
is of a twisted type in~\cite{KKKOIV}.
\end{remark}

%\sejin{Remark for the history of the above theorem}

\subsection{Denominators and KLR-type Schur-Weyl duality}
For a good module $M$ and $N$, there exist a $U_q'(\g)$-homomorphism
$$ \Rnorm_{M,N}: M_{z_M} \otimes M_{z_N} \to \ko(z_M,z_N)  \otimes_{\ko[z_M^{\pm 1},z_N^{\pm 1}]} N_{z_N} \otimes M_{z_M} $$
such that
$$ \Rnorm_{M,N} \circ z_M = z_M \circ \Rnorm_{M,N}, \ \Rnorm_{M,N} \circ z_N = z_N \circ \Rnorm_{M,N} \text{ and }
\Rnorm_{M,N} (u_M \otimes u_N) = u_N \otimes u_M,$$
where $u_M$ (resp. $u_N$) is the dominant extremal weight vector of $M$ (resp, $N$).

The \defn{denominator} $d_{M,N}$ of $\Rnorm_{M,N}$ is the unique non-zero monic polynomial $d(u) \in \ko[u]$ of smallest degree such that
\begin{equation}\label{definition: dm,n}
d_{M,N}(z_N/z_M)\Rnorm_{M,N}(M_{z_M} \otimes N_{z_N}) \subset (N_{z_N} \otimes M_{z_M}).
\end{equation}

Then,
\begin{equation*}
\Rren_{M_{z_M},N_{z_N}}\seteq d_{M,N}(z_N/z_M)\Rnorm_{M_{z_M},N_{z_N}}
\colon M_{z_M} \otimes N_{z_N} \Lto
N_{z_N} \otimes M_{z_M}
\end{equation*}
is called the \defn{renormalized $R$-matrix}, and
$$\rmat{M,N}\seteq \Rren\vert_{z_M=z_N=1}: M \otimes N \Lto N \otimes M$$
is called the \defn{R-matrix}.

\begin{theorem}~\cite{AK,Chari,Kas02} $($see also~\cite[Theorem 2.2.1]{KKK13A}$)$  \label{Thm: basic properties} %\hfill
\begin{enumerate}
\item[{\rm (1)}] For good modules $M,N$, the zeroes of $d_{M,N}(z)$ belong to
$\cqm$ for some $m\in\Z_{>0}$.
\item[{\rm (2)}] Let $M_k$ be a good module
with a dominant extremal vector $u_k$ of weight $\lambda_k$, and
$a_k\in\ko^\times$ for $k=1,\ldots, t$.
Assume that $a_j/a_i$ is not a zero of $d_{M_i, M_j}(z) $ for any
$1\le i<j\le t$. Then the following statements hold.
\begin{enumerate}
\item[{\rm (i)}]
 $(M_1)_{a_1}\otimes\cdots\otimes (M_t)_{a_t}$ is generated by
$u_1\otimes\cdots \otimes u_t$.
\item[{\rm (ii)}] The head of
$(M_1)_{a_1}\otimes\cdots\otimes (M_t)_{a_t}$ is simple.
\item[{\rm (iii)}] Any non-zero
submodule of $(M_t)_{a_t}\otimes\cdots\otimes (M_1)_{a_1}$ contains the vector $u_t\otimes\cdots\otimes u_1$.
\item[{\rm (iv)}] The socle of $(M_t)_{a_t}\otimes\cdots\otimes (M_1)_{a_1}$
is simple.
\item[{\rm (v)}]
 Let
$$\rmat{}: (M_1)_{a_1}\otimes\cdots\otimes (M_t)_{a_t}
\to (M_t)_{a_t}\otimes\cdots\otimes (M_1)_{a_1}$$
 be  the specialization of $R^{{\rm norm}}_{M_1,\ldots, M_t}$
at $z_k=a_k$.  Then the image of $\rmat{}$ is simple and
it coincides with the head of
$(M_1)_{a_1}\otimes\cdots\otimes (M_t)_{a_t}$
and also with the socle of $(M_t)_{a_t}\otimes\cdots\otimes (M_1)_{a_1}$.
\end{enumerate}
\item[{\rm (3)}]
For a simple integrable $U_q'(\g)$-module M, there exists
a finite sequence $$\big((i_1,a_1),\ldots, (i_t,a_t)\big)$$
in $I_0\times \ko^\times$
such that
$d_{V(\varpi_{i_k}),V(\varpi_{ i_{k'} })}(a_{k'}/a_k) \seteq d_{i_k,i_{k'}}(a_{k'}/a_k)\not=0$ for $1\le k<k'\le t$ and
$M$ is isomorphic to the head of
$V(\varpi_{i_1})_{a_1}\otimes\cdots\otimes V(\varpi_{i_t})_{a_t}$.
Moreover, such a sequence $\big((i_1,a_1),\ldots, (i_t,a_t)\big)$
is unique up to permutation.
\item[{\rm (4)}] $d_{k,l}(z)=d_{l,k}(z)=d_{k^*,l^*}(z)=d_{l^*,k^*}(z)$ for $k,l \in I_0$.
\end{enumerate}
\end{theorem}

From the above theorem, one can notice that the denominator formulas between fundamental representations
provides crucial information of the representation theory on $\mathcal{C}_\g$. The denominator formulas between fundamental representations over classical quantum affine algebras were calculated in~\cite{AK,DO94,KKK13B,Oh14R}.

\medskip

Now we shall shortly review the KLR-type Schur-Weyl duality which was constructed in~\cite{KKK13A} and does depend only on
the denominator formulas.

\medskip

Let $\mathcal{S}$ be an index set. A \defn{Schur-Weyl datum} $\Xi$ is a quintuple %~\referee{What's this?? Change the notation!!}
\[
\Xi \seteq (U_q'(\g),J,X,s,\{V_s\}_{s \in \mathcal{S}})
\] consisting of
\begin{enumerate}
\item[{\rm (a)}] a quantum affine algebra $U_q'(\g)$,
\item[{\rm (b)}] an index set $J$,
\item[{\rm (c)}] two maps $X:J \to \ko^\times$ and $s: J \to \mathcal{S}$,
\item[{\rm (d)}] a family of good $U'_q(\g)$-modules $\{V_s\}$ indexed by $\mathcal{S}$.
\end{enumerate}

For a given $\Xi$, we define a quiver $\Gamma^{\Xi}=(\Gamma^{\Xi}_0,\Gamma^{\Xi}_1)$ in the following way:
\begin{enumerate}
\item $\Gamma^{\Xi}_0= J$.
\item For $i,j \in J$, we assign $\mathtt{d}_{ij}$ many arrows from $i$ to $j$, where $\mathtt{d}_{ij}$ is the order of the zero of
$d_{V_{s(i)},V_{s(j)}}(z_2/z_1)$ at $X(j)/X(i)$.
\end{enumerate}
We call $\Gamma^{\Xi}$ the \defn{Schur-Weyl quiver associated to $\Xi$}.

For a Schur-Weyl quiver $\Gamma^{\Xi}$, we have
\begin{itemize}
\item a symmetric Cartan matrix $\cmA^{\Xi}=(a^{\Xi}_{ij})_{i,j \in J}$ by
\begin{align} \label{eq: sym Cartan mtx}
a^{\Xi}_{ij} =  2 \quad \text{ if } i =j, \quad a^{\Xi}_{ij} = -\mathtt{d}_{ij}-\mathtt{d}_{ji} \quad \text{ if } i \ne j,
\end{align}
\item the set of polynomials $(\mathcal{Q}^{\Xi}_{i,j}(u,v))_{i,j \in J}$ such that
$\mathcal{Q}^{\Xi}_{i,j}(u,v) = (u-v)^{\mathtt{d}_{ij}}(v-u)^{\mathtt{d}_{ji}}$  if $i \ne j$.
\end{itemize}

We denote by $R^{\Xi}$ the symmetric KLR-algebra associated with $(\mathcal{Q}^{\Xi}_{i,j}(u,v))$.

\begin{theorem}~\cite{KKK13A} \label{thm:gQASW duality} For a given ${\Xi}$, there exists a functor
$$\F : \Rep(R^{\Xi}) \rightarrow \mathcal{C}_\g.$$
Moreover, $\F$ satisfies the following properties:
\begin{enumerate}
\item[{\rm (a)}] $\F$ is a tensor functor; that is, for any $M_1, M_2 \in \Rep(R^{\Xi})$, we have
$$\F(R^{\Xi}(0)) \simeq \ko \quad \text{ and } \quad \F(M_1 \circ M_2) \simeq \F(M_1) \tens \F(M_2).$$
%for any $M_1, M_2 \in \Rep(R^{\Xi})$.
\item[{\rm (b)}] If the Schur-Weyl quiver $\Gamma^{\Xi}$ is a Dynkin quiver of type $A_n$, $D_n$ or $E_{6,7,8}$, then $\F$ is exact.
\end{enumerate}
\end{theorem}
We call the functor $\F$ the \defn{KLR-type Schur-Weyl duality functor}.

%\Referee{About the following theorem, is the assignment of fundamental representations
%the same as in [12] ?}

\begin{theorem}~\cite{KKK13B,KKKOIV} \label{thm:gQASW duality 2}
Let $U_q'(\g)$ be a quantum affine algebra of type $A^{(t)}_{n}$  $($resp. $D^{(t)}_{n})$
and $Q$ be a Dynkin quiver of type $A_n$  $($resp. $D_n)$ for $t=1,2$.
Take $J$ and $\mathcal{S}$ as the set of simple roots $\Pi$ associated to $Q$.
Set $J$ as the set of simple roots $\Pi$ associated to $Q$.
We define two maps $$ s: \Pi \to \{ V(\varpi_i) \ | \ i \in I_0 \} \quad \text{ and } \quad X: \Pi \to  \ko^\times $$ as follows $\colon$
For $\al \in \Pi$ with $\phi_Q^{-1}(\al,0)=(i,p)$, we define
$$ s(\al)=\begin{cases} V(\varpi_i) & \text{ if } \g=A^{(1)}_{n} \text{ or } D^{(1)}_{n}, \\
V(\varpi_{i^\star}) & \text{ otherwise,} \end{cases} \ X(\al)= \begin{cases} (-q)^p  & \text{ if } \g=A^{(1)}_{n} \text{ or } D^{(1)}_{n}, \\
((-q)^p)^\star & \text{ otherwise,} \end{cases}$$
which are induced by the maps $\phi_Q$ and $\star$.
\begin{itemize}
\item[({\rm a})] The quiver $\Gamma^{\Xi}$ is isomorphic to $Q^{{\rm rev}}$. Hence the functor
$$\F^{(t)}_Q: \Rep(R^{\Xi}) \rightarrow \mathcal{C}^{(t)}_Q \quad (t =1,2) \text{ in {\rm Theorem~\ref{thm:gQASW duality}} is exact.}$$
\item[({\rm b})] The functor $\F^{(t)}_Q$ sends simples to simples, bijectively. In particular, $\F^{(t)}_Q$ sends $S_Q(\beta)$ to $V^{(t)}_Q(\beta)$.
\item[({\rm c})] The functors $\F^{(1)}_Q$ and $\F^{(2)}_Q$ induce the ring isomorphisms between $[\mathcal{C}^{(1)}_Q]$ and $[\mathcal{C}^{(2)}_Q]$.
Moreover, the induced functor between $\mathcal{C}^{(1)}_Q$ and $\mathcal{C}^{(2)}_Q$ preserves dimensions and sends simples to simples, bijectively.
\end{itemize}
\end{theorem}

%By applying the exact functors in Theorem~\ref{thm:gQASW duality 2}, the results in previous %section hold for $\mathcal{C}_Q^{(t)}$
%of affine types $A^{(t)}_n$ and $D^{(t)}_n$ also.
%\ber \beb I CAN PROVE THE DOREY'S RULE\eb
%The following corollary which can be considered as Dorey's rule for $E_{6,7,8}^{(1)}$ types:

Now we shall focus on $U'_q(E_{6,7,8}^{(1)})$ for next two propositions, which are also applicable for $U'_q(A_n^{(t)})$ and $U'_q(D_n^{(t)})$ $(t=1,2)$:

%\Referee{Is there a relation with the formulas established in Section 7.3
%of [Frenkel, Edward; Hernandez, David Baxter's relations and spectra
%of quantum integrable models. Duke Math. J. 164 (2015), no. 12,
%2407--2460]? }

%\Sejin{Yes!}
%\Sejin{Emphasize $E$-Dorey's More!! Maybe Add more Appendix}

\begin{proposition} \label{eq: Dorey's step1}
For a fixed $\redez \in [Q]$ and a $[Q]$-minimal pair $(\al,\be)$ of a simple sequence $\us$, we have an injective $U_q'(E_{6,7,8}^{(1)})$-homomorphism$:$
$$  V^{(1)}_Q(\be^\redez_1)^{\tens s_1} \tens \cdots \tens
V^{(1)}_Q(\be^\redez_{\N})^{\tens s_{\N}} \hooklongrightarrow V^{(1)}_Q(\al) \otimes V^{(1)}_Q(\beta).$$
\end{proposition}

\begin{proof}
Since the composition length of
$S_Q(\al) \conv S_Q(\beta)$ is $2$, we have
$$ [S_Q(\al) \conv S_Q(\beta)]= q^x[S_Q(\al) \hconv S_Q(\beta)] + q^y[\overset{\Lto}{S_Q}(\us) ] \text{ for some } x,y \in \Z.$$
Thus, under the isomorphism $\Uppsi^{(1)}$ in~\eqref{eq:KLRU}, we have
\[
\FQup(\al)\FQup(\be)= q^{x'} \mathbf{b} + q^{y'} \prod_{i=1}^\N \FQup(\be^\redez_i)^{\us_i} \quad \text{ for some } x',y' \in \Z  \text{ and } b \in \mathbf{B}^{{\rm up}}.
\]
Note that
\begin{itemize}
\item[{\rm (i)}] $\prod_{i=1}^\N \FQup(\be^\redez_i)^{\us_i}$ is contained in $\mathbf{B}^{{\rm up}}_\Z \seteq\bigsqcup_{n\in\Z}q^n\mathbf{B}^{{\rm up}}$,
\item[{\rm (ii)}] $\FQup(\al)\FQup(\be)$ and $\prod_{i=1}^\N \FQup(\be^\redez_i)^{\us_i}$
correspond to $V^{(1)}_Q(\al) \otimes V^{(1)}_Q(\beta)$ and $\tens_{i=1}^{\N}  V^{(1)}_Q(\be^\redez_i)^{\tens \us_i}$ under the surjection $\Uppsi^{(1)}$ in~\eqref{eq:KOHLU}, respectively.
\end{itemize}
Thus, we can conclude that
\begin{itemize}
\item $V^{(1)}_Q(\al) \otimes V^{(1)}_Q(\beta)$ has composition length $2$,
\item $\tens_{i=1}^{\N}  V^{(1)}_Q(\be^\redez_i)^{\tens \us_i}$ is the simple socle of $V^{(1)}_Q(\al) \otimes V^{(1)}_Q(\beta)$, by the quantum affine algebra version of Theorem~\ref{thm: socle head}
(see also the argument in the proof of \cite[Theorem 4.3.4]{KKK13B}). \qedhere
\end{itemize}
\end{proof}

\begin{theorem} \label{thm: Dorey's step2}
For any pair $\up=(\al,\be)$ and a Dynkin quiver $Q$ of type $E_{6,7,8}$, we have a $U_q'(E_{6,7,8}^{(1)})$-module isomorphism:
$$ \soc\left( V^{(1)}_{Q}(\al) \otimes V^{(1)}_{Q}(\be) \right)\simeq V^{(1)}_Q(\soc_Q(\up)).$$
In particular, if $\al+\be=\ga \in \PR$, we have an injective  $U_q'(E_{6,7,8}^{(1)})$-module homomorphism:
$$ V^{(1)}_{Q}(\ga) \hooklongrightarrow  V^{(1)}_{Q}(\al) \otimes V^{(1)}_{Q}(\be).$$
\end{theorem}

\begin{proof}
By Proposition~\ref{eq: Dorey's step1} and the quantum affine algebra version of Proposition~\ref{prop: not vanishing}, we can apply the same argument of Theorem~\ref{thm: socle of Q-pair with al+be in PR}, by replacing
$S_Q(\eta)$ in the proof with $V^{(1)}_{Q}(\eta)$ $(\eta\in \PR)$.
\end{proof}

%\begin{remark}
The morphisms in
\begin{align}\label{eq: Dorey}
{\rm Hom}_{U_q'(\g)}\big( V(\varpi_k)_c, V(\varpi_i)_a \tens V(\varpi_j)_b\big) \quad \text{ for } i,j,k\in I_0 \text{ and } a,b,c \in \ko^\times
\end{align}
are studied by~\cite{CP96,FH15,KKK13B,Oh14R,YZ11} and called \defn{Dorey's type} morphisms.
For classical untwisted types $A_n^{(1)}$,
$B_n^{(1)}$, $C_n^{(1)}$ and $D_n^{(1)}$, such morphisms are studied by Chari-Pressley using (twisted) Coxeter elements.
By the quantum affine algebra version of Theorem~\ref{thm: socle head}, Equation~\eqref{eq: Dorey} can be interpreted as follows:
\begin{enumerate}
\item The dimension of Hom-space is $1$ if it is not trivial,
\item $V(\varpi_k)_c$ is the simple socle of $V(\varpi_i)_a \tens V(\varpi_j)_b$.
\end{enumerate}
Thus Theorem~\ref{thm: Dorey's step2} implies the Dorey's type morphisms for $U_q(E_{6,7,8}^{(1)})$.
In Appendix~\ref{Sec:Dynkin E6} and Appendix~\ref{Sec:Dynkin E7}, we list several Dorey's type morphisms for $U_q(E_6^{(1)})$ and $U_q(E_7^{(1)})$, respectively.

\begin{remark} \hfill
\begin{enumerate}
\item In \cite{YZ11}, Young and Zeger studied the necessary condition (\cite[Theorem 3.1]{YZ11}) of Dorey's rule
by using $q$-character.  Theorem \ref{thm: Dorey's step2} implies that
the condition  is also sufficient for  $E_{6,7,8}^{(1)}$ as it is the necessary and sufficient condition for $A_{n}^{(1)}$ and $D_{n}^{(1)}$
(see also \cite[Theorem 6.1, Theorem 7.1]{CP96}, \cite[Corollary 3.5]{Oh14A} and
\cite[Proposition 4.4]{Oh14D}). For instance, their formula given in \cite[Eq.~(3.10)]{YZ11} can be refined as
\[
 V(\varpi_4)_{q^0} \hooklongrightarrow  V(\varpi_1)_{-q^{3}} \otimes V(\varpi_6)_{q^{-2}},
\]
by Theorem \ref{thm: Dorey's step2} and the morphism
$$
S_Q\left({\scriptstyle\prt{111}{101}}\right) \hooklongrightarrow  S_Q\left({\scriptstyle\prt{100}{000}}\right) \conv S_Q\left({\scriptstyle\prt{011}{101}}\right).
$$
Here, $Q$ is of the Dynkin quiver in Appendix~\ref{Sec:Dynkin E6}. Note that $q$ in \cite{YZ11} corresponds to $(-q)$ in this paper.
\item The Dorey's type morphisms for $E_{6,7,8}^{(1)}$ were also studied in \cite{FH15} also.
In the paper, Frenkel and Hernandez used $q$-characters and $T$-systems to show that
all fundamental representation appears as a composition factor of tensor product of fundamental representations $V(\varpi_i)$'s corresponding extremal vertices $i$.
More precisely, they computed several Dorey's type morphisms related to $V(\varpi_i)$ when $i$ is extremal. For instance, they showed that, for $U_q'(E_6^{(1)})$, (see \cite[page 2452]{FH15})
\[
V(\varpi_6)_{q^0} \hooklongrightarrow  V(\varpi_1)_{-q^{3}} \otimes V(\varpi_5)_{q^{-3}},
\]
which corresponds to
\[
S_Q\left({\scriptstyle\prt{111}{000}}\right) \hooklongrightarrow  S_Q\left({\scriptstyle\prt{100}{000}}\right) \conv S_Q\left({\scriptstyle\prt{011}{000}}\right)
\]
where $Q$ is of the Dynkin quiver in Appendix~\ref{Sec:Dynkin E6}.
\end{enumerate}
\end{remark}

\begin{example} For the AR-quiver $\Gamma_Q$ in Appendix~\ref{Sec:Dynkin E6}, one can check that
\[
\left( \prt{010}{000},\prt{111}{111},\prt{111}{100} \right) = \soc_Q\left(\prt{110}{000},\prt{122}{211}\right).
\]
Thus we have
\[
V(\varpi_1)_{-q^{1}} \otimes V(\varpi_1)_{-q^{5}} \otimes V(\varpi_5)_{-q^{3}} \hookrightarrow  V(\varpi_2)_{q^{6}} \otimes V(\varpi_2).
\]
via the KLR-type Schur-Weyl functor.
\end{example}

Now, the results in previous section also hold for $\mathcal{C}_Q^{(t)}$ of affine types $A^{(t)}_n$, $D^{(t)}_n$ $(t=1,2)$ and $E_{6,7,8}^{(1)}$.

\begin{theorem} \label{thm: for C_Q} Let $Q$ be a Dynkin quiver of type $A_n$, $D_n$ $(t = 1,2)$ and $E_{6,7,8}$ $(t=1)$, and $\um_{\redez} \in \Z^{\N}_{\ge 0}$ for a fixed $\redez \in [Q]$.
\begin{enumerate}
\item[{\rm (1)}] $\overset{\Lto}{V^{(t)}_Q}(\um) \seteq V^{(t)}_Q(\be^\redez_1)^{\tens m_1} \tens \cdots \tens
V^{(t)}_Q(\be^\redez_{\N})^{\tens m_{\N}}$ is well-defined.
%\item[{\rm (2)}] $\big[\overset{\Lto}{V^{(t)}_{Q}}(\um)\big]
%\in [\F^{(t)}_Q\big({\rm Im}(\rmat{\um})\big)] + \displaystyle \sum_{\um' \prec^{\tb}_{Q} \um} %\Z_{\ge 0} [\F^{(t)}_Q\big({\rm Im}(\rmat{\um'})\big)].$
\item[{\rm (2)}] For any pair $\up=(\al,\be)$, $\soc\left(\overset{\Lto}{V^{(t)}_{Q}}(\up)\right)\simeq V^{(t)}_Q(\soc_Q(\up))$.
\item[{\rm (3)}] $\overset{\Lto}{V^{(t)}_Q}(\um) \simeq \overset{\Lgets}{V^{(t)}_Q}(\um)$ is simple if and only if $\um$ is $[Q]$-simple.
\item[{\rm (4)}] For any pair $\up=(\al,\be)$ with $\dist_Q(\up) \ge 1$, the composition length of $\overset{\Lto}{V^{(t)}_{Q}}(\up)$ is larger than or equal to $\len_Q(\up)+2$.
In particular, if $\gdist_Q(\up)=1$, $\overset{\Lto}{V^{(t)}_{Q}}(\up)$ has equal composition length $2$.
\end{enumerate}
Furthermore, if $Q$ is of type $A_n$ or $D_n$, we have
\begin{enumerate}
\item[{\rm (5)}] $\big[\overset{\Lto}{V^{(t)}_{Q}}(\um)\big]
\in [\F^{(t)}_Q\big({\rm Im}(\rmat{\um})\big)] + \displaystyle \sum_{\um' \prec^{\tb}_{Q} \um} \Z_{\ge 0} [\F^{(t)}_Q\big({\rm Im}(\rmat{\um'})\big)],$
\end{enumerate}
by applying the exact functors in Theorem~\ref{thm:gQASW duality 2}.
\end{theorem}

\begin{proof}
For $A^{(t)}_n$, $D^{(t)}_n$ $(t=1,2)$, our assertion follows from the exact functors in Theorem~\ref{thm:gQASW duality 2}.
For $E^{(1)}_{6,7,8}$, our assertion follows from Theorem~\ref{thm: Dorey's step2}.
\end{proof}

\begin{example} \label{ex: non-trival}
For the AR-quiver $\Gamma_Q$ in Section~\ref{Sec:Dynkin E6} of type $E_6$, there exists a homomorphism
$$  S_Q\left({\scriptstyle\prt{123}{211}}\right) \hooklongrightarrow  S_Q\left({\scriptstyle\prt{122}{101}}\right) \conv S_Q\left({\scriptstyle\prt{001}{110}}\right)$$
where $\left( {\scriptstyle\prt{122}{101}},{\scriptstyle\prt{001}{110}} \right)$ is a pair for ${\scriptstyle\prt{123}{211}}$. Note that this pair is {\it not} a minimal pair of $\prt{123}{211}$.
By Theorem~\ref{thm: for C_Q},  we have a $U_q'(E_6^{(1)})$-module homomorphism
$$  V(\varpi_3)_{q^0} \hooklongrightarrow  V(\varpi_3)_{q^{4}} \otimes V(\varpi_3)_{q^{-4}},$$
by considering the coordinates for ${\scriptstyle\prt{123}{211}}$, ${\scriptstyle\prt{122}{101}}$ and
${\scriptstyle\prt{001}{110}}$ in $\Gamma_Q$. %\sejin{non-trivial examples}
\end{example}

\begin{remark} \label{rem: important observation}
From Theorem~\ref{Thm: basic properties} and Theorem~\ref{thm: for C_Q}, we can observe the following:

For a Dynkin quiver $Q$ of type $A_n$, $D_n$ $(t \in \{ 1,2\})$ and $E_{6,7,8}$ $(t=1)$,
\begin{equation*} \label{eq: important observation}
\begin{aligned}
& \text{ the pair } (\al,\be) \text{ is {\it not} $[Q]$-simple} \longleftrightarrow  S_Q(\al) \conv S_Q(\be) \text{ is reducible }    \longleftrightarrow V^{(t)}_Q(\al) \tens V^{(t)}_Q(\be) \text{ is reducible}  \\
& \longleftrightarrow \begin{cases} d_{k,l}(z)  \\
d_{k^\star,l^\star}(z) \end{cases} \hspace{-4ex} \text{ has a root
at one of } \begin{cases}
(-q)^a/(-q)^b \text{ and } (-q)^b/(-q)^a & \hspace{-2ex} \text{ if } t=1, \\
((-q)^a)^\star/((-q)^b)^\star \text{ and } ((-q)^b)^\star/((-q)^a)^\star & \hspace{-2ex} \text{ if } t=2, \end{cases}
\end{aligned}
\end{equation*}
where $\phi_Q^{-1}(\al,0)=(k,a)$ and $\phi_Q^{-1}(\be,0)=(l,b)$. Thus we can obtain the {\it partial} information of
the denominator formulas for $E_{6,7,8}^{(1)}$ by Theorem~\ref{thm: for C_Q}. More precisely,
we can get the roots of $d_{k,l}(z)$ by reading the position of non $[Q]$-simple pairs $(\al,\be)$ in $\GQ$ whose
residues are $k,l$. But we do {\it not} know the order of roots (see Corollary~\ref{cor: dist denom} and the conjectural formulas in Section~\ref{subsec: E conjectures} below).
\end{remark}

Note that the Schur-Weyl datum in Theorem~\ref{thm:gQASW duality 2} and the functors $\mathcal{F}_Q^{(1)}$ (resp. $\mathcal{F}_Q^{(2)}$)  are determined by
$(U_q'(\g),\phi^{-1}_Q)$ (resp. $(U_q'(\g),\phi^{-1}_Q,^\star)$) indeed.

\begin{conjecture} \label{conj: KKK E}~\cite[Conjecture 4.3.2]{KKK13B} Let $Q$ be a Dynkin quiver of type $E_{6,7,8}$. Then
the functor $\mathcal{F}_Q^{(1)}$ determined by $(U_q'(E^{(1)}_{6,7,8}),\phi^{-1}_Q)$ enjoys the same properties in {\rm Theorem~\ref{thm:gQASW duality 2}}.
\end{conjecture}

By the lack of denominator formulas for the quantum affine algebra of type $E^{(1)}_{6,7,8}$, we do not know whether Conjecture~\ref{conj: KKK E} holds or not.

\subsection{Distance polynomials on AR-quivers} In this subsection, we denote by $U'_q(\g)$ the quantum affine algebra of untwisted type
$A^{(1)}_n$, $D^{(1)}_n$ or $E^{(1)}_{6,7,8}$, and $Q$ the Dynkin quiver of type $A_n$, $D_n$ or $E_{6,7,8}$, respectively. Thus we take index sets
$I=\{ 0,1,\ldots,n\}$ and $I_0 = I \setminus \{ 0\}$, and the set of positive roots $\PR$ associated to $Q$.

\medskip

Now we shall develop the observation in Remark~\ref{eq: important observation}; that is, under the situation in Remark~\ref{eq: important observation}, we shall
interpret the meaning of $\gdist_Q(\al,\be)$ in the denominator formula $d_{k,l}(z)$.

%\begin{definition}
%For a Dynkin quiver $Q$, indices $k,l \in I_0$ and an integer $t \in \mathbb{N}$, we define the subset $\Phi_{Q}(k,l)[t]$ of $\PR \times \PR$ as follows:
%A pair $(\alpha,\beta)$ is contained in $\Phi_{Q}(k,l)[t]$ if
%\begin{enumerate}
%\item[{\rm (1)}] $\alpha \prec_Q \beta$ or $\be \prec_Q \al$,
%\item[{\rm (2)}] $\{ \phi_Q^{-1}(\al,0),\phi_Q^{-1}(\be,0) \}=\{ (k,a), (l,b)\}$ such that $|a-b|=t.$
%\end{enumerate}
%\end{definition}

For a Dynkin quiver, indices $k,l \in I_0$ and an integer $t \in \Z_{\ge 1}$,
we define the subset \defn{$\Phi_{Q}(k,l)[t]$} as the pairs $(\alpha,\beta) \in \PR \times \PR$ such that $\alpha$ and $\beta$ are {\it comparable} under $\prec_{Q}$ and
\[
 \{ \phi_Q^{-1}(\al,0),\phi_Q^{-1}(\be,0) \}=\{ (k,a), (l,b)\} \quad \text{ such that } \quad |a-b| = t.
\]

\begin{lemma}
For any  $(\alpha^{(1)},\beta^{(1)})$ and $(\alpha^{(2)},\beta^{(2)})$ in $\Phi_{Q}(k,l)[t]$, we have
$$\gdist_Q(\alpha^{(1)},\beta^{(1)})=\gdist_{Q}(\alpha^{(2)},\beta^{(2)}). $$
\end{lemma}

\begin{proof}
Note that for $\al \prec_Q \ga \prec_Q \be$,
$$\text{$\tau_Q(\be) \in \PR$ implies $\tau_Q(\ga) \in \PR$ and $\tau^{-1}_Q(\al) \in \PR$ implies $\tau^{-1}_Q(\ga) \in \PR$.}$$

Since, for $(\alpha^{(1)},\beta^{(1)}),(\alpha^{(2)},\beta^{(2)}) \in \Phi_{Q}(k,\ell)[t]$, there exists
$s \in \Z$ such that $\tau_Q^{s}(\al^{(1)}) =\al^{(2)}$ and $\tau_Q^{s}(\be^{(1)}) =\be^{(2)}$. Thus our assertion follows.
\end{proof}
Thus the integer
\[
o^{Q}_t(k,l) \seteq \gdist_Q(\alpha,\beta)  \quad \text{ for any $(\alpha,\beta) \in \Phi_{Q}(k,l)[t]$}
\]
is well-defined.

By using $Q^*,Q^\rev$ and $Q^{*\rev}$ of a Dynkin quiver $Q$ (Remark~\ref{rem: coxeter element and sink}), one can check the following lemma:

\begin{lemma} \label{lem: property of o}
For $k,\ell \in I_0$ and $t \in \Z$, we have
\begin{enumerate}
\item[{\rm (a)}] $o^{Q}_t(k,l)=o^{Q^\rev}_t(l^*,k^*)$ and $o^{Q^*}_t(k,l)=o^{Q^{*\rev}}_t(l^*,k^*)$.
\item[{\rm (b)}] $o^{Q}_t(k,l)=o^{Q^{*\rev}}_t(l,k)$ and $o^{Q^*}_t(k,l)=o^{Q^{\rev}}_t(l,k)$.
\item[{\rm (c)}] $ o^{Q}_t(k,l) = o^{Q}_t(l,k)$ and $o^{Q}_t(k,l)=o^{Q^*}_t(k^*,l^*)$.
\end{enumerate}
\end{lemma}

\begin{definition} \label{def: Dist poly Q}
For $k,l \in I_0$ and a Dynkin quiver $Q$, we define a polynomial $D^Q_{k,l}(z;-q) \in \ko[z]$
\[
D^Q_{k,l}(z;-q) \seteq  \prod_{ t \in \Z_{\ge 0} } (z-(-q)^{t} )^{\mathtt{o}^{\overline{Q}}_t(k,l)}
\]
where
\[
\mathtt{o}^{\overline{Q}}_t(k,l) \seteq  \max( o^{Q}_t(k,l),o^{Q^\rev}_t(k,l) ).
\]
\end{definition}

Note that
\begin{eqnarray} &&
\parbox{90ex}{
\begin{enumerate}
\item[{\rm (a)}] $\mathtt{o}^{\overline{Q}}_t(k,l)=0$ for $0 \le t < \Delta(k,l)+2$,
\item[{\rm (b)}] $D^Q_{k,l}(z;-q)=D^Q_{l,k}(z;-q)=D^Q_{k^*,l^*}(z;-q)=D^Q_{l^*,k^*}(z;-q)$,
\end{enumerate}
}\label{eq: observation of multi}
\end{eqnarray}
by {\rm (a)} of Proposition~\ref{prop: dir Q cnt} and Lemma~\ref{lem: property of o}.
%the fact that $\cmA$ is symmetric.

\begin{proposition} \label{prop: DQ DQ'}
For $k,l \in I_0$ and any Dynkin quivers $Q$ and $Q'$, we have
$$D^Q_{k,l}(z;-q)=D^{Q'}_{k,l}(z;-q).$$
\end{proposition}

\begin{proof}
By~\eqref{eq: adapted} and~\eqref{eq: additive}, for Dynkin quivers $Q$ and $Q'$, we have
$$o^{Q}_t(k,l)=o^{Q'}_t(k,l) \quad \text{ if } \Phi_{Q}(k,l)[t],\Phi_{Q'}(k,l)[t] > 0.$$

To prove our assertion, it suffices to prove that
$$ \mathtt{o}^{\overline{Q}}_t(k,l)=\mathtt{o}^{\overline{s_i(Q)}}_t(k,l) \quad \text{ for every source $i$ of $Q$}, $$
since every Dynkin quiver $Q'$ can be obtained by applying sequence of the reflection functors $\mathbf{r}^-_i$ on $Q$.

Assume that $a=\mathtt{o}^{\overline{Q}}_t(k,l)>0$; that is, there exists a $(\al,\be) \in \Phi_{Q}(k,l)[t]$ or $\Phi_{Q^\rev}(k,l)[t]$ such that $\gdist_Q(\al,\be)=a$
or $\gdist_{Q^\rev}(\al,\be)=a$. Set
$$\text{$(\al,\be) \in \Phi_{Q}(k,l)[t]$, $\phi^{-1}_{Q,1}(\al)=k$, $\phi^{-1}_{Q,1}(\be)=l$ and $\al \prec_Q \be$.}$$

If $\al$ is not a simple root, then $s_i(\al),s_i(\be) \in \PR$ for every source $i$ of $Q$ and
hence $\gdist_Q(\al,\be)=\gdist_{\mathbf{r}^-_iQ}(s_i(\al),s_i(\be))=a$.
Thus $\mathtt{o}^{\overline{Q}}_t(k,l)=\mathtt{o}^{\overline{s_i(Q)}}_t(k,l)$ for every source $i$ of $Q$.

If $|\Phi_{Q}(k,l)[t]|>1$, then we have $\mathtt{o}^{\overline{Q}}_t(k,l)=\mathtt{o}^{\overline{s_i(Q)}}_t(k,l)$
for every source $i$ of $Q$ by the same reason.

\medskip

Now assume that
\begin{align} \label{eq: assu0}
\text{(1) $k$ is a source of $Q$ and $\al = \al_{k}$},  \quad  \text{(2) $\Phi_{Q}(k,l)[t]=\{ (\al_s,\be) \}$ and $\mathtt{o}^{Q}_t(k,l)>0$}.
\end{align}

Since $(\al,\be)$ is not sectional,~\eqref{eq: assu0} implies that
\begin{eqnarray} &&
\parbox{85ex}{
\begin{itemize}
\item $\tau_Q(\al_k) \in \PR$ and $\be=\theta^Q_{l^*}$ (since $|\Phi_{Q}(k,l)[t]|=1$),
\item $\phi^{-1}_{Q,1}(\al)-\phi^{-1}_{Q,1}(\be)=|k-l|+2s=t$ for some $1 \le s \le \mQ_l$ (since $\mathtt{o}^{Q}_t(k,l)>0$).
\end{itemize}
}\label{eq: assu}
\end{eqnarray}

($Q$ is of type $A_n$, $k \ne l^*$) In this case, the Dynkin quiver $Q$ is the one of the following forms:
\begin{align} \label{eq: unique}
Q: \ (a) \ \xymatrix@R=3ex{*{ }<3pt>
\ar@{<.}[rr]  &&*{\circ}<3pt>
\ar@{-}[r]_<{k} &*{\circ}<3pt>
\ar@{<-}[l]^<{\ \ k+1} \ar@{<.}[rr] &&*{\circ}<3pt>
\ar@{}[l]^<{\ \ l^*} } \text{ or } \quad  \ (b) \
\xymatrix@R=3ex{*{\circ}<3pt>
\ar@{}[r]_<{l^*}  &&*{\circ}<3pt>
\ar@{<.}[ll]^<{k-1}\ar@{<-}[r] &*{\circ}<3pt>
\ar@{->}[r]_<{k} &*{\circ}<3pt>
\ar@{}[l]^<{\ \ k+1} }.
\end{align}
by~\eqref{eq: gamma,theta 2} and~\eqref{eq: Nakayama}. Hence we have $\ga=\al_k+\theta^Q_{l^*} \in \PR$, $\mul(\ga)=1$ and $\dist_Q(\al_k,\be)=1$.

Otherwise,
$$\supp(\theta^Q_{l^*}) \cap \{k\} =  \emptyset  \quad\text{ and }\quad \theta^Q_{l^*}+\al_k  \not\in \PR $$
and we have a contradiction to the assumption that $\gdist_Q(\al,\be) >0$.

By Proposition~\ref{prop: Q-socle for pair}, $\al$ and $\be$ are located at one of the following form:

\vskip -2em

\begin{align} \label{eq: ext form}
\scalebox{0.8}{{\xy
(0,0)*{}="DL";(10,-10)*{}="DD";(20,20)*{}="DT";(30,10)*{}="DR";
"DL"; "DD" **\dir{-};"DL"; "DT"+(-4,-4) **\dir{.};
"DT"+(-40,-4); "DT"+(70,-4)**\dir{.};
"DD"+(-30,-6); "DD"+(80,-6) **\dir{.};
"DT"+(4,-4); "DR" **\dir{.};"DR"; "DD" **\dir{-};
"DL"+(-15,0)*{\scriptstyle (l,\xiQ_l-2\mQ_{l})};
%"DL"+(10,0)*{{\rm (a)}};
"DL"+(-1,0); "DL"+(-6,0) **\dir{.};
"DR"+(10,0)*{\scriptstyle (k,\xiQ_k)};
"DR"+(1,0); "DR"+(6,0) **\dir{.};
"DL"*{\bullet};"DR"*{\bullet};
%"DT"+(4,-4)*{\bullet};"DT"+(-4,-4)*{\bullet};"DD"*{\bullet};
%"DD"+(0,-2)*{\scriptstyle (k,u)};
%"DT"+(-8,0)*{\scriptstyle (1,s+(i-1))};"DT"+(9,0)*{\scriptstyle (1,t-(j-1))};
"DT"+(-44,-4)*{\scriptstyle 1};
"DT"+(-44,-8)*{\scriptstyle 2};
"DT"+(-44,-12)*{\scriptstyle \vdots};
"DT"+(-44,-16)*{\scriptstyle \vdots};
"DT"+(-44,-36)*{\scriptstyle n};
"DL"+(40,-10); "DD"+(36,-6) **\dir{.};
"DR"+(40,-10); "DD"+(44,-6) **\dir{.};
"DT"+(40,-10); "DR"+(40,-10) **\dir{-};
"DT"+(40,-10); "DL"+(40,-10) **\dir{-};
"DL"+(25,-10)*{\scriptstyle (l,\xiQ_l-2\mQ_l)};
"DL"+(39,-10); "DL"+(34,-10) **\dir{.};
"DR"+(50,-10)*{\scriptstyle (k,\xiQ_k)};
"DR"+(41,-10); "DR"+(46,-10) **\dir{.};
"DL"+(40,-10)*{\bullet};"DR"+(40,-10)*{\bullet};
%"DT"+(40,-10)*{\bullet};
%"DT"+(40,-8)*{\scriptstyle (k,u)};
%"DT"+(40,-10)*{\bullet};
%"DD"+(44,-6)*{\bullet};"DD"+(36,-6)*{\bullet};
%"DL"+(60,0)*{{\rm(b)}};
%"DD"+(30,-10)*{\scriptstyle (n,s+(n-i))};
%"DD"+(51,-10)*{\scriptstyle (n,t-(n-j))};
%
\endxy}}
\end{align}

Then the Dynkin quivers $s_k(Q)$ and $(s_k(Q))^\rev$ are give as follows:
\begin{equation*}%\label{eq: unique2}
\begin{aligned}
s_k(Q): &  \ (a) \ \xymatrix@R=3ex{*{\circ}<3pt> \ar@{<.}[r] & *{\circ}<3pt> \ar@{->}[r]_<{k-1} & *{\circ}<3pt>
\ar@{-}[r]_<{k} &*{\circ}<3pt>
\ar@{->}[l]^<{\ \ k+1} \ar@{<.}[rr] &&*{\circ}<3pt>
\ar@{}[l]^<{\ \ l^*} } \text{ or } \ (b) \
\xymatrix@R=3ex{*{\circ}<3pt>
\ar@{}[r]_<{l^*}  &&*{\circ}<3pt>
\ar@{<.}[ll]^<{k-1}\ar@{->}[r] &*{\circ}<3pt>
\ar@{-}[l]^<{k} &*{\circ}<3pt>
\ar@{->}[l]^<{\ \ k+1}  } \\
(s_k(Q))^\rev: &  \ (a) \ \xymatrix@R=3ex{*{\circ}<3pt> \ar@{.>}[r] & *{\circ}<3pt> \ar@{<-}[r]_<{k-1} & *{\circ}<3pt>
\ar@{-}[r]_<{k} &*{\circ}<3pt>
\ar@{<-}[l]^<{\ \ k+1} \ar@{.>}[rr] &&*{\circ}<3pt>
\ar@{}[l]^<{\ \ l^*} } \text{ or }  \ (b) \
\xymatrix@R=3ex{*{\circ}<3pt>
\ar@{}[r]_<{l^*}  &&*{\circ}<3pt>
\ar@{.>}[ll]^<{k-1}\ar@{<-}[r] &*{\circ}<3pt>
\ar@{-}[l]^<{k} &*{\circ}<3pt>
\ar@{<-}[l]^<{\ \ k+1}  }.
\end{aligned}
\end{equation*}
Hence there exist no intersection between
$$\begin{cases}
\text{$N$-path of $\theta^{s_k(Q^\rev)}_{l^*}$ and the $S$-path of $\al_k$  if $Q$ is of the form (a),} \\
\text{$S$-path of $\theta^{s_k(Q^\rev)}_{l^*}$ and the $N$-path of $\al_k$  if $Q$ is of the form (b).} \\
\end{cases}
$$

Thus there exists $x \in \Z_{\ge 0}$ such that $\tau_{s_k(Q^\rev)}^{-x}(\theta^{s_k(Q^\rev)}_{l^*})$ and $ \al_k$
are of the form~\eqref{eq: ext form} in $\Gamma_{s_k(Q)}$ or $\Gamma_{(s_k(Q))^\rev}$.

($Q$ is of type $A_n$, $k = l^*$) In this case, by~\eqref{eq: assu}, we have
$$ (\gamma_k^{s_k(Q)},\al_k) \in \Phi_{s_k(Q)}(k,l)[t] $$
and hence our assertion follows.

($Q$ is of type $D_n$, $1 \le k,l \le n-2$) In this case,~\eqref{eq: assu} implies that the Dynkin diagram $Q$ is one of the forms in~\eqref{eq: unique}, $k^*=k$ and $l^*=l$.
The case when $k=l$ is trivial. Thus we assume that $k \ne l$.
Moreover, $\al=\ve_k-\ve_{k+1}$
$\be=\theta_l^{Q}$, $\al+\be \in \PR$ and $\mul(\al+\be)=1$. Thus $\mathtt{o}^{Q}_t(k,l)=1$ and $\GQ$ can be depicted as follows:

\vskip -2em

\begin{align} \label{eq: D path}
\scalebox{0.79}{{\xy
(-20,0)*{}="DL";(-10,-10)*{}="DD";(0,20)*{}="DT";(10,10)*{}="DR";
"DT"+(-30,-4); "DT"+(65,-4)**\dir{.};
"DD"+(-20,-6); "DD"+(75,-6) **\dir{.};
"DD"+(-20,-10); "DD"+(75,-10) **\dir{.};
"DT"+(-32,-4)*{\scriptstyle 1};
"DT"+(-32,-8)*{\scriptstyle 2};
"DT"+(-32,-12)*{\scriptstyle \vdots};
"DT"+(-32,-16)*{\scriptstyle \vdots};
"DT"+(-34,-36)*{\scriptstyle n-1};
"DT"+(-33,-40)*{\scriptstyle n};
"DD"+(-2,4); "DD"+(8,-6) **\dir{-};
"DD"+(-2,4); "DD"+(20,26) **\dir{.};
"DD"+(18,4); "DD"+(8,-6) **\dir{-};
"DD"+(18,4); "DD"+(40,26) **\dir{.};
"DD"+(52,18); "DD"+(44,26) **\dir{.};
"DD"+(52,18); "DD"+(28,-6) **\dir{-};
"DD"+(18,4); "DD"+(28,-6) **\dir{-};
"DD"+(52,18)*{\bullet};
"DD"+(42,18)*{\scriptstyle \al=\ve_k-\ve_{k+1}};
"DD"+(-2,4)*{\bullet}; "DD"+(-4,4); "DD"+(-10,4) **\dir{.}; "DD"+(-13,4)*{\scriptstyle l};
"DD"+(2,4)*{\scriptstyle \be};
"DD"+(8,-7)*{\scriptstyle k+1-\text{swing}};
"DD"+(40,26); "DD"+(44,26) **\crv{"DD"+(42,28)};
"DD"+(42,29)*{\scriptstyle 2};
"DD"+(8,14)*{\bullet};
"DD"+(13,14)*{\scriptstyle \al+\be};
"DD"+(18,4); "DD"+(8,14) **\dir{.};
\endxy}}
\end{align}
By Theorem~\ref{thm: short path}, all positive roots in the shallow $N$-path of $\al$ shares $-\ve_{k+1}$ as their summand,
and all positive roots in the shallow $S$-path of $\be$ are contained in $\GQ$ indeed. As in the $A_n$-case, $k$ becomes a sink of $s_k(Q)$ and
$\Gamma_{(s_k(Q))}$ can be depicted as follows:

\vskip -1.5em

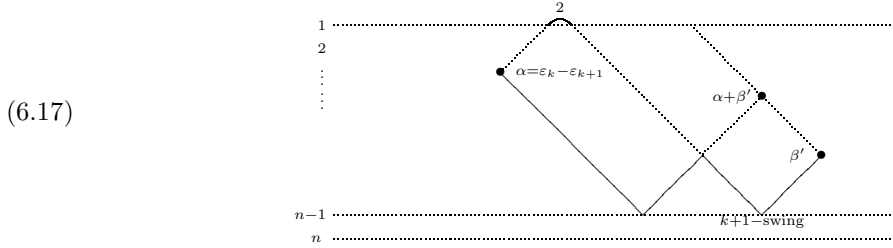
\begin{align} \label{eq: D path 2}
\scalebox{0.79}{{\xy
(-20,0)*{}="DL";(-10,-10)*{}="DD";(0,20)*{}="DT";(10,10)*{}="DR";
"DT"+(-30,-4); "DT"+(65,-4)**\dir{.};
"DD"+(-20,-6); "DD"+(75,-6) **\dir{.};
"DD"+(-20,-10); "DD"+(75,-10) **\dir{.};
"DT"+(-32,-4)*{\scriptstyle 1};
"DT"+(-32,-8)*{\scriptstyle 2};
"DT"+(-32,-12)*{\scriptstyle \vdots};
"DT"+(-32,-16)*{\scriptstyle \vdots};
"DT"+(-34,-36)*{\scriptstyle n-1};
"DT"+(-33,-40)*{\scriptstyle n};
"DD"+(62,4); "DD"+(52,-6) **\dir{-};
"DD"+(62,4); "DD"+(40,26) **\dir{.};
"DD"+(42,4); "DD"+(52,-6) **\dir{-};
"DD"+(42,4); "DD"+(20,26) **\dir{.};
"DD"+(8,18); "DD"+(16,26) **\dir{.};
"DD"+(8,18); "DD"+(32,-6) **\dir{-};
"DD"+(42,4); "DD"+(32,-6) **\dir{-};
"DD"+(8,18)*{\bullet};
"DD"+(18,18)*{\scriptstyle \al=\ve_k-\ve_{k+1}};
"DD"+(62,4)*{\bullet};
"DD"+(58,4)*{\scriptstyle \be'};
"DD"+(20,26); "DD"+(16,26) **\crv{"DD"+(18,28)};
"DD"+(18,29)*{\scriptstyle 2};
"DD"+(52,14)*{\bullet};
"DD"+(52,-7)*{\scriptstyle k+1-\text{swing}};
"DD"+(47,14)*{\scriptstyle \al+\be'};
"DD"+(42,4); "DD"+(52,14) **\dir{.};
\endxy}}
\end{align}
by Theorem~\ref{Thm: V-swing}. Thus there exist $\be' \in \PR$ such that
$(\be',\al) \in \Phi_{s_k(Q)}(k,l)[t]$.

($Q$ is of type $D_n$, $|\{ k,l \} \cap \{n-1,n\}|=1$) We shall prove only for $1 \le k \le n-2$ and $l \in \{n-1,n\}$.
For the case when $1 \le l \le n-2$ and $k \in \{n-1,n\}$, we can apply the similar strategy.
In this case, the Dynkin diagram $Q$ is of the following form:
$$
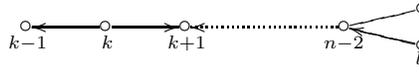
\xymatrix@R=0.1ex{
&&&&&*{\circ}<3pt> \ar@{-}[dl]  \\
*{\circ}<3pt>\ar@{<-}[r]_<{k-1 \ }&*{\circ}<3pt>\ar@{->}[r]_<{ k \ } &*{\circ}<3pt>\ar@{<.}[rr]_<{ k+1 \ }  &&*{\circ}<3pt>
\ar@{<-}[dr]_<{n-2} \\  &&&&&*{\circ}<3pt>
\ar@{-}[ul]^<{\ \ l}}$$

Furthermore, $\GQ$ and $\Gamma_{(s_k(Q))}$ can be depicted as follows:

\vskip -1.5em

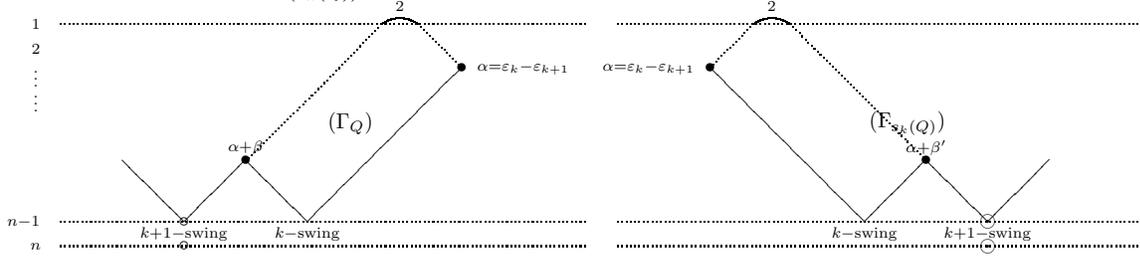
$$
\DpathOne \ \ \DpathTwo
$$
where $\be$ is one of the $\circ$'s. Thus there exist $\be' \in \PR$ such that
$(\be',\al) \in \Phi_{s_k(Q)}(k,l)[t]$ where $\be'$ is one of the $\odot$'s.

($Q$ is of type $D_n$, $\{ k,l \} \subset \{n-1,n\}$) Note that $k$ (resp. $l$) is a sink or a source. Thus the Dynkin quiver is of the following form:
$\raisebox{1em}
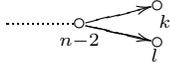
{\xymatrix@R=0.1ex{
&&*{\circ}<3pt> \ar@{<-}[dl]^<{\ \ k}  \\
 \ar@{.}[r] &*{\circ}<3pt> \ar@{->}[dr]_<{n-2} \\  &&*{\circ}<3pt>
\ar@{-}[ul]^<{\ \ l}}}$

Then $t=\mathsf{h}^\vee-2$ and $|\Phi_Q(k,l)[t]|=2$ by~\eqref{eq: mQ D}. Then it contradicts our assumption~\eqref{eq: assu0}.
\end{proof}

From the above proposition, we can define $D_{k,l}(z)$ in a natural way and call it \defn{the distance polynomial} at $k$ and $l$. We emphasize that $D_{k,l}(z)$ for $E$-types are also well-defined.

\begin{theorem}~\cite{AK,KKK13B} \label{thm: denom 1}
\begin{enumerate}
\item[{\rm (a)}] $d^{A^{(1)}_{n}}_{k,l}(z) = \displaystyle\prod_{x=1}^{\min(k,l,n+1-k,n+1-l)} (z-(-q)^{|k-l|+2x}).$
\item[{\rm (b)}] $ d^{D^{(1)}_{n}}_{k,l}(z)  =
\begin{cases}
\ \displaystyle \prod_{x=1}^{\min(k,l)} (z - (-q)^{|k-l|+2x})(z -(-q)^{2n-2-k-l+2x})  & \text{if } 1 \le k,l \le n-2, \\
\ \displaystyle \prod_{x=1}^{k}(z-(-q)^{n-k-1+2x}) &  \text{if } 1 \le k \le n-2,  l \in \{ n-1, n\}, \\
\ \displaystyle \prod_{x=1}^{\lfloor \frac{n-1}{2} \rfloor} (z-(-q)^{4x}) &  \text{if } \{k,l\}=\{n,n-1\},   \\
\ \displaystyle \prod_{x=1}^{\lfloor \frac{n}{2} \rfloor} (z-(-q)^{4x-2}) &   \text{if }  k=l \in \{ n-1, n\}.
 \end{cases}$

\noindent
In particular, $d^{D^{(1)}_{n}}_{k,l}(z)$ has a zero of
multiplicity $2$ at $z=(-q)^s$ when
\begin{equation} \label{eq: dpole 1}
2 \le k,l \le n-2, \ k+l > n-1, \ 2n-k-l \le s \le k+l \text{ and } s \equiv k+l \ {\rm mod} \ 2.
\end{equation}
\end{enumerate}
\end{theorem}

\begin{theorem} \label{eq: dist denom}
For any Dynkin quiver $Q$ of type $A_n$ $($resp. $D_n)$, the denominator formulas for $U'_q(A_n^{(1)})$ $($resp. $U'_q(D_n^{(1)}))$ can be read
from $\GQ$ and $\Gamma_{Q^{\rev}}$ as follows:
\begin{align*}
d_{k,l}(z) & = D_{k,l}(z;-q) \times (z-(-q)^{\mathsf{h}^\vee})^{\delta_{l,k^*}}
\end{align*}
where $\mathsf{h}^\vee$ is the dual Coxeter number of $A_n$ $($resp. $D_n)$.
\end{theorem}

\begin{proof}
By Proposition~\ref{prop: DQ DQ'}, It is enough to show our assertion for a fixed Dynkin quiver $Q$. Among $2^{n-1}$-many Dynkin quivers, we choose a canonical one
whose height function $\xiQ$ is given by $\xiQ_i=n-\Delta(1,i)$:
\begin{equation}
\begin{aligned}
& Q \ : \ \xymatrix@R=3ex{ *{ \circ }<3pt> \ar@{->}[r]_<{1}  &*{\circ}<3pt>
\ar@{->}[r]_<{2} &\cdots\ar@{->}[r] &*{\circ}<3pt>
\ar@{->}[r]_<{n-1} &*{\circ}<3pt>
\ar@{-}[l]^<{\ \ n}}  \quad {and} \quad
&   \raisebox{1.3em}{\xymatrix@R=0.1ex{
&&&&*{\circ}<3pt> \ar@{<-}[dl]^<{n-1}  \\
*{\circ}<3pt>\ar@{->}[r]_<{1 \ }&*{\circ}<3pt>\ar@{->}[r]_<{2 \ } &\cdots\ar@{->}[r]  &*{\circ}<3pt>
\ar@{->}[dr]_<{n-2} \\  &&&&*{\circ}<3pt>
\ar@{-}[ul]^<{\ \ n}}}.% \quad \ \ \ \text{ if $Q$ is of type $D_n$}.
\end{aligned}
\end{equation}

($Q$ of type $A_n$) The Auslander-Reiten quivers $\GQ$ and $\Gamma_{Q^\rev}$ are given as follows:
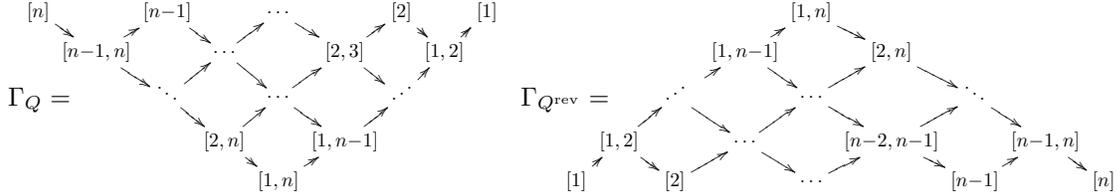
\begin{align*}
\GQ = \hspace{-5ex} \raisebox{3.5em}{\scalebox{0.74}{\xymatrix@R=1ex@C=0.1ex{
[n]\ar[dr] && [n\hspace{-0.5ex}-\hspace{-0.5ex}1]\ar[dr] && \cdots \ar[dr]  && [2]\ar[dr] && [1] \\
& [n\hspace{-0.5ex}-\hspace{-0.5ex}1,n] \ar[dr]\ar[ur] && \cdots \ar[dr]\ar[ur] &  &  [2,3] \ar[dr]\ar[ur] && [1,2] \ar[ur]  \\
&& \ddots \ar[dr]\ar[ur] && \cdots \ar[dr]\ar[ur] &&   \iddots \ar[ur]  \\
&&&[2,n]  \ar[dr]\ar[ur] && [1,n\hspace{-0.5ex}-\hspace{-0.5ex}1]  \ar[ur]\\
&&&& [1,n] \ar[ur]
}}}
\ \
\Gamma_{Q^\rev} = \hspace{-5ex}  \raisebox{3.5em}{\scalebox{0.74}{\xymatrix@R=1ex@C=0.1ex{
&&&& [1,n] \ar[dr]\\
&&& [1,n\hspace{-0.5ex}-\hspace{-0.5ex}1] \ar[ur]\ar[dr] & & [2,n] \ar[dr]\\
&& \ \ \iddots \ \ \ar[ur]\ar[dr] && \cdots \ar[ur]\ar[dr]  && \ddots \ar[dr]\\
& [1,2]\ar[ur]\ar[dr]  && \cdots\ar[ur]\ar[dr]  &&  [n\hspace{-0.5ex}-\hspace{-0.5ex}2,n\hspace{-0.5ex}-\hspace{-0.5ex}1] \ar[ur]\ar[dr] && [n\hspace{-0.5ex}-\hspace{-0.5ex}1,n] \ar[dr]\\
[1]\ar[ur]&&[2]\ar[ur]&& \cdots \ar[ur]&& [n\hspace{-0.5ex}-\hspace{-0.5ex}1]\ar[ur] && [n]
}}}
\end{align*}
Thus $\mQ_k= n-k$ and $m^{Q^\rev}_k= k-1$ for $1 \le k \le n$.

Recall that, for any Dynkin quiver $Q'$ of type $A_n$, we have $0 \le \gdist_{Q'}(\al,\be) \le 1$. In particular, if  $\gdist_{Q'}(\al,\be) = 1$,
$\al,\be$ is located in $Q'$ satisfying one of the following forms:

\vskip -1.5em

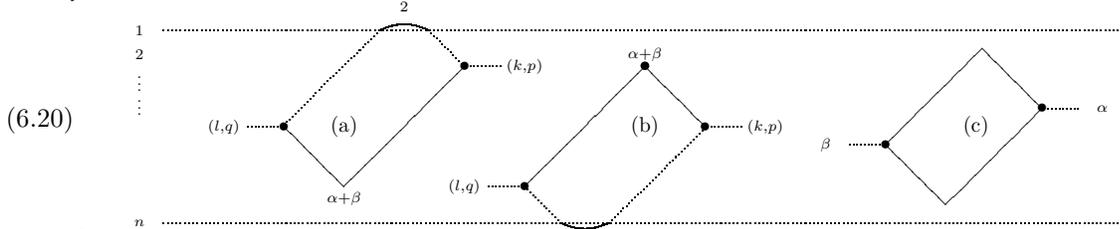
\begin{align} \label{eq: Dist poly A}
\scalebox{0.8}{{\xy
(0,0)*{}="DL";(10,-10)*{}="DD";(20,20)*{}="DT";(30,10)*{}="DR";
"DL"; "DD" **\dir{-};"DL"; "DT"+(-4,-4) **\dir{.};
"DT"+(-40,-4); "DT"+(120,-4)**\dir{.};
"DD"+(-30,-6); "DD"+(130,-6) **\dir{.};
"DT"+(4,-4); "DR" **\dir{.};"DR"; "DD" **\dir{-};
"DL"+(-10,0)*{\scriptstyle (l,q)};
"DL"+(10,0)*{{\rm (a)}};
"DL"+(-1,0); "DL"+(-6,0) **\dir{.};
"DR"+(10,0)*{\scriptstyle (k,p)};
"DR"+(1,0); "DR"+(6,0) **\dir{.};
"DL"*{\bullet};"DR"*{\bullet};
%"DT"+(4,-4)*{\bullet};"DT"+(-4,-4)*{\bullet};"DD"*{\bullet};
"DT"+(4,-4); "DT"+(-4,-4) **\crv{"DT"+(0,-2)};
"DT"+(0,0)*{\scriptstyle 2};
"DD"+(0,-2)*{\scriptstyle \al+\be};
%"DT"+(-8,0)*{\scriptstyle (1,s+(i-1))};"DT"+(9,0)*{\scriptstyle (1,t-(j-1))};
"DT"+(-44,-4)*{\scriptstyle 1};
"DT"+(-44,-8)*{\scriptstyle 2};
"DT"+(-44,-12)*{\scriptstyle \vdots};
"DT"+(-44,-16)*{\scriptstyle \vdots};
"DT"+(-44,-36)*{\scriptstyle n};
"DL"+(40,-10); "DD"+(36,-6) **\dir{.};
"DR"+(40,-10); "DD"+(44,-6) **\dir{.};
"DT"+(40,-10); "DR"+(40,-10) **\dir{-};
"DT"+(40,-10); "DL"+(40,-10) **\dir{-};
"DL"+(30,-10)*{\scriptstyle (l,q)};
"DL"+(39,-10); "DL"+(34,-10) **\dir{.};
"DR"+(50,-10)*{\scriptstyle (k,p)};
"DR"+(41,-10); "DR"+(46,-10) **\dir{.};
"DL"+(40,-10)*{\bullet};"DR"+(40,-10)*{\bullet};
"DT"+(40,-10)*{\bullet};
"DT"+(40,-8)*{\scriptstyle \al+\be};
"DT"+(40,-10)*{\bullet};
"DD"+(44,-6); "DD"+(36,-6) **\crv{"DD"+(40,-8)};
"DD"+(40,-10)*{\scriptstyle 2};
%"DD"+(44,-6)*{\bullet};"DD"+(36,-6)*{\bullet};
"DL"+(60,0)*{{\rm(b)}};
%"DD"+(30,-10)*{\scriptstyle (n,s+(n-i))};
%"DD"+(51,-10)*{\scriptstyle (n,t-(n-j))};
%
"DL"+(100,-3); "DD"+(100,-3) **\dir{-};
"DR"+(96,-7); "DD"+(100,-3) **\dir{-};
"DT"+(96,-7); "DR"+(96,-7) **\dir{-};
"DT"+(96,-7); "DL"+(100,-3) **\dir{-};
"DL"+(90,-3)*{\scriptstyle \be};
"DL"+(99,-3); "DL"+(94,-3) **\dir{.};
"DR"+(106,-7)*{\scriptstyle \al};
"DR"+(97,-7); "DR"+(102,-7) **\dir{.};"DL"+(115,0)*{{\rm(c)}};
"DL"+(100,-3)*{\bullet};"DR"+(96,-7)*{\bullet};
%"DT"+(96,-7)*{\bullet}; "DD"+(100,-3)*{\bullet};
%"DT"+(96,-5)*{\scriptstyle (k,u)};
%"DD"+(100,-5)*{\scriptstyle (k',u')};
%
\endxy}}
\end{align}

\vskip -1.5em

\noindent
for $\phi^{-1}_{Q'}(\al,0)=(k,p)$ and $\phi^{-1}_{Q'}(\be,0)=(l,q)$. Then one can see that
$$\text{ $(a)$, $(b)$ are contained in $Q$ and $(b)$, $(c)$ are contained in $Q^\rev$.}$$

Furthermore, there exists $1 \le s \le \min(k,l,n+1-k,n+1-l)$ such that the length between $\al$ and $\be$, that is $|p-q|$,
is equal to $|k-l|+2s$. Conversely, for $1 \le s \le \min(k,l,n+1-k,n+1-l)$, we can draw one of {\rm (a)} $\sim$ {\rm (c)} in $\GQ$ or $\Gamma_{Q^\rev}$
satisfying whose length between $\al$ and $\be$ is $|a-b|=|k-l|+2s$, except when $k =l^*$ and $|a-b|=\mathsf{h}^\vee$ (see Remark~\ref{rem: comb refl}).
Thus our assertion follows.

\smallskip

($Q$ of type $D$) The AR-quivers $\GQ$ and $\Gamma_{Q^\rev}$ are given as follows:
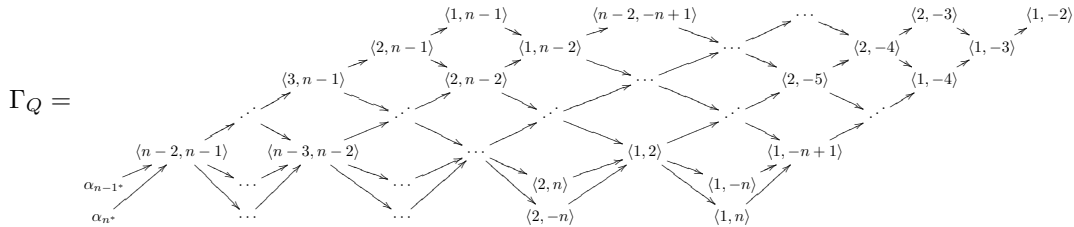
\begin{align*}
&\GQ = \raisebox{3.5em}{\scalebox{0.6}{\xymatrix@R=1ex@C=0.05ex{
&&&&& \lf 1 , n-1\rf\ar[dr]  && \lf n-2 ,  -n+1 \rf\ar[dr]  && \cdots \ar[dr]  && \lf 2 , -3 \rf \ar[dr] && \lf 1 , -2 \rf \\
&&&& \lf 2 , n-1 \rf \ar[dr] \ar[ur] && \lf 1 ,  n-2 \rf\ar[dr] \ar[ur]  && \cdots \ar[dr]\ar[ur]   && \lf 2 , -4 \rf\ar[dr] \ar[ur]  && \lf 1 , -3 \rf \ar[ur]  \\
&&&  \lf 3 , n-1 \rf\ar[dr] \ar[ur]  && \lf 2 ,  n-2 \rf \ar[dr] \ar[ur] && \cdots\ar[dr] \ar[ur]  && \lf 2 , -5 \rf\ar[dr] \ar[ur]  && \lf 1 , -4 \rf \ar[ur]   \\
&&  \iddots\ar[dr] \ar[ur]  &&  \iddots \ar[dr] \ar[ur]  &&  \iddots\ar[dr] \ar[ur]   & &  \iddots\ar[dr] \ar[ur]  &&  \iddots \ar[ur] \\
& \lf n-2 , n-1\rf \ar[ddr] \ar[dr] \ar[ur]&& \lf n-3 , n-2\rf \ar[ddr] \ar[dr] \ar[ur]&& \cdots \ar[ddr] \ar[dr] \ar[ur]&& \lf 1 , 2\rf \ar[ddr] \ar[dr] \ar[ur]&&  \lf 1 , -n+1\rf \ar[ur] \\
\al_{n-1^*} \ar[ur] && \cdots\ar[ur] && \cdots  \ar[ur]&& \lf2 , n\rf \ar[ur] && \lf 1 , -n\rf \ar[ur]\\
\al_{n^*} \ar[uur] && \cdots\ar[uur] && \cdots  \ar[uur]&& \lf2 , -n\rf\ar[uur] && \lf 1 , n\rf \ar[uur]
}}} \allowdisplaybreaks\\
&\Gamma_{Q^\rev} = \raisebox{3.5em}{\scalebox{0.6}
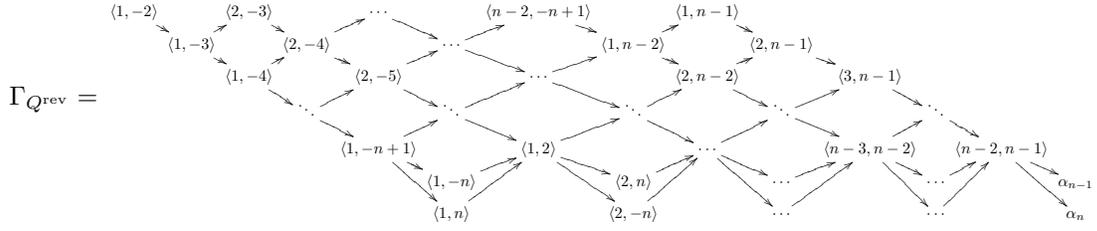
{\xymatrix@R=1ex@C=0.05ex{
 \lf 1 , -2\rf\ar[dr]  && \lf 2 ,  -3 \rf\ar[dr]  && \cdots \ar[dr]  && \lf n-2 , -n+1 \rf \ar[dr] && \lf 1 , n-1 \rf \ar[dr]\\
& \lf 1 , -3 \rf \ar[dr] \ar[ur] && \lf 2 ,  -4 \rf\ar[dr] \ar[ur]  && \cdots \ar[dr]\ar[ur]   && \lf 1 , n-2 \rf\ar[dr] \ar[ur]  && \lf 2 , n-1 \rf \ar[dr] \\
&&  \lf 1 , -4 \rf\ar[dr] \ar[ur]  && \lf 2 ,  -5 \rf \ar[dr] \ar[ur] && \cdots\ar[dr] \ar[ur]  && \lf 2 , n-2 \rf\ar[dr] \ar[ur]  && \lf 3 , n-1 \rf \ar[dr] \\
&&&  \ddots\ar[dr] \ar[ur]  &&  \ddots \ar[dr] \ar[ur]  &&  \ddots\ar[dr] \ar[ur]   &&  \ddots\ar[dr] \ar[ur]  &&  \ddots \ar[dr]\\
&&&& \lf1 , -n+1\rf \ar[ddr] \ar[dr] \ar[ur]&& \lf 1 , 2\rf \ar[ddr] \ar[dr] \ar[ur]&& \cdots \ar[ddr] \ar[dr] \ar[ur]&& \lf n-3 , n-2\rf \ar[ddr] \ar[dr] \ar[ur] &&  \lf n-2 , n-1\rf  \ar[ddr] \ar[dr]\\
&&&&& \lf1 , -n\rf \ar[ur] && \lf2 , n\rf\ar[ur] && \cdots  \ar[ur]&& \cdots \ar[ur] && \al_{n-1}\\
&&&&& \lf1 , n\rf \ar[uur] && \lf2 , -n\rf\ar[uur] && \cdots  \ar[uur]&& \cdots \ar[uur] && \al_{n}
}}}
\end{align*}

\vskip -1.5em

\noindent
Thus $\mQ_k=m^{Q^\rev}_k= n-2$ for $1 \le k \le n$.

Recall that, for any Dynkin quiver $Q'$ of type $D$, we have $0 \le \gdist_{Q'}(\al,\be) \le 2$.

\noindent
(1) If  $\gdist_{Q'}(\al,\be) = 1$,
$\al$ and $\be$ are located in $Q'$ satisfying one of the following forms:

\vskip -1.5em

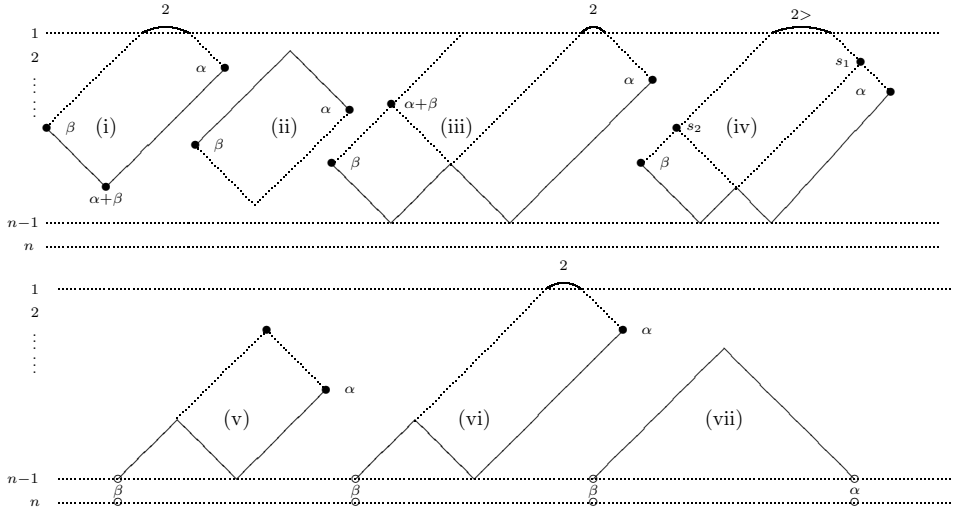
\begin{align*}
& \scalebox{0.79}{{\xy
(-20,0)*{}="DL";(-10,-10)*{}="DD";(0,20)*{}="DT";(10,10)*{}="DR";
"DT"+(-30,-4); "DT"+(120,-4)**\dir{.};
"DD"+(-20,-6); "DD"+(130,-6) **\dir{.};
"DD"+(-20,-10); "DD"+(130,-10) **\dir{.};
"DT"+(-32,-4)*{\scriptstyle 1};
"DT"+(-32,-8)*{\scriptstyle 2};
"DT"+(-32,-12)*{\scriptstyle \vdots};
"DT"+(-32,-16)*{\scriptstyle \vdots};
"DT"+(-34,-36)*{\scriptstyle n-1};
"DT"+(-33,-40)*{\scriptstyle n};
"DL"+(-10,0); "DD"+(-10,0) **\dir{-};"DL"+(-10,0); "DT"+(-14,-4) **\dir{.};
"DT"+(-6,-4); "DR"+(-10,0) **\dir{.};"DR"+(-10,0); "DD"+(-10,0) **\dir{-};
"DL"+(-6,0)*{\scriptstyle \be};
"DL"+(0,0)*{{\rm(i)}};
"DR"+(-14,0)*{\scriptstyle \al};
"DL"+(-10,0)*{\bullet};
"DR"+(-10,0)*{\bullet};
"DT"+(-6,-4); "DT"+(-14,-4) **\crv{"DT"+(-10,-2)};
"DT"+(-10,0)*{\scriptstyle 2};
"DD"+(-10,0)*{\bullet};
"DD"+(-10,-2)*{\scriptstyle \al+\be};
"DL"+(15,-3); "DD"+(15,-3) **\dir{.};
"DR"+(11,-7); "DD"+(15,-3) **\dir{.};
"DT"+(11,-7); "DR"+(11,-7) **\dir{-};
"DT"+(11,-7); "DL"+(15,-3) **\dir{-};
"DL"+(19,-3)*{\scriptstyle \be};
"DR"+(7,-7)*{\scriptstyle \al};
"DL"+(30,0)*{{\rm(ii)}};
"DL"+(15,-3)*{\bullet};"DR"+(11,-7)*{\bullet};
"DD"+(28,4); "DD"+(38,-6) **\dir{-};
"DD"+(28,4); "DD"+(50,26) **\dir{.};
"DD"+(48,4); "DD"+(38,-6) **\dir{-};
"DD"+(48,4); "DD"+(70,26) **\dir{.};
"DD"+(82,18); "DD"+(74,26) **\dir{.};
"DD"+(82,18); "DD"+(58,-6) **\dir{-};
"DD"+(48,4); "DD"+(58,-6) **\dir{-};
"DD"+(82,18)*{\bullet};
"DD"+(78,18)*{\scriptstyle \al};
"DD"+(28,4)*{\bullet};
"DD"+(32,4)*{\scriptstyle \be};
"DD"+(74,26); "DD"+(70,26) **\crv{"DD"+(72,28)};
"DD"+(72,30)*{\scriptstyle 2};
"DD"+(38,14)*{\bullet};
"DD"+(43,14)*{\scriptstyle \al+\be};
"DL"+(59,0)*{{\rm(iii)}};
"DD"+(48,4); "DD"+(38,14) **\dir{.};
%\endxy}}
%& \text{ if } 1 \le k,l \le n-2, \allowdisplaybreaks \\
%\scalebox{0.79}{{\xy
%(-20,0)*{}="DL";(-10,-10)*{}="DD";(0,20)*{}="DT";(10,10)*{}="DR";
%"DT"+(-30,-4); "DT"+(70,-4)**\dir{.};
%"DD"+(-20,-6); "DD"+(80,-6) **\dir{.};
%"DD"+(-20,-10); "DD"+(80,-10) **\dir{.};
%"DT"+(-32,-4)*{\scriptstyle 1};
%"DT"+(-32,-8)*{\scriptstyle 2};
%"DT"+(-32,-12)*{\scriptstyle \vdots};
%"DT"+(-32,-16)*{\scriptstyle \vdots};
%"DT"+(-34,-36)*{\scriptstyle n-1};
%"DT"+(-33,-40)*{\scriptstyle n};
%
"DD"+(80,4); "DD"+(90,-6) **\dir{-};
"DD"+(96,0); "DD"+(90,-6) **\dir{-};
"DD"+(96,0); "DD"+(102,-6) **\dir{-};
"DD"+(122,16); "DD"+(102,-6) **\dir{-};
"DD"+(80,4); "DD"+(102,26) **\dir{.};
"DD"+(122,16); "DD"+(112,26) **\dir{.};
"DD"+(96,0); "DD"+(117,21) **\dir{.};
"DD"+(96,0); "DD"+(86,10) **\dir{.};
"DD"+(117,21)*{\bullet};
"DD"+(114,21)*{\scriptstyle s_1};
"DD"+(86,10)*{\bullet};
"DD"+(89,10)*{\scriptstyle s_2};
"DD"+(102,26); "DD"+(112,26) **\crv{"DD"+(107,28)};
"DD"+(107,29)*{\scriptstyle 2 >};
"DD"+(80,4)*{\bullet};
"DL"+(107,0)*{{\rm (iv)}};
"DD"+(84,4)*{\scriptstyle \be};
"DD"+(122,16)*{\bullet};
"DD"+(117,16)*{\scriptstyle \al};
\endxy}}
%& \text{ if } 1 \le k,l \le n-2,
\allowdisplaybreaks \\
& \scalebox{0.79}{{\xy
(-20,0)*{}="DL";(-10,-10)*{}="DD";(0,20)*{}="DT";(10,10)*{}="DR";
"DT"+(-30,-4); "DT"+(120,-4)**\dir{.};
"DD"+(-20,-6); "DD"+(130,-6) **\dir{.};
"DD"+(-20,-10); "DD"+(130,-10) **\dir{.};
"DT"+(-34,-4)*{\scriptstyle 1};
"DT"+(-34,-8)*{\scriptstyle 2};
"DT"+(-34,-12)*{\scriptstyle \vdots};
"DT"+(-34,-16)*{\scriptstyle \vdots};
"DT"+(-36,-36)*{\scriptstyle n-1};
"DT"+(-34,-40)*{\scriptstyle n};
"DD"+(0,4); "DD"+(10,-6) **\dir{-};
"DD"+(0,4); "DD"+(-10,-6) **\dir{-};
"DD"+(0,4); "DD"+(15,19) **\dir{.};
"DD"+(25,9);"DD"+(10,-6) **\dir{-};
"DD"+(25,9);"DD"+(15,19) **\dir{.};
"DD"+(25,9)*{\bullet};
"DD"+(15,19)*{\bullet};
"DD"+(-10,-6)*{\circ};
"DD"+(-10,-10)*{\circ};
"DD"+(-10,-8)*{\scriptstyle \be};
%"DD"+(15,21)*{\scriptstyle (k,u)};
"DD"+(29,9)*{\scriptstyle \al};
"DD"+(10,4)*{{\rm (v)}};
"DD"+(40,4); "DD"+(50,-6) **\dir{-};
"DD"+(40,4); "DD"+(30,-6) **\dir{-};
"DD"+(40,4); "DD"+(62,26) **\dir{.};
"DD"+(75,19);"DD"+(50,-6) **\dir{-};
"DD"+(75,19);"DD"+(68,26) **\dir{.};
"DD"+(75,19)*{\bullet};
%"DD"+(50,-6)*{\bullet};
%"DD"+(50,-8)*{\scriptstyle (n-1,s+2l)};
"DD"+(50,4)*{{\rm (vi)}};
"DD"+(79,19)*{\scriptstyle \al};
"DD"+(30,-6)*{\circ};
"DD"+(30,-10)*{\circ};
"DD"+(30,-8)*{\scriptstyle \be};
"DD"+(62,26); "DD"+(68,26) **\crv{"DD"+(65,28)};
"DD"+(65,30)*{\scriptstyle 2};
%\endxy}}
%&  \begin{matrix}\text{ if } 1 \le k \le n-2, \\[2ex] l \in \{ n-1,n\}, \end{matrix} \allowdisplaybreaks \\
%\scalebox{0.82}{{\xy
%(-20,0)*{}="DL";(-10,-10)*{}="DD";(0,20)*{}="DT";(10,10)*{}="DR";
%"DT"+(-30,-4); "DT"+(40,-4)**\dir{.};
%"DD"+(-20,-6); "DD"+(50,-6) **\dir{.};
%"DD"+(-20,-10); "DD"+(50,-10) **\dir{.};
%"DT"+(-34,-4)*{\scriptstyle 1};
%"DT"+(-34,-8)*{\scriptstyle 2};
%"DT"+(-34,-12)*{\scriptstyle \vdots};
%"DT"+(-34,-16)*{\scriptstyle \vdots};
%"DT"+(-36,-36)*{\scriptstyle n-1};
%"DT"+(-34,-40)*{\scriptstyle n};
%
"DD"+(92,16); "DD"+(114,-6) **\dir{-};
"DD"+(92,16); "DD"+(70,-6) **\dir{-};
"DD"+(92,4)*{ {\rm (vii)}};
"DD"+(114,-6)*{\circ};
"DD"+(114,-10)*{\circ};
"DD"+(114,-8)*{\scriptstyle \al};
"DD"+(70,-6)*{\circ};
"DD"+(70,-10)*{\circ};
"DD"+(70,-8)*{\scriptstyle \be};
\endxy}} %& \text{ if } k, l \in \{ n-1,n\}
\end{align*}
where $\phi^{-1}_{Q'}(\al,a)=k$, $\phi^{-1}_{Q'}(\be)=(l,b)$ and $k \le l$.

(2) If  $\gdist_{Q'}(\al,\be) = 2$,
$\al$ and $\be$ are located in $Q'$ satisfying one of the following forms:

\vskip -1.5em

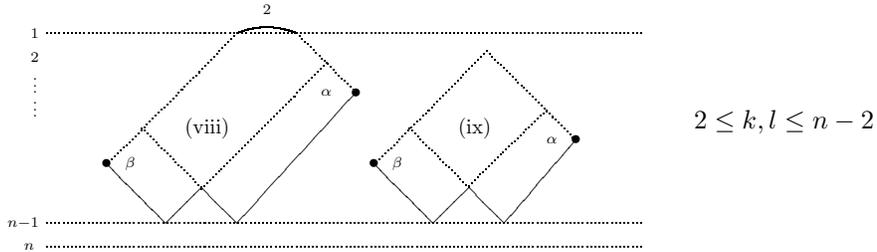
$$
\scalebox{0.79}{{\xy
(-20,0)*{}="DL";(-10,-10)*{}="DD";(0,20)*{}="DT";(10,10)*{}="DR";
"DT"+(-30,-4); "DT"+(70,-4)**\dir{.};
"DD"+(-20,-6); "DD"+(80,-6) **\dir{.};
"DD"+(-20,-10); "DD"+(80,-10) **\dir{.};
"DT"+(-32,-4)*{\scriptstyle 1};
"DT"+(-32,-8)*{\scriptstyle 2};
"DT"+(-32,-12)*{\scriptstyle \vdots};
"DT"+(-32,-16)*{\scriptstyle \vdots};
"DT"+(-34,-36)*{\scriptstyle n-1};
"DT"+(-33,-40)*{\scriptstyle n};
"DD"+(-10,4); "DD"+(0,-6) **\dir{-};
"DD"+(6,0); "DD"+(0,-6) **\dir{-};
"DD"+(6,0); "DD"+(12,-6) **\dir{-};
"DD"+(32,16); "DD"+(12,-6) **\dir{-};
"DD"+(-10,4); "DD"+(12,26) **\dir{.};
"DD"+(32,16); "DD"+(22,26) **\dir{.};
"DD"+(6,0); "DD"+(27,21) **\dir{.};
"DD"+(6,0); "DD"+(-4,10) **\dir{.};
"DD"+(12,26); "DD"+(22,26) **\crv{"DD"+(17,28)};
"DD"+(17,30)*{\scriptstyle 2};
%"DD"+(-4,10)*{\bullet};
%"DD"+(12,26)*{\bullet};
%"DD"+(6,28)*{\scriptstyle (1,s+(i-1))};
%"DD"+(22,26)*{\bullet};
%"DD"+(25,28)*{\scriptstyle (1,t-(j-1))};
"DD"+(-10,4)*{\bullet};
"DL"+(17,0)*{{\rm(viii)}};
"DD"+(-6,4)*{\scriptstyle \be};
"DD"+(32,16)*{\bullet};
"DD"+(27,16)*{\scriptstyle \al};
"DD"+(35,4); "DD"+(45,-6) **\dir{-};
"DD"+(51,0); "DD"+(45,-6) **\dir{-};
"DD"+(51,0); "DD"+(57,-6) **\dir{-};
"DD"+(69,8); "DD"+(57,-6) **\dir{-};
"DD"+(35,4); "DD"+(54,23) **\dir{.};
"DD"+(69,8); "DD"+(54,23) **\dir{.};
"DD"+(51,0); "DD"+(64,13) **\dir{.};
"DD"+(51,0); "DD"+(41,10) **\dir{.};
%"DD"+(41,10)*{\bullet};
%"DD"+(46,10)*{\scriptstyle (i',s')};
"DD"+(35,4)*{\bullet};
%"DD"+(45,-6)*{\bullet};
%"DD"+(54,23)*{\bullet};
%"DD"+(64,13)*{\bullet};
%"DD"+(59,13)*{\scriptstyle (j',t')};
%"DD"+(57,-6)*{\bullet};
%"DD"+(51,0)*{\bullet};
%"DD"+(51,2)*{\scriptstyle (k',u')};
"DL"+(62,0)*{{\rm(ix)}};
"DD"+(39,4)*{\scriptstyle \be};
"DD"+(69,8)*{\bullet};
"DD"+(65,8)*{\scriptstyle \al};
%"DD"+(54,25)*{\scriptstyle (k,u)};
\endxy}} \qquad
 2 \le k ,l \le n-2
$$
where $\phi^{-1}_{Q'}(\al)=(k,a)$, $\phi^{-1}_{Q'}(\be)=(l,b)$  and $k \le l$.

By choosing $\al=\gamma^Q_k$, one can check that $(\al,\be)$ in {\rm (i)} $\sim$ {\rm (ix)} is contained in $\Phi_Q(k,l)\big[|a-b|\big]$ when $k \le l$.
For the case when $k \ge l$, one can easily check that $(\al,\be)$ in {\rm (i)} $\sim$ {\rm (ix)} is contained in $\Phi_{Q^\rev}(k,l)[|a-b|]$, by choosing
$\be=\theta^Q_{l^*}$.

Furthermore, the length of a path between $\al$ and $\be$, $s \seteq |a-b|$, satisfies the condition in~\eqref{eq: dpole 1} only when
$(\al,\be)$ is one of the forms {\rm (viii)} $\sim$ {\rm (ix)}.

Conversely, we can draw one of the diagrams {\rm (i)} $\sim$ {\rm (ix)} in $\GQ$ or $\Gamma_{Q^\rev}$
satisfying whose length between $\al$ and $\be$, ($\phi^{-1}_{Q}(\al)=(k,a)$, $\phi^{-1}_{Q}(\be)=(l,b)$ or $\phi^{-1}_{Q^\rev}(\al)=(k,a)$, $\phi^{-1}_{Q^\rev}(\be)=(l,b)$)
is given as
$$|a-b| =
\begin{cases}
\begin{matrix}|k-l|+2x \text{ or } \\
 2n-2-k-l+2x \end{matrix} & \text{ if } 1 \le k,l \le n-2 \text{ and } 1 \le x \le \min(k,l), \\
\ \ n-k-1+2x & \text{ if } 1 \le k \le n-2, \ l \in \{ n-1, n\} \text{ and } 1 \le k \le x, \\
\qquad 4x & \text{ if } \{k,l\}=\{n,n-1\} \text{ and } 1 \le x \le \lfloor \frac{n-1}{2} \rfloor, \\
\qquad 4x-2 & \text{ if } \{k,l\}=\{n,n-1\} \text{ and } 1 \le x \le \lfloor \frac{n}{2} \rfloor,
\end{cases}$$
except when $k =l^*$ and $|a-b|=\mathsf{h}^\vee$  (see Remark~\ref{rem: comb refl}).
Thus our assertion follows.
\end{proof}

\begin{corollary} \label{cor: dist denom}
For $k,l \in I_0$ of $U'_q(E_{6,7,8}^{(1)})$, let us denote by $\widetilde{D}_{k,l}(z;-q)$ which has roots of order $1$ and whose linear factors are the same as the ones of
$D_{k,l}(z;-q)$. Then the denominator formulas $d_{k,l}(z)$ of $U'_q(E_{6,7,8}^{(1)})$
is divisible by $\widetilde{D}_{k,l}(z;-q) \times (z-(-q)^{\mathsf{h}^\vee})^{\delta_{l,k^*}}$.
\end{corollary}

\begin{proof}
Our assertion follows from Theorem~\ref{thm: for C_Q} and Remark~\ref{rem: important observation}.
\end{proof}

\begin{theorem} \label{thm: all tensor 1}
Let $\g$ be a quantum affine algebra of type $A_n^{(1)}$ or
$D_n^{(1)}$. For any $i,j \in I_0$ and $x,y \in \ko$, we can compute
the simple socle or the simple head of the tensor product of two
fundamental representations $$V \seteq V(\varpi_i)_x \tens
V(\varpi_j)_y$$ as a $U'_q(\g)$-module. Moreover, it is a tensor product of fundamental representations.
\end{theorem}

\begin{proof}
By Theorem~\ref{Thm: basic properties} and Theorem~\ref{thm: denom 1}, we suffice to consider that $x/y \seteq (-q)^{a} \in \Z[q,q^{-1}]^\times$
for some $a \in \Z \setminus \{ 0 \}$ and hence can assume that $V$ is reducible.
Theorem~\ref{eq: dist denom} tells that
there exist a Dynkin quiver $Q$ and positive roots $\al,\be \in \PR$ such that
$V^{(1)}(\varpi_i)_x =V_Q(\al)$ and $V^{(1)}(\varpi_j)_y =V_Q(\be)$ (up to parameter shift) unless
$i= j^*$ and $|a| = \mathsf{h}^\vee$.
By~\eqref{eq: dual affine} and Theorem~\ref{thm: for C_Q}, we have
$$
\begin{cases}
\hd(V) \simeq V_Q(\us) & \text{ if $a <0$ and $a \ne -\mathsf{h}^\vee$, } \\
\soc(V) \simeq V_Q(\us) & \text{ if $a >0$ and $a \ne \mathsf{h}^\vee$, }
\end{cases}
\qquad
\begin{cases}
\hd(V) \simeq \ko & \text{ if $i= j^*$ and $a = -\mathsf{h}^\vee$}, \\
\soc(V) \simeq \ko & \text{ if $i= j^*$ and $a = \mathsf{h}^\vee$},
\end{cases}
$$
where $$\text{$\us$ is equal to $\soc_Q(\al,\be)$ if $a >0$ and $a \ne \mathsf{h}^\vee$, and $\soc_Q(\be,\al)$ otherwise.}$$ Then our assertion follows from the fact that
$V_Q(\us)$ is of the form of tensor product of fundamental representations.
\end{proof}

\begin{remark}
By the similarities between untwisted and twisted affine types (\cite{H01,KKKOIV}), one can also read the denominator formulas for $A^{(2)}_{n}$
and $D^{(2)}_{n}$, which were computed in~\cite{Oh14R}, by folding $\Gamma_Q$ with respect to the map $^\star$. Furthermore,
 we have the same result of Theorem~\ref{thm: all tensor 1} for twisted affine types $A^{(2)}_{n}$ and $D^{(2)}_{n}$.
\end{remark}

\subsection{Conjectures on $E^{(1)}_{6,7,8}$} \label{subsec: E conjectures}
From Theorem~\ref{eq: dist denom}, we have a conjecture given as follows:

\begin{conjecture} \label{conj: dist E}
For any Dynkin quiver $Q$ of type $E_{6,7,8}$, the denominator formulas for $U'_q(E_{6,7,8}^{(1)})$ can be read
from $\GQ$ and $\Gamma_{Q^{\rev}}$ as follows:
\begin{align*}
d_{k,l}(z) & = D_{k,l}(z) \times (z-(-q)^{\mathsf{h}^\vee})^{\delta_{l,k^*}}
\end{align*}
where $\mathsf{h}^\vee$ is the dual Coxeter number of $E_{6,7,8}$.
\end{conjecture}

Since $\gdist_{Q}(\al_i,\al_j)=1$ for $\Delta(i,j)=1$ and $\gdist_{Q}(\al_i,\al_j)=0$ for $\Delta(i,j)>1$, the Schur-Weyl quiver $\Gamma^{\Xi}$
determined by $(U'_q(E_{6,7,8}^{(1)}),\phi_Q^{-1})$ is isomorphic to $Q^\rev$ under the assumption that Conjecture~\ref{conj: dist E} holds. Moreover, one can prove the existence of
the functor $\mathcal{F}_Q^{(1)}$ satisfying {\rm (a)} and {\rm (b)} in Theorem~\ref{thm:gQASW duality 2} by using the same arguments in~\cite{KKK13B},
under the same assumption.

Now we give a conjectural denominator formulas $d_{k,l}(z)$ for $U_q'(E_6^{(1)})$:
\begin{align*}
& d_{1,1}(z)=d_{5,5}(z)=(z-q^2)(z-q^8),  \allowdisplaybreaks\\
& d_{1,2}(z)=d_{2,1}(z)=d_{4,5}(z)=d_{5,4}(z)=(z+q^3)(z+q^7)(z+q^9),\allowdisplaybreaks\\
& d_{1,3}(z)=d_{3,1}(z)=d_{3,5}(z)=d_{5,3}(z)=(z-q^4)(z-q^6)(z-q^8)(z-q^{10}),\allowdisplaybreaks\\
& d_{1,4}(z)=d_{4,1}(z)=d_{2,5}(z)=d_{5,2}(z)=(z+q^5)(z+q^7)(z+q^{11}),\allowdisplaybreaks\\
& d_{1,5}(z)=d_{5,1}(z)=(z-q^6)(z-q^{12}),   \allowdisplaybreaks\\
& d_{1,6}(z)=d_{6,1}(z)=d_{5,6}(z)=d_{6,5}(z)=(z+q^5)(z+q^9),\allowdisplaybreaks\\
& d_{2,2}(z)=d_{4,4}(z)=(z-q^2)(z-q^4)(z-q^6)(z-q^8)^2(z-q^{10}),\allowdisplaybreaks\\
& d_{2,3}(z)=d_{3,2}(z)=d_{3,4}(z)=d_{4,3}(z)=(z+q^3)(z+q^5)^2(z+q^7)^2(z+q^9)^2(z+q^{11}),\allowdisplaybreaks\\
& d_{2,4}(z)=d_{4,2}(z)=(z-q^4)(z-q^6)^2(z-q^8)(z-q^{10})(z-q^{12}),\allowdisplaybreaks\\
& d_{2,6}(z)=d_{6,2}(z)=d_{4,6}(z)=d_{6,4}(z)=(z-q^4)(z-q^6)(z-q^8)(z-q^{10}),\allowdisplaybreaks\\
& d_{3,3}(z)=(z-q^2)(z-q^4)^2(z-q^6)^2(z-q^8)^3(z-q^{10})^2(z-q^{12}),\allowdisplaybreaks\\
& d_{3,6}(z)=d_{6,3}(z)=(z+q^3)(z+q^5)(z+q^7)^2(z+q^9)(z+q^{11}),\allowdisplaybreaks\\
& d_{6,6}(z)=(z-q^2)(z-q^6)(z-q^8)(z-q^{12}).
\end{align*}

Now we give a conjectural denominator formulas $d_{k,l}(z)$ for $U_q'(E_7^{(1)})$ here $(p=q^{18})$:
\begin{align*}
d_{1,1}(z)&=(z-q^{2})(z-q^{8})(z-q^{12})(z-q^{18}),   \allowdisplaybreaks \\
d_{1,3}(z)&=(z+q^{3})(z+q^{7})(z+q^{9})(z+q^{11})(z+q^{13})(z+q^{17}),  \allowdisplaybreaks \\
d_{1,4}(z)&=(z-q^{4})(z-q^{6})(z-q^{8})(z-q^{10})^2(z-q^{12})(z-q^{14})(z-q^{16}),   \allowdisplaybreaks \\
d_{1,2}(z)&=(z+q^{5})(z+q^{9})(z+q^{11})(z+q^{15}),   \allowdisplaybreaks \\
d_{1,5}(z)&=(z+q^{5})(z+q^{7})(z+q^{9})(z+q^{11})(z+q^{13})(z+q^{15}),  \allowdisplaybreaks \\
d_{1,6}(z)&=(z-q^{6})(z-q^{8})(z-q^{12})(z-q^{14}),   \allowdisplaybreaks \\
d_{1,7}(z)&=(z+q^{7})(z+q^{13}),    \allowdisplaybreaks \\
d_{3,3}(z)&=(z-q^{2})(z-q^{4})(z-q^{6})(z-q^{8})^2(z-q^{10})^2(z-q^{12})^2(z-q^{14})(z-q^{16})(z-q^{18}),  \allowdisplaybreaks \\
d_{3,4}(z)&=(z+q^{3})(z+q^{5})^2(z+q^{7})^2(z+q^{9})^2(z+q^{11})^3(z+q^{13})^2(z+q^{15})^2(z+q^{17}),   \allowdisplaybreaks \\
d_{3,2}(z)&=(z-q^{4})(z-q^{6})(z-q^{8})(z-q^{10})^2(z-q^{12})(z-q^{14})(z-q^{16}),    \allowdisplaybreaks \\
d_{3,5}(z)&=(z-q^{4})(z-q^{6})^2(z-q^{8})^2(z-q^{10})^2(z-q^{12})^2(z-q^{14})^2(z-q^{16}),  \allowdisplaybreaks \\
d_{3,6}(z)&=(z+q^{5})(z+q^{7})^2(z+q^{9})(z+q^{11})(z+q^{13})^2 (z+q^{15}),  \allowdisplaybreaks \\
d_{3,7}(z)&=(z-q^{6})(z-q^{8})(z-q^{12})(z-q^{14}),    \allowdisplaybreaks \\
d_{4,4}(z)&=(z-q^{2})(z-q^{4})^2(z-q^{6})^2(z-q^{8})^3(z-q^{10})^3(z-q^{12})^4(z-q^{14})^3(z-q^{16})^2(z-q^{18}), \allowdisplaybreaks \\
d_{4,2}(z)&=(z+q^{3})(z+q^{5})(z+q^{7})^2(z+q^{9})^2(z+q^{11})^2(z+q^{13})^2(z+q^{15})(z+q^{17}),   \allowdisplaybreaks \\
d_{4,5}(z)&=(z+q^{3})(z+q^{5})^2(z+q^{7})^2(z+q^{9})^3(z+q^{11})^2(z+q^{13})^3(z+q^{15})^2(z+q^{17}),    \allowdisplaybreaks \\
d_{4,6}(z)&=(z-q^{4})(z-q^{6})^2(z-q^{8})^2(z-q^{10})^2(z-q^{12})^2(z-q^{14})^2(z-q^{16}),   \allowdisplaybreaks \\
d_{4,7}(z)&=(z+q^{5})(z+q^{7})(z+q^{9})(z+q^{11})(z+q^{13})(z+q^{15}), \allowdisplaybreaks \\
d_{2,2}(z)&=(z-q^{2})(z-q^{6})(z-q^{8})(z-q^{10})(z-q^{12})(z-q^{14})(z-q^{18}),  \allowdisplaybreaks \\
d_{2,5}(z)&=(z-q^{4})(z-q^{6})(z-q^{8})^2(z-q^{10})(z-q^{12})^2(z-q^{14})(z-q^{16}),  \allowdisplaybreaks \\
d_{2,6}(z)&=(z+q^{5})(z+q^{7})(z+q^{9})(z+q^{11})(z+q^{13})(z+q^{15}),    \allowdisplaybreaks \\
d_{2,7}(z)&=(z-q^{6})(z-q^{10})(z-q^{14}),   \allowdisplaybreaks \\
d_{5,5}(z)&=(z-q^{2})(z-q^{4})^2(z-q^{6})^2(z-q^{8})^2(z-q^{10})^3(z-q^{12})^2(z-q^{14})^3 (z-q^{16})(z-q^{18}),   \allowdisplaybreaks \\
d_{5,6}(z)&=(z+q^{3})(z+q^{5})(z+q^{7})(z+q^{9})^2(z+q^{11})^2(z+q^{13})(z+q^{15})(z+q^{17}),   \allowdisplaybreaks \\
d_{5,7}(z)&=(z-q^{4})(z-q^{8})(z-q^{10})(z-q^{12})(z-q^{16}),\allowdisplaybreaks  \\
d_{6,6}(z)&=(z-q^{2})(z-q^{4})(z-q^{8})(z-q^{10})^2(z-q^{12})(z-q^{16})(z-q^{18}),  \allowdisplaybreaks \\
d_{6,7}(z)&=(z+q^{3})(z+q^{9})(z+q^{11})(z+q^{17}),  \allowdisplaybreaks \\
d_{7,7}(z)&=(z-q^{2})(z-q^{10})(z-q^{18}).    \allowdisplaybreaks \\
\end{align*}

%\newpage

\appendix

%\begin{landscape}
\section{Dynkin quiver $Q$ of type $E_6$, its AR-quiver $\GQ$ and Dorey's rule for $E_6^{(1)}$} \label{Sec:Dynkin E6}

(1) Let us consider the Dynkin quiver $Q$ given as follows:
$  \raisebox{1.3em}
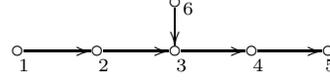
{\xymatrix@R=3ex{ && *{\circ}<3pt>\ar@{->}[d]^<{6} \\
*{ \circ }<3pt> \ar@{->}[r]_<{1}  &*{\circ}<3pt>
\ar@{->}[r]_<{2} &*{ \circ }<3pt> \ar@{->}[r]_<{3} &*{\circ}<3pt>
\ar@{->}[r]_<{4} &*{\circ}<3pt>
\ar@{-}[l]^<{\ \ 5}}}
$

Its AR-quiver $\GQ$ can be drawn as follows
$\left({\scriptstyle\prt{a_1a_2a_3}{a_4a_5a_6}} \seteq (a_1a_2a_3a_4a_5a_6).\right)$:
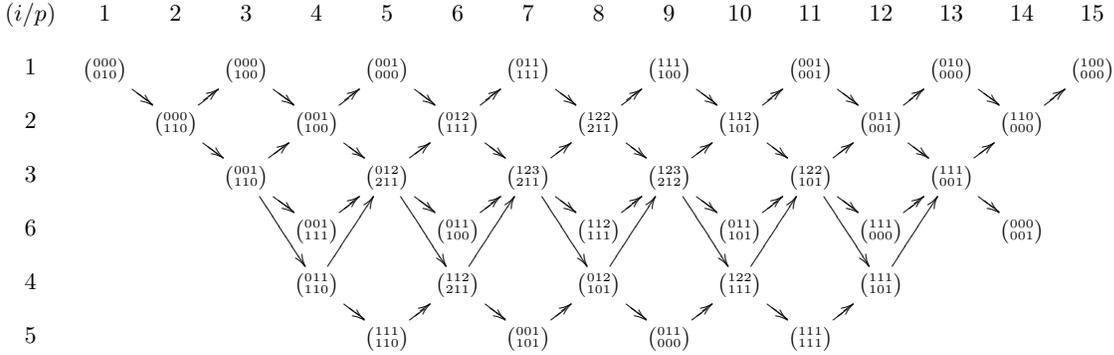
$$ \scalebox{0.9}{\xymatrix@R=1.5ex@C=1.3ex{
(i/p) & 1 & 2& 3 &4 & 5 & 6 & 7  & 8 & 9 & 10 & 11 & 12 & 13 & 14 & 15 \\
1&{\scriptstyle\prt{000}{010}} \ar[dr] && {\scriptstyle\prt{000}{100}} \ar[dr] && {\scriptstyle\prt{001}{000}} \ar[dr]&& {\scriptstyle\prt{011}{111}}\ar[dr]
&&{\scriptstyle\prt{111}{100}}\ar[dr] && {\scriptstyle\prt{001}{001}}\ar[dr] && {\scriptstyle\prt{010}{000}} \ar[dr]&& {\scriptstyle\prt{100}{000}} \\
2&& {\scriptstyle\prt{000}{110}}\ar[dr]\ar[ur] && {\scriptstyle\prt{001}{100}}\ar[dr]\ar[ur] && {\scriptstyle\prt{012}{111}}\ar[dr]\ar[ur]
&& {\scriptstyle\prt{122}{211}}\ar[dr]\ar[ur] && {\scriptstyle\prt{112}{101}}\ar[dr]\ar[ur] && {\scriptstyle\prt{011}{001}}\ar[dr]\ar[ur]
&& {\scriptstyle\prt{110}{000}}\ar[ur] \\
3&&& {\scriptstyle\prt{001}{110}}\ar[ddr]\ar[dr]\ar[ur]&& {\scriptstyle\prt{012}{211}}\ar[ddr]\ar[dr]\ar[ur]&& {\scriptstyle\prt{123}{211}}\ar[ddr]\ar[dr]\ar[ur]
&& {\scriptstyle\prt{123}{212}}\ar[ddr]\ar[dr]\ar[ur] && {\scriptstyle\prt{122}{101}}\ar[ddr]\ar[dr]\ar[ur]&& {\scriptstyle\prt{111}{001}}\ar[dr]\ar[ur] \\
6&&&& {\scriptstyle\prt{001}{111}}\ar[ur]&& {\scriptstyle\prt{011}{100}}\ar[ur]&& {\scriptstyle\prt{112}{111}}\ar[ur]&& {\scriptstyle\prt{011}{101}}\ar[ur]
&& {\scriptstyle\prt{111}{000}}\ar[ur]&& {\scriptstyle\prt{000}{001}} \\
4&&&& {\scriptstyle\prt{011}{110}}\ar[uur]\ar[dr]&& {\scriptstyle\prt{112}{211}}\ar[uur]\ar[dr]&& {\scriptstyle\prt{012}{101}}\ar[uur]\ar[dr]
&& {\scriptstyle\prt{122}{111}}\ar[uur]\ar[dr]&& {\scriptstyle\prt{111}{101}}\ar[uur] \\
5&&&&& {\scriptstyle\prt{111}{110}}\ar[ur]&& {\scriptstyle\prt{001}{101}}\ar[ur]&& {\scriptstyle\prt{011}{000}}\ar[ur]&& {\scriptstyle\prt{111}{111}}\ar[ur] \\
}}
$$

By reading pairs and their socles, we can obtain Dorey's rule for $E_6^{(1)}$. Here we list several Dorey's type morphisms:
\begin{align*}
V(\varpi_6) & \hookrightarrow  V(\varpi_6)_{ q^{4}} \otimes V(\varpi_6)_{q^{-4}}, &
V(\varpi_3) &  \hookrightarrow  V(\varpi_1)_{q^{2}} \otimes V(\varpi_2)_{-q^{-1}}, &
V(\varpi_4) &  \hookrightarrow  V(\varpi_1)_{-q^{3}} \otimes V(\varpi_6)_{q^{-2}},
\\
V(\varpi_6) & \hookrightarrow  V(\varpi_1)_{-q^{3}} \otimes V(\varpi_5)_{-q^{-3}}, &
V(\varpi_5) &  \hookrightarrow  V(\varpi_2)_{-q^{9}} \otimes V(\varpi_3)_{q^{-2}}, &
V(\varpi_5) &  \hookrightarrow  V(\varpi_1)_{q^{4}} \otimes V(\varpi_1)_{q^{-4}},
\\
V(\varpi_4) & \hookrightarrow  V(\varpi_1)_{-q^{9}} \otimes V(\varpi_3)_{-q^{-1}}, &
V(\varpi_2) &  \hookrightarrow  V(\varpi_1)_{-q^{5}} \otimes V(\varpi_4)_{q^{-2}},&
V(\varpi_1) &  \hookrightarrow  V(\varpi_1)_{-q^{5}} \otimes V(\varpi_6)_{-q^{-5}},
\\
V(\varpi_5) & \hookrightarrow  V(\varpi_2)_{-q^{5}} \otimes V(\varpi_6)_{-q^{-5}}, &
V(\varpi_3) &  \hookrightarrow  V(\varpi_3)_{q^{4}} \otimes V(\varpi_3)_{q^{-4}}, &
V(\varpi_6) &  \hookrightarrow  V(\varpi_3)_{-q^{5}} \otimes V(\varpi_3)_{-q^{-5}},
\end{align*}
\vskip -3.3em
\begin{align*}
V(\varpi_6)_{q^{2}} \otimes V(\varpi_5)_{-q^{1}} & \hookrightarrow  V(\varpi_1)_{-q^{5}} \otimes V(\varpi_4), &
V(\varpi_6)_{q^{2}} \otimes V(\varpi_4)_{q^{2}} & \hookrightarrow  V(\varpi_2)_{-q^{4}} \otimes V(\varpi_2),
\\
V(\varpi_1)_{q^{6}} \otimes V(\varpi_6)_{q^{4}} & \hookrightarrow  V(\varpi_2)_{-q^{7}} \otimes V(\varpi_1), &
V(\varpi_6)_{-q^{3}} \otimes V(\varpi_4)_{-q^{3}} & \hookrightarrow  V(\varpi_3)_{q^{4}} \otimes V(\varpi_1),
\end{align*}
\vskip -3.3em
\begin{align*}
V(\varpi_1)_{-q^{1}} \otimes V(\varpi_1)_{-q^{5}} \otimes V(\varpi_5)_{-q^{3}} & \hookrightarrow  V(\varpi_2)_{q^{6}} \otimes V(\varpi_2).
\end{align*}

(2) The convex order $\prec_{[\redez]}$ in {\rm (b)} of Example~\ref{ex: D4} can be visualized by the result of~\cite{OS15} as follows:
$$\scalebox{0.83}
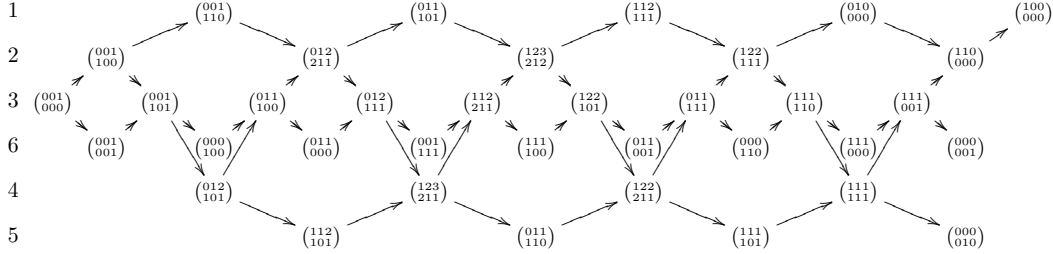
{\xymatrix@C=0.1ex@R=1ex{
1 &&&& {\scriptstyle\prt{001}{110}} \ar@{->}[drr] &&&& {\scriptstyle\prt{011}{101}}\ar@{->}[drr] &&&& {\scriptstyle\prt{112}{111}}\ar@{->}[drr]
&&&& {\scriptstyle\prt{010}{000}} \ar@{->}[drr] &&&& {\scriptstyle\prt{100}{000}} \\
2 && {\scriptstyle\prt{001}{100}} \ar@{->}[urr]\ar@{->}[dr] &&&& {\scriptstyle\prt{012}{211}}\ar@{->}[urr]\ar@{->}[dr] &&&&{\scriptstyle\prt{123}{212}}
\ar@{->}[urr]\ar@{->}[dr]
&&&&{\scriptstyle\prt{122}{111}} \ar@{->}[urr]\ar@{->}[dr] &&&& {\scriptstyle\prt{110}{000}}\ar@{->}[urr]\\
3 & {\scriptstyle\prt{001}{000}}\ar@{->}[ur]\ar@{->}[dr] && {\scriptstyle\prt{001}{101}} \ar@{->}[dr]\ar@{->}[ddr] && {\scriptstyle\prt{011}{100}} \ar@{->}[ur]\ar@{->}[dr]
&& {\scriptstyle\prt{012}{111}} \ar@{->}[dr]\ar@{->}[ddr] && {\scriptstyle\prt{112}{211}}\ar@{->}[ur]\ar@{->}[dr] &&{\scriptstyle\prt{122}{101}}\ar@{->}[ddr]\ar@{->}[dr]
&&{\scriptstyle\prt{011}{111}}\ar@{->}[ur]\ar@{->}[dr] &&
{\scriptstyle\prt{111}{110}}\ar@{->}[dr]\ar@{->}[ddr] && {\scriptstyle\prt{111}{001}}\ar@{->}[ur]\ar@{->}[dr]\\
6 && {\scriptstyle\prt{001}{001}}\ar@{->}[ur] && {\scriptstyle\prt{000}{100}} \ar@{->}[ur] && {\scriptstyle\prt{011}{000}} \ar@{->}[ur] &&
{\scriptstyle\prt{001}{111}}\ar@{->}[ur] &&{\scriptstyle\prt{111}{100}}\ar@{->}[ur] && {\scriptstyle\prt{011}{001}} \ar@{->}[ur]&& {\scriptstyle\prt{000}{110}}
\ar@{->}[ur] &&{\scriptstyle\prt{111}{000}}\ar@{->}[ur]  &&{\scriptstyle\prt{000}{001}} \\
4 &&&& {\scriptstyle\prt{012}{101}}\ar@{->}[drr]\ar@{->}[uur] &&&& {\scriptstyle\prt{123}{211}}\ar@{->}[drr]\ar@{->}[uur]
&&&& {\scriptstyle\prt{122}{211}}\ar@{->}[drr]\ar@{->}[uur] &&&& {\scriptstyle\prt{111}{111}} \ar@{->}[drr]\ar@{->}[uur]\\
5 &&&&&& {\scriptstyle\prt{112}{101}}\ar@{->}[urr]  &&&& {\scriptstyle\prt{011}{110}} \ar@{->}[urr]
&&&& {\scriptstyle\prt{111}{101}} \ar@{->}[urr]  &&&& {\scriptstyle\prt{000}{010}}
}}
$$
%\end{landscape}
%\newpage
%\begin{landscape}
\section{Dynkin quiver $Q$ of type $E_7$, its AR-quiver $\GQ$ and Dorey's rule for $E_7^{(1)}$} \label{Sec:Dynkin E7}

Let us consider the Dynkin quiver $Q$ given as follows:
$
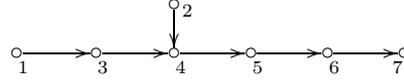
\raisebox{1.3em}{\xymatrix@R=3ex{ && *{\circ}<3pt>\ar@{->}[d]^<{2} \\
*{ \circ }<3pt> \ar@{->}[r]_<{1}  &*{\circ}<3pt>
\ar@{->}[r]_<{3} &*{ \circ }<3pt> \ar@{->}[r]_<{4} &*{\circ}<3pt>
\ar@{->}[r]_<{5} &*{\circ}<3pt>
\ar@{->}[r]_<{6} &*{\circ}<3pt>
\ar@{-}[l]^<{7} } }
$

Its AR-quiver $\GQ$ can be drawn as follows:
$$ \scalebox{0.70}
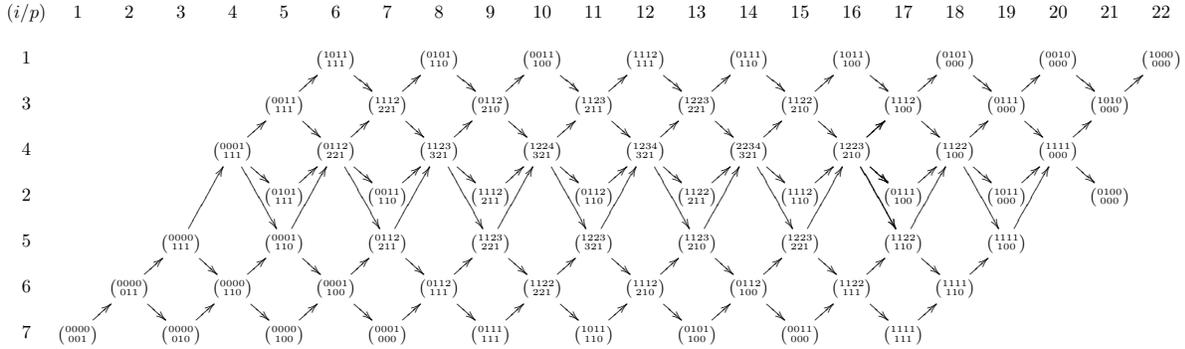
{\xymatrix@R=2ex@C=0.1ex{
(i/p)& 1 & 2& 3 &4 & 5 & 6 & 7  & 8 & 9 & 10 & 11 & 12 & 13 & 14 & 15 &16 & 17  &18 & 19 & 20  &21 & 22\\
1&&&&&&{\scriptstyle\prt{1011}{111}} \ar[dr]&&{\scriptstyle\prt{0101}{110}} \ar[dr]&&{\scriptstyle\prt{0011}{100}} \ar[dr]&&{\scriptstyle\prt{1112}{111}} \ar[dr]
&&{\scriptstyle\prt{0111}{110}} \ar[dr] &&{\scriptstyle\prt{1011}{100}} \ar[dr]&&{\scriptstyle\prt{0101}{000}} \ar[dr]
&&{\scriptstyle\prt{0010}{000}} \ar[dr]&&{\scriptstyle\prt{1000}{000}} \\
3&&&&&{\scriptstyle\prt{0011}{111}} \ar[ur]\ar[dr]&&{\scriptstyle\prt{1112}{221}}\ar[ur]\ar[dr]&&{\scriptstyle\prt{0112}{210}}\ar[ur]\ar[dr]
&&{\scriptstyle\prt{1123}{211}}\ar[ur]\ar[dr]&&{\scriptstyle\prt{1223}{221}}\ar[ur]\ar[dr]
&&{\scriptstyle\prt{1122}{210}}\ar[ur]\ar[dr]&&{\scriptstyle\prt{1112}{100}}\ar[ur]\ar[dr]&&{\scriptstyle\prt{0111}{000}}\ar[ur]\ar[dr]&&{\scriptstyle\prt{1010}{000}}\ar[ur]\\
4&&&&{\scriptstyle\prt{0001}{111}}\ar[ur]\ar[ddr]\ar[dr]&&{\scriptstyle\prt{0112}{221}}\ar[ur]\ar[ddr]\ar[dr]
&&{\scriptstyle\prt{1123}{321}}\ar[ur]\ar[ddr]\ar[dr]&&{\scriptstyle\prt{1224}{321}}\ar[ur]\ar[ddr]\ar[dr]&&{\scriptstyle\prt{1234}{321}}\ar[ur]\ar[ddr]\ar[dr]
&&{\scriptstyle\prt{2234}{321}}\ar[ur]\ar[ddr]\ar[dr]&&{\scriptstyle\prt{1223}{210}}\ar[ur]\ar[ddr]\ar[dr]\ar[ur]\ar[ddr]\ar[dr]
&&{\scriptstyle\prt{1122}{100}}\ar[ur]\ar[ddr]\ar[dr]&&{\scriptstyle\prt{1111}{000}}\ar[ur]\ar[dr]
\\
2&&&&&{\scriptstyle\prt{0101}{111}}\ar[ur]&&{\scriptstyle\prt{0011}{110}}\ar[ur]&&{\scriptstyle\prt{1112}{211}}\ar[ur]&&{\scriptstyle\prt{0112}{110}}\ar[ur]
&&{\scriptstyle\prt{1122}{211}}\ar[ur]&&{\scriptstyle\prt{1112}{110}}\ar[ur]&&{\scriptstyle\prt{0111}{100}}\ar[ur]&&{\scriptstyle\prt{1011}{000}}\ar[ur]
&&{\scriptstyle\prt{0100}{000}}
\\
5&&&{\scriptstyle\prt{0000}{111}}\ar[uur]\ar[dr]&&{\scriptstyle\prt{0001}{110}}\ar[uur]\ar[dr]&&{\scriptstyle\prt{0112}{211}}\ar[uur]\ar[dr]
&&{\scriptstyle\prt{1123}{221}}\ar[uur]\ar[dr]&&{\scriptstyle\prt{1223}{321}}\ar[uur]\ar[dr]&&{\scriptstyle\prt{1123}{210}}\ar[uur]\ar[dr]
&&{\scriptstyle\prt{1223}{221}}\ar[uur]\ar[dr]&&{\scriptstyle\prt{1122}{110}}\ar[uur]\ar[dr]&&{\scriptstyle\prt{1111}{100}}\ar[uur]
\\
6&&{\scriptstyle\prt{0000}{011}}\ar[ur]\ar[dr]&&{\scriptstyle\prt{0000}{110}}\ar[ur]\ar[dr]&&{\scriptstyle\prt{0001}{100}}\ar[ur]\ar[dr]
&&{\scriptstyle\prt{0112}{111}}\ar[ur]\ar[dr]&&{\scriptstyle\prt{1122}{221}}\ar[ur]\ar[dr]
&&{\scriptstyle\prt{1112}{210}}\ar[ur]\ar[dr]&&{\scriptstyle\prt{0112}{100}}\ar[ur]\ar[dr]&&{\scriptstyle\prt{1122}{111}}\ar[ur]\ar[dr]
&&{\scriptstyle\prt{1111}{110}}\ar[ur]
\\
7& {\scriptstyle\prt{0000}{001}}\ar[ur] &&{\scriptstyle\prt{0000}{010}}\ar[ur]&&{\scriptstyle\prt{0000}{100}}\ar[ur]&&{\scriptstyle\prt{0001}{000}}\ar[ur]
&&{\scriptstyle\prt{0111}{111}}\ar[ur] &&{\scriptstyle\prt{1011}{110}}\ar[ur]&&{\scriptstyle\prt{0101}{100}}\ar[ur]
&&{\scriptstyle\prt{0011}{000}}\ar[ur]&&{\scriptstyle\prt{1111}{111}}\ar[ur]
}}
$$
By reading pairs and their socles, we can obtain Dorey's rule for $E_7^{(1)}$. Here we list several Dorey's type morphisms:
\begin{align*}
V(\varpi_4) & \hookrightarrow  V(\varpi_3)_{-q^{1}} \otimes V(\varpi_1)_{q^{-2}}, &
V(\varpi_1) &  \hookrightarrow  V(\varpi_3)_{-q^{3}} \otimes V(\varpi_1)_{q^{-10}}, &
V(\varpi_3) &  \hookrightarrow  V(\varpi_4)_{-q^{1}} \otimes V(\varpi_1)_{-q^{-15}},
\\
V(\varpi_2) & \hookrightarrow  V(\varpi_5)_{q^{2}} \otimes V(\varpi_1)_{-q^{-13}}, &
V(\varpi_1) &  \hookrightarrow  V(\varpi_6)_{q^{4}} \otimes V(\varpi_1)_{q^{-10}}, &
V(\varpi_7) &  \hookrightarrow  V(\varpi_7)_{q^{8}} \otimes V(\varpi_1)_{-q^{-5}},
\\
V(\varpi_6) & \hookrightarrow  V(\varpi_1)_{q^{4}} \otimes V(\varpi_3)_{q^{-4}}, &
V(\varpi_5) &  \hookrightarrow  V(\varpi_2)_{q^{2}} \otimes V(\varpi_1)_{-q^{-3}},&
V(\varpi_3) &  \hookrightarrow  V(\varpi_3)_{q^{6}} \otimes V(\varpi_3)_{q^{-6}},
\\
V(\varpi_6) & \hookrightarrow  V(\varpi_4)_{q^{4}} \otimes V(\varpi_3)_{-q^{-11}}, &
V(\varpi_1) &  \hookrightarrow  V(\varpi_4)_{q^{8}} \otimes V(\varpi_4)_{q^{-8}}, &
V(\varpi_4) &  \hookrightarrow  V(\varpi_4)_{q^{6}} \otimes V(\varpi_4)_{q^{-6}},
\end{align*}
\vskip -1.8em
\begin{align*}
V(\varpi_2)_{-q^{7}} \otimes V(\varpi_7)_{-q^{5}} & \hookrightarrow  V(\varpi_3)_{-q^{9}} \otimes V(\varpi_1), &
V(\varpi_1)_{q^{6}} \otimes V(\varpi_7)_{-q^{5}} & \hookrightarrow  V(\varpi_2)_{-q^{9}} \otimes V(\varpi_1),
\\
V(\varpi_1)_{-q^{3}} \otimes V(\varpi_7)_{q^{4}} & \hookrightarrow  V(\varpi_2)_{q^{6}} \otimes V(\varpi_1), &
V(\varpi_3)_{q^{2}} \otimes V(\varpi_2)_{q^{2}} & \hookrightarrow  V(\varpi_6)_{-q^{5}} \otimes V(\varpi_5),
\end{align*}
\vskip -1.8em
\begin{align*}
V(\varpi_7)_{-q^{1}} \otimes V(\varpi_7)_{-q^{7}} \otimes V(\varpi_1)_{q^{4}} & \hookrightarrow  V(\varpi_6)_{q^{8}} \otimes V(\varpi_6).
\end{align*}

\section{Dynkin quiver $Q$ of type $E_8$ and its AR-quiver $\GQ$} \label{Sec:Dynkin E8}

Let us consider the Dynkin quiver $Q$ given as follows:
$
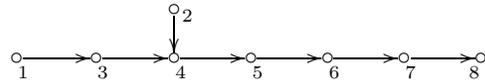
\raisebox{1.3em}{\xymatrix@R=3ex{ && *{\circ}<3pt>\ar@{->}[d]^<{2} \\
*{ \circ }<3pt> \ar@{->}[r]_<{1}  &*{\circ}<3pt>
\ar@{->}[r]_<{3} &*{ \circ }<3pt> \ar@{->}[r]_<{4} &*{\circ}<3pt>
\ar@{->}[r]_<{5} &*{\circ}<3pt>
\ar@{->}[r]_<{6} &*{\circ}<3pt>
\ar@{->}[r]_<{7} &*{\circ}<3pt>
\ar@{-}[l]^<{8} } }.
$

Note that $i^*=i$ for all $i\in I$.
Its AR-quiver $\GQ$ can be drawn as follows:
$$ \scalebox{0.42}{\xymatrix@R=2ex@C=0.1ex{
(i/p)& 1 & 2& 3 &4 & 5 & 6 & 7  & 8 & 9 & 10 & 11 & 12 & 13 & 14 & 15 &16 & 17  &18 & 19 & 20  &21 & 22& 23 &24 & 25 & 26 & 27  & 28 & 29 & 30&31 & 32& 33 &34 & 35   \\
1&&&&&&&\pprt{10}{11}{11}{11}\ar[dr]&&\pprt{01}{01}{11}{10}\ar[dr]&&\pprt{00}{11}{11}{00}\ar[dr]&&\pprt{11}{12}{21}{11}\ar[dr]&&\pprt{01}{12}{11}{10}\ar[dr]
&&\pprt{11}{22}{22}{11}\ar[dr]&&\pprt{11}{12}{21}{10}\ar[dr]&&\pprt{01}{12}{11}{00}\ar[dr]&&\pprt{11}{22}{21}{11}\ar[dr]&&\pprt{11}{12}{11}{10}\ar[dr]
&&\pprt{01}{11}{11}{00}\ar[dr]&&\pprt{10}{11}{10}{00}\ar[dr]&&\pprt{01}{01}{00}{00}\ar[dr]&&\pprt{00}{10}{00}{00}\ar[dr]&&\pprt{10}{00}{00}{00}
\\
3&&&&&&\pprt{00}{11}{11}{11}\ar[dr]\ar[ur]&&\pprt{11}{12}{22}{21}\ar[dr]\ar[ur]&&\pprt{01}{12}{22}{10}\ar[dr]\ar[ur]&&\pprt{11}{23}{32}{11}\ar[dr]\ar[ur]
&&\pprt{12}{24}{32}{21}\ar[dr]\ar[ur]&&\pprt{12}{34}{33}{21}\ar[dr]\ar[ur]&&\pprt{22}{34}{43}{21}\ar[dr]\ar[ur]&&\pprt{12}{24}{32}{10}\ar[dr]\ar[ur]
&&\pprt{12}{34}{32}{11}\ar[dr]\ar[ur]&&\pprt{22}{34}{32}{21}\ar[dr]\ar[ur]&&\pprt{12}{23}{22}{10}\ar[dr]\ar[ur]&&\pprt{11}{22}{21}{00}\ar[dr]\ar[ur]
&&\pprt{11}{12}{10}{00}\ar[dr]\ar[ur]&&\pprt{01}{11}{00}{00}\ar[dr]\ar[ur]&&\pprt{10}{10}{00}{00}\ar[ur]
\\
4&&&&&\pprt{00}{01}{11}{11}\ar[dr]\ar[ur]\ar[ddr]&&\pprt{01}{12}{22}{21}\ar[dr]\ar[ur]\ar[ddr]&&\pprt{11}{23}{33}{21}\ar[dr]\ar[ur]\ar[ddr]&&
\pprt{12}{24}{43}{21}\ar[dr]\ar[ur]\ar[ddr]&&\pprt{12}{35}{43}{21}\ar[dr]\ar[ur]\ar[ddr]&&\pprt{23}{46}{54}{32}\ar[dr]\ar[ur]\ar[ddr]
&&\pprt{23}{46}{54}{31}\ar[dr]\ar[ur]\ar[ddr]&&\pprt{23}{46}{54}{21}\ar[dr]\ar[ur]\ar[ddr]&&\pprt{23}{46}{53}{21}\ar[dr]\ar[ur]\ar[ddr]
&&\pprt{23}{46}{43}{21}\ar[dr]\ar[ur]\ar[ddr]&&\pprt{23}{45}{43}{21}\ar[dr]\ar[ur]\ar[ddr]&&\pprt{22}{34}{32}{10}\ar[dr]\ar[ur]\ar[ddr]
&&\pprt{12}{23}{21}{00}\ar[dr]\ar[ur]\ar[ddr]&&\pprt{11}{22}{10}{00}\ar[dr]\ar[ur]\ar[ddr]&&\pprt{11}{11}{00}{00}\ar[dr]\ar[ur]
\\
2&&&&&&\pprt{01}{01}{11}{11}\ar[ur]&&\pprt{00}{11}{11}{10}\ar[ur]&&\pprt{11}{12}{22}{11}\ar[ur]&&\pprt{01}{12}{21}{10}\ar[ur]&&\pprt{11}{23}{22}{11}\ar[ur]
&&\pprt{12}{23}{32}{21}\ar[ur]&&\pprt{11}{23}{32}{10}\ar[ur]&&\pprt{12}{23}{32}{11}\ar[ur]&&\pprt{11}{23}{21}{10}\ar[ur]&&\pprt{12}{23}{22}{11}\ar[ur]
&&\pprt{11}{22}{21}{10}\ar[ur]&&\pprt{11}{12}{11}{00}\ar[ur]&&\pprt{01}{11}{10}{00}\ar[ur]&&\pprt{10}{11}{00}{00}\ar[ur]&&\pprt{01}{00}{00}{00}
\\
5&&&&\pprt{00}{00}{11}{11}\ar[uur]\ar[dr]&&\pprt{00}{01}{11}{10}\ar[uur]\ar[dr]&&\pprt{01}{12}{22}{11}\ar[uur]\ar[dr]&&\pprt{11}{23}{32}{21}\ar[uur]\ar[dr]
&&\pprt{12}{24}{33}{21}\ar[uur]\ar[dr]&&\pprt{12}{34}{43}{21}\ar[uur]\ar[dr]&&\pprt{22}{35}{43}{21}\ar[uur]\ar[dr]&&\pprt{13}{35}{43}{21}\ar[uur]\ar[dr]
&&\pprt{22}{45}{43}{21}\ar[uur]\ar[dr]&&\pprt{23}{35}{43}{21}\ar[uur]\ar[dr]&&\pprt{12}{34}{32}{10}\ar[uur]\ar[dr]&&\pprt{22}{34}{32}{11}\ar[uur]\ar[dr]
&&\pprt{12}{23}{21}{10}\ar[uur]\ar[dr]&&\pprt{11}{22}{11}{00}\ar[uur]\ar[dr]&&\pprt{11}{11}{10}{00}\ar[uur]
\\
6&&&\pprt{00}{00}{01}{11}\ar[ur]\ar[dr]&&\pprt{00}{00}{11}{10}\ar[ur]\ar[dr]&&\pprt{00}{01}{11}{00}\ar[ur]\ar[dr]&&\pprt{01}{12}{21}{11}\ar[ur]\ar[dr]&&\pprt{11}{23}{22}{21}\ar[ur]\ar[dr]
&&\pprt{12}{23}{33}{21}\ar[ur]\ar[dr]&&\pprt{11}{23}{32}{10}\ar[ur]\ar[dr]&&\pprt{12}{24}{32}{11}\ar[ur]\ar[dr]&&\pprt{12}{34}{32}{21}\ar[ur]\ar[dr]
&&\pprt{22}{34}{33}{21}\ar[ur]\ar[dr]&&\pprt{12}{23}{32}{10}\ar[ur]\ar[dr]&&\pprt{11}{23}{21}{00}\ar[ur]\ar[dr]&&\pprt{12}{23}{21}{11}\ar[ur]\ar[dr]
&&\pprt{11}{22}{11}{10}\ar[ur]\ar[dr]&&\pprt{11}{11}{11}{00}\ar[ur]
\\
7&&\pprt{00}{00}{00}{11}\ar[ur]\ar[dr]&&\pprt{00}{00}{01}{10}\ar[ur]\ar[dr]&&\pprt{00}{00}{11}{00}\ar[ur]\ar[dr]&&\pprt{00}{01}{10}{00}\ar[ur]\ar[dr]
&&\pprt{01}{12}{11}{11}\ar[ur]\ar[dr]&&\pprt{11}{22}{22}{21}\ar[ur]\ar[dr]&&\pprt{11}{12}{22}{10}\ar[ur]\ar[dr]
&&\pprt{01}{12}{21}{00}\ar[ur]\ar[dr]&&\pprt{11}{23}{21}{11}\ar[ur]\ar[dr]&&\pprt{12}{23}{22}{21}\ar[ur]\ar[dr]&&\pprt{11}{22}{22}{10}\ar[ur]\ar[dr]
&&\pprt{11}{12}{21}{00}\ar[ur]\ar[dr]&&\pprt{01}{12}{10}{00}\ar[ur]\ar[dr]&&\pprt{11}{22}{11}{11}\ar[ur]\ar[dr] &&\pprt{11}{11}{11}{10}\ar[ur]
\\
8&\pprt{00}{00}{00}{01}\ar[ur]&&\pprt{00}{00}{00}{10}\ar[ur]&&\pprt{00}{00}{01}{00}\ar[ur]&&\pprt{00}{00}{10}{00}\ar[ur]
&&\pprt{00}{01}{00}{00}\ar[ur]&&\pprt{01}{11}{11}{11}\ar[ur]&&\pprt{10}{11}{11}{10}\ar[ur]&&\pprt{01}{01}{11}{00}\ar[ur]
&&\pprt{00}{11}{10}{00}\ar[ur]&&\pprt{11}{12}{11}{11}\ar[ur]&&\pprt{01}{11}{11}{10}\ar[ur]
&&\pprt{10}{11}{11}{00}\ar[ur]&&\pprt{01}{01}{10}{00}\ar[ur]&&\pprt{00}{11}{00}{00}\ar[ur]&&\pprt{11}{11}{11}{11}\ar[ur]
}}
$$
%Here $\pprt{a_1a_2}{a_3a_4}{a_5a_6}{a_7a_8}$ denotes $(a_1a_2a_3a_4a_5a_6a_7a_8)$
%\Sejin{Maybe Add more Appendix about E-type Dorey's rule!!}

\end{document}